\documentclass[twoside,11pt]{amsart}
\usepackage[all]{xy}
\CompileMatrices
\usepackage{amsfonts}
\usepackage{amssymb}
\usepackage{amsthm}
\usepackage{amsmath}
\usepackage{ifthen}
\usepackage{enumitem} 
\usepackage{verbatim} 
\usepackage{url}
\usepackage[usenames]{color}
\usepackage[draft]{todonotes}   

\allowdisplaybreaks 

 \newtheorem{thm}{Theorem}[section]
 \newtheorem{prop}[thm]{Proposition}
 \newtheorem{cor}[thm]{Corollary}
 \newtheorem{lem}[thm]{Lemma}

\theoremstyle{definition}
\newtheorem{defn}[thm]{Definition}

\theoremstyle{remark}
\newtheorem{rem}[thm]{Remark}

\newcommand{\N}{\mathbb{N}}
\newcommand{\Z}{\mathbb{Z}}
\newcommand{\Q}{\mathbb{Q}}
\newcommand{\R}{\mathbb{R}}
\newcommand{\C}{\mathbb{C}}
\renewcommand{\Im}{\mathrm{Im}}
\newcommand{\A}{\mathbb{A}}
\renewcommand{\H}{\mathcal{H}}
\newcommand{\Sp}{\mathrm{Sp}}
\newcommand{\GL}{\mathrm{GL}}
\newcommand{\Jac}{\b{G}}
\renewcommand{\Re}{\mathrm{Re}}
\renewcommand{\Im}{\mathrm{Im}}

\newcommand{\diag}{\mathrm{diag}}
\newcommand{\sgn}{\mathrm{sgn}}
\newcommand{\tr}{\mathrm{tr}\,}
\renewcommand{\b}[1]{\boldsymbol{#1}} 
\renewcommand{\a}{\mathbf{a}} 
\newcommand{\h}{\mathbf{h}} 
\newcommand{\f}[1]{\mathfrak{{#1}}}

\renewcommand{\l}{\lambda}
\renewcommand{\t}{\theta}

\newcommand{\ti}{^{\times}}
\newcommand{\back}{\backslash}
\newcommand{\T}[1]{\,{}^t\! {{#1}}} 
\newcommand{\transpose}[1]{\text{$^t\!#1$}} 
\renewcommand{\(}{\left(} \renewcommand{\)}{\right)}
\newcommand{\spmatrix}[1]{{\small\arraycolsep=0.3\arraycolsep\ensuremath{\begin{pmatrix} #1\end{pmatrix}}}} 

\def\sectionnam{\@empty}
\def\subsectionnam{\@empty}

\begin{document}
\title[On the standard $L$-function attached to Siegel-Jacobi modular forms] 
{On the standard $L$-function attached to Siegel-Jacobi modular forms of higher index}
\keywords{Jacobi group, Siegel-Jacobi modular forms, $L$-functions, Eisenstein series, Hecke operators}

\author{Thanasis Bouganis and Jolanta Marzec}
\address{Department of Mathematical Sciences\\ Durham University\\
 Durham, UK.}
\email{athanasios.bouganis@durham.ac.uk\\ jolanta.m.marzec@durham.ac.uk}
\thanks{The authors acknowledge support from EPSRC through the grant EP/N009266/1, Arithmetic of automorphic forms and special $L$-values}
\subjclass[2010]{11R42, 11F50, 11F66, 11F67 (primary), and 11F46 (secondary)}

\maketitle In this work we study the analytic properties of the standard $L$-function attached to Siegel-Jacobi modular forms of higher index, generalizing previous results of Arakawa and Murase. Furthermore, we obtain algebraicity results on special $L$-values in the spirit of Deligne's Period Conjectures. 

\tableofcontents


\section{Introduction} 
\definecolor{paleyellow}{RGB}{255,255,190}
\newcommand{\yel}[1]{\colorbox{paleyellow}{#1}}

The standard $L$-function attached to a cuspidal Siegel eigenform $f$ is perhaps one of the most well-studied automorphic $L$-functions. Indeed, its analytic properties have been extensively studied by many authors such as Andrianov and Kalinin \cite{AK}, B\"{o}cherer \cite{B83,B85,B86}, Garrett \cite{Garrett}, Piatetski-Shapiro and Rallis \cite{PS-R}, and Shimura \cite{Sh94,Sh95}. Moreover, if one assumes that $f$ is algebraic, in the sense that the Fourier coefficients 
of $f$ at infinity are algebraic, then the values of the $L$-function at specific points (usually called special $L$-values), after dividing by appropriate powers of $\pi$ and the Petersson self inner product $<f,f>$, are algebraic. Results of this kind have been obtained first by Sturm \cite{Sturm}, then extended by B\"{o}cherer and Schmidt \cite{BS} and Shimura \cite{Sh00}.

Siegel-Jacobi modular forms - called here after \cite{Kramer} - are higher dimensional generalizations of classical Jacobi forms. As in the one-dimensional case they are very closely related to Siegel modular forms. Indeed, many examples may be naturally obtained from Fourier-Jacobi expansion of Siegel modular forms. 
However, one of the main differences of these automorphic forms in comparison to Siegel modular forms is that the underlying algebraic group, the Jacobi group, is not reductive. 
In particular, this means that these automoprhic forms cannot be understood as sections of line bundles of Shimura varieties, but rather of mixed Shimura varieties \cite{Milne}. We will come back to this point later in the introduction when discussing our results regarding the algebraicity of the special $L$-values. 

Siegel-Jacobi modular forms have already been studied by many researchers. The ones that are best understood are classical Jacobi forms. Their first systematic study was carried out in 
\cite{EZ}, but they were already used in earlier papers (cf. \cite{Skoruppa_developments}).
For the higher dimensional situation we would like to mention works which are especially relevant to this paper, namely the papers of Shimura \cite{Sh78}, Ziegler \cite{Z89} and Kramer \cite{Kramer}. The approach of Ziegler is what may be called classical, Shimura's is arithmetic and Kramer's is geometric. We will come back to Shimura's approach later in the introduction.  

In spite of such a variety of methods to study Siegel-Jacobi modular forms, still not much is known about associated Dirichlet series. A systematic study of a Hecke algebra acting on the space of Siegel-Jacobi modular forms, and of the resulting standard $L$-function was started by Shintani (unpublished). However, the first results concerning analytic properties of this $L$-function were obtained by Murase - in \cite{Mu89,Mu91} he established the analytic continuation, a representation as an Euler product and a functional equation.
In this paper we not only extend the results of Murase, but also study arithmetic properties of the $L$-function at particular points. 
 
Before going any further we give a brief account of main theorems proved in this paper. For simplicity we describe them here only for Siegel-Jacobi modular forms over the rational numbers, even though our results are more general and are proved over a totally real field. First we need to introduce some notation. 

Let $S \in M_{l,l}(\mathbb{Q})$ be a positive definite half-integral symmetric matrix, and $f$ a Siegel-Jacobi modular form of weight  $k$ and index $S$ for the congruence subgroup $\b{\Gamma}_0(N)$. We give the detailed definition in  section \ref{sec:Jacobi_forms} but for the purposes of this introduction it is enough to say that $f$ is a holomorphic function on the space $\mathcal{H}_{n,l}:=\mathbb{H}_n \times M_{n,l}(\mathbb{C})$, where $\mathbb{H}_n$ is the Siegel upper half space, satisfying a particular modular property with respect to the group $\b{\Gamma}_0(N) := H(\mathbb{Z}) \rtimes \Gamma_0(N)$, a congruence subgroup of the Jacobi group $\b{G}^{n,l}(F) := H(F) \rtimes \Sp_n(F)$. Here $H(\mathbb{Z})$ denotes the $\mathbb{Z}$-points of the Heisenberg group of degree $n$ and index $l$, and $\Gamma_0(N)$ the classical congruence subgroup of level $N$ in the theory of Siegel modular forms. 

Shintani (unpublished), Murase \cite{Mu89} and Murase and Sugano \cite{MS} defined and studied Hecke operators $T(m)$ acting on $f$. Actually, this was done only for the case of $N=1$. In this work (see section \ref{Hecke algebra}) we extend this to the case of any $N$. Then, assuming that $f$ is an eigenform for all $T(m)$ with eigenvalues $\lambda(m)$ and $\chi$ is a Dirichlet character of a conductor $M$, we consider a Dirichlet series $D(s,f,\chi) = \sum_{m=1}^{\infty} \lambda(m)  \chi(m) m^{-s}$. This series is absolutely convergent for $\Re(s) > 2n+l+1$ and - as we will show in section \ref{Hecke algebra} - after multiplying by an appropriate factor it possesses an Euler product representation. More precisely, we prove the following: 

\begin{thm} Assume that the matrix $S$ satisfies the condition $M_{p}^+$ (see section \ref{Hecke algebra} for a definition) for every prime ideal $p$ with $(p,N)=1$. Then
\[
\f{L}(\chi,s) D(s+n+l/2,f,\chi) = L(s,f,\chi) := \prod_{p} L_{p}(\chi(p)p^{-s})^{-1},
\]
where for every prime number $p$ 
\[
 L_{p}(X) = \begin{cases} \prod_{i=1}^n \left((1- \mu_{p,i}X)(1-\mu^{-1}_{p,i}X ) \right)\,\,\, \mu_{p,i} \in \mathbb{C}^{\times},\,\,\hbox{if}\,\,\, (p, N)=1,\\
 \prod_{i=1}^n (1- \mu_{p,i}X)\,\,\, \mu_{p,i} \in \mathbb{C},\,\,\hbox{otherwise}. 
 \end{cases}
\]
Moreover, $\f{L}(\chi,s) = \prod_{(p,N)=1} \f{L}_{p}(\chi,s)$, where
\[
\f{L}_{p}(\chi,s) := G_{p}(\chi,s) \begin{cases}  \prod_{i=1}^{n} L_{p}(2s+2n-2i,\chi^{2}) &\mbox{if } l \in 2\mathbb{Z} \\ 
\prod_{i=1}^{n} L_{p}(2s+2n-2i+1,\chi^{2})  & \mbox{if }  l \not \in 2\mathbb{Z}\end{cases} ,
\]
and $G_{p}(\chi,s)$ is a ratio of Euler factors which for almost all $p$ is one. 
\end{thm}
The above theorem was originally shown by Murase and Sugano in the case of $N=1$, $\chi=1$ and $l=1$. We extended it to any $N$, any character $\chi$ and any $l$. Together with generalization to any $l$ certain new phenomena appear, such as for example the presence of the factor $G(\chi,s)$, which is equal to one in the case of $l=1$.
We defer a more detailed discussion to section \ref{Hecke algebra}. \newline

\noindent \textbf{Analytic properties of $L(s,f,\chi)$:} The theorem above establishes that the function $L(s,f,\chi)$ is absolutely convergent for $\Re(s) >n + \frac{l}{2}+1$ and hence holomorphic. Regarding its meromorphic continuation we prove the following:

\begin{thm}   With notation as above, assume that $\chi(-1) = (-1)^k$. Then, for some $Q | N$, the function $\left(\prod_{q | Q} L_{q}(\chi(q)q^{-s})\right) L(s,f,\chi)$ 
has a meromorphic continuation to the whole complex plane. 
\end{thm}
Actually in the full version of the theorem (Theorem \ref{Main Theorem 1}), after introducing an extra factor depending on the parity of $l$ and some Gamma factors, we also provide information on the location of the poles of the function. Our theorem extends previous work of Murase \cite{Mu89,Mu91} in various directions: we consider the case of totally real fields, non-trivial level and twisting by characters. However, perhaps the most important difference with the works \cite{Mu89,Mu91} is the method used. Even though both in our work and in these of Murase the doubling method is used, there are some very serious differences with advantages and disadvantages. The work of Murase has as its prototype the approach of Piatetski-Shapiro and Rallis \cite{PS-R} and their theory of zeta integrals. Murase uses an embedding of the form
 \[
 \b{G}^{n,l}(\mathbb{Q}) \times \b{G}^{n,l}(\mathbb{Q}) \hookrightarrow \Sp_{2n+l}(\mathbb{Q}),
 \]
 and computes an adelic zeta integral \`{a} la Piatetski-Shapiro and Rallis of a Siegel-type Eisenstein series of $\Sp_{2n+l}$ restricted to the image of the product $\b{G}^{n,l}(\mathbb{A}_{\mathbb{Q}}) \times \b{G}^{n,l}(\mathbb{A}_{\mathbb{Q}})$ against two copies of the adelic counterpart  $\mathbf{f}$ of $f$. 
 
Our approach is completely different. We use instead a map of the form
\[
 \b{G}^{n,l}(\mathbb{Q}) \times \b{G}^{m,l}(\mathbb{Q}) \rightarrow \b{G}^{m+n,l}(\mathbb{Q}),
\]
which is not quite an embedding; this
map was first used by Arakawa in \cite{Ar94}. We will later discuss in more details the differences of our approach to the one of Arakawa, but first we give a brief account of the comparison of the method employed by Murase and the one of this paper. One of the big advantages of the first approach is that one can read off analytic properties of the standard $L$-function associated to a Siegel-Jacobi modular form by making use of well-studied analytic properties of Siegel-type symplectic 
Eisenstein series. On the other hand, the method used in this paper allows us to obtain analytic properties of the standard $L$-function by studying analytic properties of Siegel-type Jacobi Eisenstein series. More precisely,
for a Dirichlet character $\chi$ with $\chi(-1)= (-1)^k$ and $m \geq n$ we prove a formula of the form
\[
<f(w), E^{n+m}(\diag[z,w],s ; \chi, k, N)> = L(s, f,\chi, s) E^{m}(z,s;f,\chi,N),
\]
where $E^{n+m}(\diag[z,w],s ; \chi, k, N)$ is the restriction under the diagonal embedding $\mathcal{H}_{n,l} \times \mathcal{H}_{m,l} \hookrightarrow \mathcal{H}_{n+m,l}$ of a Siegel-type Jacobi Eisenstein series of degree $n+m$ associated to the character $\chi$, and $E^{m}(z,s;f,\chi,N)$ is a Klingen-type Jacobi Eisenstein series of degree $m$ associated to the cuspidal form $f$ through parabolic induction. That is, we obtain an identity in the spirit of the doubling method which says that after taking the Petersson inner product of a restricted Siegel-type Eisenstein series against a cusp form, we obtain a Klingen-type Eisenstein series induced by the cusp form normalized by the standard $L$-function associated to the same cusp form. 

This identity was first obtained by Arakawa in \cite{Ar94} in the case of $N=1$ and trivial $\chi$ (and hence $k$ even), and in this paper is extended to the situation of totally real fields, arbitrary level as well as non-trivial characters $\chi$. However, we should stress here that our approach is quite different than that of Arakawa. Indeed, Arakawa's approach is modeled to the original approach of Garrett in \cite{Garrett} who invented the doubling method and applied it to the case of Siegel modular forms over $\mathbb{Q}$ of trivial level and without twists by Dirichlet characters. Our approach is modeled after the work of Shimura \cite{Sh95}, where he extended Garrett's approach to the case of totally real field, arbitrary level as well as twisting by Hecke characters.  

It is important to note here that opposite to the first map used by Murase, in the map used in this work we have the option to take $n \neq m$. And indeed we will make use of this in order to obtains results towards the analytic properties of Klingen-type Jacobi Eisenstein series (see Theorem \ref{Main Theorem 2}).

\noindent \textbf{Algebraic properties of $L$-values:} In this paper we also investigate algbebraic properties of special values of the $L$-function under consideration. The starting point of our investigation is a result of Shimura in \cite{Sh78} on the arithmeticity of Siegel-Jacobi modular forms. Namely, if we write $M_{k,S}^n$ for the space of Siegel-Jacobi modular forms of weight $k$ and index $S$, and of any congruence subgroup, and we also denote by $M_{k,S}^n(K)$ the subspace of  $M_{k,S}^n$ with the property that the Fourier expansion at infinity of an element in the space has Fourier coefficients in a subfield $K$ of $\mathbb{C}$, then it is shown in  (loc. cit.) that  $M_{k,S}^n(K)= M_{k,S}^n(\mathbb{Q}) \otimes_{\mathbb{Q}} K$. In particular, for a given $f \in M_{k,S}^n$ and a $\sigma \in Aut(\mathbb{C}/\mathbb{Q})$ one can define the element $f^{\sigma} \in M_{k,S}^n$ by letting $\sigma$ act on the Fourier coefficients of $f$. The main theorem we proved regarding algebraicity (Theorem \ref{Main Theorem on algebraicity}) is stated below in the simplest form of $N=1$. In the following, and for $l$ even, we write $\psi_S$ for the non-trivial quadratic character corresponding to the extension $K_S:= \mathbb{Q}(\sqrt{(-1)^{l/2} \det(2S)})$ if $K_S \neq \mathbb{Q}$, and we set $\psi_S =1$ otherwise. 

\begin{thm} \label{Main Theorem 2_Introduction}
Assume $n>1$ and let $0 \neq f \in \b{S}_{k,S}^{n}(\b{\Gamma},\overline{\mathbb{Q}})$ be an eigenfunction, and $\chi$ be a Dirichlet character such that $\chi(-1) = (-1)^k$. Assume that $k > 2n + l +1$ and let $\sigma \in  \mathbb{Z}$ be such that 
\begin{enumerate}
\item $2n+1 - (k - l/2)  \leq \sigma -l/2 \leq k -l/2$,  
\item $| \sigma - \frac{l}{2} - \frac{2n+1}{2} | + \frac{2n+1}{2} - (k - l/2) \in 2 \mathbb{Z}$,
\item $k > l/2 + n(1+k - l/2 -|\sigma - l/2 - (2n+1)/2| - (2n+1)/2)$,
\end{enumerate}
but exclude the cases
\begin{enumerate}
\item $\sigma = n+1+ l/2$ and $\chi^2  = 1$,
\item $\sigma = l/2$ and $\chi \psi_S  = 1$,
\item $0 < \sigma - l/2 \leq n$ and $\chi^2  = 1$.
\end{enumerate}
If we set
\[
\b{\Lambda}(s,f,\chi) := L(2s-n-l/2,f,\chi) \begin{cases} L_{\mathfrak{c}}(2s-l/2,\chi \psi_S) & \hbox{if } l \in 2\mathbb{Z},\\
1 & \hbox{if } l \notin 2\mathbb{Z},
 \end{cases}
\]
then
\[
\frac{\b{\Lambda}(\sigma/2,f,\chi)}{\pi^{e_{\sigma}} <f,f>} \in \overline{\mathbb{Q}},
\]
where  
$$e_{\sigma} =  n  (k - l + \sigma) - e \quad\mbox{ and }\quad
e := \begin{cases} n^2 + n - \sigma + l/2 & \mbox{if } 2 \sigma -l  \in 2 \mathbb{Z} \mbox{ and } \sigma \geq 2n + l/2, \\
  n^2 & \mbox {otherwise} . \end{cases}$$ 
\end{thm}

We remark here that our methods can also cover the case of $n=1$ and $F=\mathbb{Q}$ if we take $\chi$ to be the trivial character.  

Let us now try to put the above theorem in some broader context. Theorems of the above form for the standard $L$-functions of automorphic forms associated to Shimura varieties, such as Siegel and Hermitian modular forms, were obtained by many researchers, most profoundly by Shimura (see for example \cite{Sh00}). These deep results can also be understood in the general framework of Deligne's Period Conjectures for critical values of motives \cite{Deligne}. Indeed, according to the general Langlands conjectures, the standard $L$-functions of automorphic forms related to Shimura varieties can be identified with motivic $L$-functions, and hence the algebraicity results for the special values of the automorphic $L$-functions can also be seen as a confirmation of Deligne's Period Conjecture, albeit is usually hard to actually show that the conjectural motivic period agrees with the automorphic one.

However, Siegel-Jacobi modular forms, and in particular the algebraicity result of the above theorem, do not fit in this framework. Indeed, since Jacobi group is not reductive, it does not satisfy the properties needed for associating a Shimura variety to it, and hence we are not in the situation described in the previous paragraph. On the other hand,
the Jacobi group can actually be associated with a geometric object, namely with a mixed Shimura variety, as it is explained for example in \cite{Milne,Kramer}. Of course, we cannot expect that the standard $L$-function studied here can be in general identified with a motivic one. Nevertheless, it is very tempting to speculate that it could be identified with an $L$-function of a mixed motive, and hence the theorem above could be seen as a confirmation of the generalization of Deligne's Period Conjecture to the mixed setting as for example stated by Scholl in \cite{Scholl}.

\noindent \textbf{What is not done in this paper:} This paper is already quite long, and we have decided to defer some interesting questions for a forthcoming work. In particular, we mention the following:

\begin{enumerate}
\item In all our theorems we assume a particular parity condition between the character $\chi$ and the weight $k$ of the Siegel-Jacobi modular form. It is, of course, very important to be able to relax this condition and obtain the theorems for any finite character $\chi$, independent of the weight $k$. 
\item In order to obtain a generalization of Theorem \ref{Main Theorem 2_Introduction} above we need to assume the Property A (see section 10). Even though there are many cases where the Property A holds, it is undoubtedly very interesting to weaken or even completely remove this condition. Furthermore, we had to exclude the case of $F=\mathbb{Q}$ and $n=1$, and it is interesting to extend our methods to cover also this case. Finally, one could try to obtain a reciprocity law for the action of the absolute Galois group on the normalized special values. That is with $\sigma_0$ as in Theorem \ref{Main Theorem 2_Introduction} to obtain results of the form 
\[
\left(\frac{\b{\Lambda}(\sigma/2, f.\chi)}{\pi^{e_{\sigma}}\omega(\chi)<f,f>}\right)^{\sigma} = \frac{\b{\Lambda}(\sigma/2, f^{\sigma}.\chi^{\sigma})}{\pi^{e_{\sigma}}\omega(\chi^{\sigma})<f^{\sigma},f^{\sigma}>},\,\,\,\sigma \in Gal(\overline{\mathbb{Q}}/\mathbb{Q}),
\]
where $\omega(\chi)$ is a product of Gauss sums associated to the character $\chi$ and $\chi^\sigma := \sigma \circ \chi$. 
\end{enumerate}

\noindent \textbf{Brief description of each section:} We finish this introduction by giving a short description of each section. In the second section we set most common notation used throughout this paper. In section three we introduce the notion of Siegel-Jacobi modular forms over a totally real field $F$, as well as the notion of adelic or automorphic Siegel-Jacobi forms. To the best of our knowledge their systematic study has not appeared before in the literature, notably Proposition \ref{prop:adelic_Fourier_expansion} on the adelic Fourier expansion. In section four we develop the theory of Klingen-type Eisenstein series. We do this in greatest generality possible. Again, to the best of our knowledge, a systematic study of the adelized Klingen-type Jacobi Eisenstein series has not appeared before in the literature. In sections five and six we employ the doubling method in the way described above and compute the Petersson inner product of a restricted Siegel-type Jacobi Eisenstein series against a cuspidal Siegel-Jacobi form. In section seven we introduce the theory of Hecke operators in the Jacobi setting and extend previous results of Murase and Sugano. In the next section we turn our attention to the analytic properties of Siegel-type Jacobi Eisenstein series. We build on an idea going back to a work of B\"{o}cherer \cite{B83} and more recently of Heim \cite{Heim}. After establishing the analytic properties of these Eisenstein series we use the results established in section 6 to obtain Theorem \ref{Main Theorem 1} on the analytic properties of the standard $L$-function. Moreover we also establish Theorem \ref{Main Theorem 2} on the analytic continuation of Klinegn-type Jacobi Eisenstein series. Finally, in the last section of this paper we turn to the algebraic properties of the standard $L$-function at specific intervals, which we call special $L$-values. The main result of this section is Theorem \ref{Main Theorem on algebraicity}.  


\section{Notation}\label{sec:notation}
Throughout the paper we use the following notation:
\begin{itemize}
	\item $F$ denotes a totally real algebraic number field of degree $d$, $\f{d}$ the different of $F$, and $\f{o}$ its ring of integers;
	\item $\A$ stands for the adeles of $F$; we write $\a$ and $\h$ for the sets of archimedean and non-archimedan places of $F$ respectively, so that e.g. $\A_{\h}:=\prod'_{v\in\h} F_v$ (restricted product) and $\A_{\a}:=\prod_{v\in\a} F_v$ denote the finite and infinite adeles of $F$; for $x\in\A$ we will write $x_{\h}, x_{\a}$ meaning the finite and infinite part of $x$, correspondingly; for a ring $R$ we use the superscript $R\ti$ to denote the invertible elements in $R$;
	\item A finite adele $a \in \mathbb{A}_{\mathbf{h}}$ corresponds to a fractional ideal $\mathfrak{a}$ of $F$ via $\mathfrak{a}:= \prod_{v \in \mathbf{h}} \mathfrak{p}_v^{n_v}$, where $a_v = \pi_v^{n_v} \mathfrak{o}_v^{\times}$, $n_v \in \mathbb{Z}$, $\pi_v$ a uniformiser at $v$ and $\mathfrak{p}_v$ the corresponding prime ideal at the finite place $v$. We will call $\mathfrak{a}$ the ideal corresponding to $a$.
		\item We define $\mathbb{Z}^{\mathbf{a}} := \mathbb{Z}^d$, and a typical element $k \in \mathbb{Z}^{\mathbf{a}}$ is of the form $k = (k_v)_{v \in \mathbf{a}}$ with $k_v \in \mathbb{Z}$. Moreover for an integer $\mu \in \mathbb{Z}$ we write $\mu \mathbf{a}: =( \mu,\mu,\ldots,\mu) \in \mathbb{Z}^{\mathbf{a}}$. 
		\item For an adelic Hecke character $\chi : \mathbb{A}^{\times}/F^{\times} \rightarrow \mathbb{C}^{\times}$, we will write $\chi^*$ for the corresponding ideal Hecke character obtained by class field theory. Furthermore, if $\chi$ is finite, then its infinite part is of the form $\chi_{\mathbf{a}}(x_{\mathbf{a}}) = \prod_{v \in \mathbf{a}} \left(\frac{x_v}{|x_v|}\right)^{k_v}$, for $k_v \in \mathbb{Z}$. We then write $\sgn_{\mathbf{a}}(x_{\mathbf{a}})^k$ for $\chi_{\mathbf{a}}(x_{\mathbf{a}})$ where $k := (k_v) \in \mathbb{Z}^{\mathbf{a}}$.

	\item $M_{l,n}$ denotes the set of $l\times n$ matrices, and we set $M_n := M_{n,n}$. We write $Sym_n\subset M_{n}$ for the subset of symmetric matrices;
	if $A\in M_{l,n}$ and $B\in M_{l,m}$, then $(A\, B)\in M_{l,n+m}$ denotes concatenation of the matrices $A,B$; if $S\in Sym_l, x\in M_{l,n}$, we set $S[x]:=\T{x}Sx$;
	\item For an invertible matrix $x$ we define $\tilde{x} := \transpose{x}^{-1}$;
	\item For two matrices $a \in M_n$ and $b \in M_m$ we define $\diag[a , b]:= \left(\begin{matrix} a & 0 \\ 0 & b \end{matrix} \right) \in M_{n+m}$;  
	\item We set $\mathbf{e}_{\mathbf{a}}(x) := \prod_{v \in \mathbf{a}} e(x_v):=\prod_{v \in \mathbf{a}} e^{2\pi ix_v}$ for $x= \prod_{v \in \mathbf{a}} x_v \in \mathbb{C}^{\mathbf{a}}$.
	\item $G^n$ stands for the algebraic group $\Sp_n$ whose $F$-points are defined as follows: 
	$$\Sp_n(F):=\left\{ g\in \mathrm{SL}_{2n}(F)\colon\T{g}\(\begin{smallmatrix}  &  -1_n   \\  1_n  & \\ \end{smallmatrix}\) g=\(\begin{smallmatrix}  &  -1_n   \\  1_n  &    \\ \end{smallmatrix}\)\right\};$$
	For an element $g \in \Sp_n$ we write $g = \left( \begin{matrix} a_g & b_g \\ c_g & d_g \end{matrix}\right)$, where $a_g,b_g,c_g,d_g \in M_n$; 
	\item for $l$ a fixed positive integer, $\Jac^{n,l}:=H^{n,l}\rtimes \Sp_n$ denotes the Jacobi group with $H^{n,l}$ denoting the Heisenberg subgroup, whose global points are defined as
	$$\Jac^{n,l}(F):=\{\b{g}=(\l,\mu,\kappa)g: \l,\mu\in M_{l,n}(F), \kappa\in Sym_l(F), g\in G^n(F)\} ,$$
	$$H^{n,l}(F):=\{ (\l,\mu,\kappa) 1_{2n}\in\Jac^{n,l}(F)\} ;$$
	the group law is given by
$$(\l,\mu,\kappa)g (\l',\mu',\kappa')g':=(\l+\tilde{\l},\mu+\tilde{\mu}, \kappa+\kappa'+\l\T{\tilde{\mu}}+\tilde{\mu}\T{\l}+\tilde{\l}\T{\tilde{\mu}}-\l'\T{\mu'}) gg',$$
where $(\tilde{\l}\,\tilde{\mu}):=(\l'\,\mu')g^{-1}=(\l'\T{d}-\mu'\T{c}\quad \mu'\T{a}-\l'\T{b})$,
	the identity element of $\Jac^{n,l}(F)$ is $1_H1_{2n}$, where $1_H:=(0,0,0)$ denotes the identity element of $H^{n,l}(F)$, i.e. we always suppress the indices $n,l$ in $1_H$ as its size will be clear from the context;\\
	whenever it does not lead to any confusion, we omit superscripts and write $G, \Jac, \Jac^n$ or $H$;\\ following the convention described above,
	$G(\A)=\prod'_{v\in\h\cup\a} G(F_v)=G_{\h}G_{\a}$, where $G_{\h}=\prod'_{v\in\h}G(F_v)$, $G_{\a}=\prod_{v\in\a}G(F_v)$;
	\item $\H_{n,l}:=(\mathbb{H}_n\times M_{l,n}(\C))^{\a}$, where
	$\mathbb{H}_n:=\{ \tau\in Sym_n(\C): \Im(\tau) \mbox{ positive definite}\}$;
	an element $z\in\H_{n,l}$ will be written as $z=(z_v)_{v\in\a}=(\tau ,w)$, where $\tau =(\tau_v)_{v\in\a}\in\mathbb{H}_n^{\a}$, $w=(w_v)_{v\in\a}\in M_{l,n}(\C)^{\a}$; we distinguish an element $\b{i}_0:=(\b{i},0)\in\H_{n,l}$, where $\b{i}:=(i1_n)^{\a}$;\\
	for $z\in\H_{n,l}$ we define $\delta(z):=\det (\Im (z)):=\prod_{v\in\a} \det(\Im (z_v)))$;
	\item For a fractional ideal $\f{b}$ and an integral ideal $\f{c}$ we define the following subgroups of $\b{G}(\mathbb{A})$:
	$$K[\mathfrak{b},\mathfrak{c}] := K^n[\mathfrak{b},\mathfrak{c}]:= K_{\h}[\mathfrak{b},\mathfrak{c}]\Jac_{\a},$$	
	$$K_0[\mathfrak{b},\mathfrak{c}] := K^n_0[\mathfrak{b},\mathfrak{c}]:= K_{\h}[\mathfrak{b},\mathfrak{c}]\times K_{\infty}\quad\mbox{and}\quad
	K:= K^n:=K_{\h}[\f{b},\f{c}] (H_{\mathbf{a}}^{n,l} \rtimes D_{\infty}^{\a}),$$
	where
	$K_{\infty}\simeq Sym_{l}(\mathbb{R})^{\mathbf{a}} \rtimes D_{\infty}^{\mathbf{a}}\subset H^{n,l}(\R)^{\a}\rtimes\Sp_n(\R)^{\a}$ is the stabilizer of the point $\mathbf{i}_0$, and $D_{\infty}$ is the maximal compact subgroup of $\Sp_n(\mathbb{R})$,
	$$K_{\h}[\mathfrak{b},\mathfrak{c}]:=C_{\h}[\mathfrak{o},\mathfrak{b}^{-1},\mathfrak{b}^{-1}] \rtimes D_{\h}[\mathfrak{b}^{-1},\mathfrak{bc}]\subset\Jac_{\h},$$
	$$
	C_{\h}[\mathfrak{o},\mathfrak{b}^{-1},\mathfrak{b}^{-1}] := \{ (\lambda,\mu,\kappa) \in \prod_{v \in \mathbf{h}}\!\!\, ' H(F_v):\forall v\in\h\hspace{-0.3cm}\begin{smallmatrix} \lambda_v  \in M_{l,n}(\mathfrak{o}_v), & \hspace{-0.6cm}\mu_v \in M_{l,n}(\mathfrak{b}_v^{-1}),\\ \hspace{1cm}\kappa_v \in Sym_l( \mathfrak{b}_v^{-1})\end{smallmatrix} \} ,
	$$
	$$D_{\h}[\mathfrak{b}^{-1},\mathfrak{bc}]:=\prod_{v\in\h}D_v[\mathfrak{b}^{-1},\mathfrak{bc}],$$
	$$D_v[\mathfrak{b}^{-1},\mathfrak{bc}]:=\left\{ x = \left(\begin{matrix} a_x & b_x  \\ c_x & d_x \end{matrix} \right) \in G_v:\begin{smallmatrix} a_x\in M_n(\f{o}_v), & b_x\in M_n(\mathfrak{b}^{-1}_v),\\ c_x\in M_n(\mathfrak{b}_v\mathfrak{c}_v), & d_x\in M_n(\f{o}_v)\end{smallmatrix}\right\} ;$$
	\item For $r\in\{ 0,1,\ldots ,n\}$ we define parabolic subgroups of $G^n$ and $\b{G}^n$ as follows:
	$$P^{n,r}(F):=\left\{ \begin{pmatrix} a_1 &  0 & b_1 & b_2 \\
a_3 & a_4 & b_3 & b_4 \\  
c_1 & 0 & d_1 & d_2 \\
 0 &0   & 0 & d_4
\end{pmatrix}\in G^n(F): a_1,b_1,c_1,d_1\in M_r(F)\right\} ,$$
	$$\b{P}^{n,r}(F):=\left\{ ((\l\, 0_{l, n-r}),\mu,\kappa )g:\l\in M_{l, r}(F), \mu\in M_{l, n}(F),\kappa\in Sym_l(F), g\in P^{n,r}(F)\right\} ;$$
	additionally, we set $\b{P}^n:=\b{P}^{n,0}$.
\end{itemize}

\section{Siegel-Jacobi modular forms of higher index}\label{sec:Jacobi_forms}

In this section we introduce the notion of Siegel-Jacobi modular form, both from a classical and an adelic point of view, and then explain the relation between the two notions. The content of this section is well-known to researchers working on Jacobi forms, but to the best of our knowledge it has not been written elsewhere in such detail and generality. Our exposition follows mainly \cite{Mu89, Z89}.

\subsection{Siegel-Jacobi modular forms}

For two natural numbers $l,n$, we consider the Jacobi group $\b{G}:=\b{G}^{n,l}$ of degree $n$ and index $l$ over a totally real algebraic number field $F$.  Note that the global points $\Jac (F)$ may be viewed as a subgroup of $G^{l+n}(F):=\Sp_{l+n}(F)$ via the embedding
\begin{equation}\label{embedding of Jacobi to Symplectic}
\b{g}=(\l,\mu,\kappa)g\longmapsto \(\begin{smallmatrix}
1_l & \l  &    \kappa -\mu\T{\l}   & \mu    \\
& 1_n &	  \T{\mu} 	&	   	 \\	
&	  &    1_l		&	   	 \\
&	  &   -\T{\l}	& 1_n  	 \\
\end{smallmatrix}\)
\(\begin{smallmatrix}
1_l	&	 &     &    \\
& a  &     & b  \\
&	 & 1_l &    \\
& c	 &     & d  \\
\end{smallmatrix}\) ,\quad g=\(\begin{smallmatrix} a & b \\ c & d 
\end{smallmatrix}\) .
\end{equation}

We write  $\{\sigma_v:F\hookrightarrow\R,\,\, v \in \mathbf{a}\}$ for the set of real embeddings of $F$. Each $\sigma_v$ induces an embedding $\b{G}(F) \hookrightarrow \b{G}(\mathbb{R})$; we will write $(\l_v,\mu_v,\kappa_v)g_v$ for $\sigma_v(\b{g})$. 
The group $\Jac(\mathbb{R})^{\mathbf{a}}$ acts on $\H_{n,l}:=(\mathbb{H}_n\times M_{l,n}(\C))^{\a}$ component wise via
\[
\b{g} z  = \b{g}(\tau,w)=(\l,\mu,\kappa)g (\tau,w)=\prod_{v\in\a} (g_v\tau_v, w_v \lambda(g_v,\tau_v)^{-1}+\l_v g_v \tau_v+\mu_v),
\]
where $g_v\tau_v=(a_v\tau_v+b_v)(c_v\tau_v+d_v)^{-1}$ and $\lambda(g_v,\tau_v) :=(c_v\tau_v+d_v)$ for $g_v=\(\begin{smallmatrix} a_v & b_v \\ c_v & d_v \end{smallmatrix}\)$.

For $k\in\Z^{\a}$ and a matrix $S\in Sym_l(\f{d}^{-1})$ we define the factor of automorphy of weight $k$ and index $S$ by
$$J_{k,S}\colon\Jac^{n,l}(F)\times\H_{n,l}\to \C$$

$$J_{k,S} (\b{g},z)=J_{k,S} (\b{g},(\tau,w)):=\prod_{v\in\a} j(g_v,\tau_v)^{k_v} \mathcal{J}_{S_v}(\b{g_v},\tau_v,w_v) ,$$
where $\b{g}=(\l,\mu,\kappa)g$, $j(g_v,\tau_v)=\det(c_v\tau_v+d_v)=\det(\lambda (g_v,\tau_v))$ and
\begin{multline*}
\mathcal{J}_{S_v}(\b{g_v},\tau_v,w_v)=e(-\tr (S_v\kappa_v)+\tr (S_v[w_v] \lambda(g_v,\tau_v)^{-1}c_v)\\
-2\tr (\T{\l}_v S_v w_v \lambda(g_v,\tau_v)^{-1})-\tr (S_v[\l_v] g_v \tau_v))
\end{multline*}
with $e(x):=e^{2\pi ix}$, and we recall that $S[x]=\T{x}Sx$. A rather long but straightforward calculation shows that $J_{k,S}$ satisfies the usual cocycle relation: 
\begin{equation}\label{eq:aut_fac}
J_{k,S} (\b{g}\b{g}',z)=J_{k,S} (\b{g},\b{g}'\,z)J_{k,S} (\b{g}',z).
\end{equation}	
For a function $f \colon\H_{n,l}\to\C$ we define
\begin{equation}\label{eq:|k,S_action}
(f|_{k,S}\, \b{g})(z):=J_{k,S} (\b{g},z)^{-1}f(\b{g}\,z).
\end{equation}
The property \eqref{eq:aut_fac} implies that
$$(f|_{k,S}\, \b{gg}')(z)=(f|_{k,S}\, \b{g}|_{k,S}\, \b{g}')(z).$$

A subgroup $\b{\Gamma}$ of $\b{G}(F)$ will be called a congruence subgroup if there exist a fractional ideal $\f{b}$ and an integral ideal $\f{c}$ of $F$ such that $\b{\Gamma}$ is a subgroup of finite index of the group $G(F) \cap \b{g}K[\mathfrak{b},\mathfrak{c}] \b{g}^{-1}$ for some $\b{g} \in \b{G}_{\mathbf{h}}$. 

Of particular interest will be the congruence subgroup,
\begin{multline*}
\hspace{-0.3cm}\b{\Gamma}_0(\mathfrak{b},\mathfrak{c})\! :=\b{\Gamma}_0^{n,l}(\f{b},\f{c})\! :=\!\{ (\l,\mu,\kappa)\!\(\begin{smallmatrix} a & b \\ c & d 
\end{smallmatrix}\) \!\in\! \b{G}(F)\! : \l\!\in\! M_{l,n}(\f{o}),\mu\!\in\! M_{l,n}(\f{b}^{-1}),\kappa\!\in\! Sym_l(\f{b}^{-1}),\\
a,d\in M_n(\f{o}), b\in M_n(\f{b}^{-1}), c\in M_n(\f{bc})\}.
\end{multline*}

Often we will be given a congruence subgroup $\b{\Gamma}$ equipped with a homomorphism $\chi : \b{\Gamma} \rightarrow \mathbb{C}^{\times}$. For example, given a Hecke character $\chi$ of $F$ of conductor $\mathfrak{f}_\chi$ dividing $\mathfrak{c}$, we can extend it to a homomorphism
\[
\chi : \b{\Gamma}_0(\mathfrak{b},\mathfrak{c}) \rightarrow \mathbb{C}^{\times},\,\,\,\chi\left((\lambda,\mu,\kappa)
	\begin{pmatrix} a & b \\ c & d \end{pmatrix}  \right) := \chi(\det d).
\]

We now consider an $S \in \mathfrak{b}\mathfrak{d}^{-1}\mathcal{T}_l$ where 
\begin{equation}\label{index}
\mathcal{T}_l:= \{ x \in Sym_l(F):\tr(xy)\in\mathfrak{o} \mbox{ for all } y \in Sym_l(\mathfrak{o}) \} .
\end{equation}
Moreover we assume that $S$ is positive definite in the sense that if we write $S_v := \sigma_v(S) \in Sym_{l}(\mathbb{R})$ for $v \in \mathbf{a}$, then all $S_v$ are positive definite.

\begin{defn}\label{definition modular Jacobi}
Let $k$ and $S$ be as above, and $\b{\Gamma}$ a congruence subgroup equipped with a homorphism $\chi$. A Siegel-Jacobi modular form of weight $k\in\Z^{\a}$, index $S$, level $\Gamma$ and Nebentypus $\chi$ is a holomorphic function $f\colon\H_{n,l}\to\C$ such that
	\begin{enumerate}
	\item \label{def:Phi} $f|_{k,S}\, \b{g} =\chi(\b{g})f$ for every $\b{g} \in \b{\Gamma}$,
	\item for each $g\in G^n(F)$, $f|_{k,S}\, g$ admits a Fourier expansion 
of the form
\[
f|_{k,S}\, g(\tau,w)=\sum_{\substack{t\in L\\ t\geq 0}}\sum_{r\in M} c(\b{g};t,r) \mathbf{e}_{\mathbf{a}}(\tr(t\tau)) \mathbf{e}_{\mathbf{a}}(\tr(\T{r}w))\,\,\,\,(*)
\]
for some appropriate lattices $L \subset Sym_n(F)$ and $M \subset M_{l,n}(F)$, where $t\geq 0$ means that $t_v$ is semi-positive definite for each $v\in\a$.
	\end{enumerate}
We will denote the space of such functions by $M^n_{k,S}(\b{\Gamma} ,\chi)$.	
\end{defn}
The second property is really needed only in the case of $n=1$ and $F = \mathbb{Q}$ thanks to the K\"{o}cher principle for Siegel-Jacobi forms, as it is explained for example in \cite[Lemma 1.6]{Z89}.

We note that if $f\in M^n_{k,S}(\b{\Gamma}_{0}(\f{b},\f{c}),\chi)$, then
$$f (\tau,w)=\sum_{\substack{t\in\f{bd}^{-1}\mathcal{T}_n\\ t\geq 0}}\sum_{r\in \f{bd}^{-1}\mathcal{T}_{l,n}} c(t,r) \mathbf{e}_{\mathbf{a}}(\tr(t\tau)) \mathbf{e}_{\mathbf{a}}(\tr(\T{r}w)),$$
where $\mathcal{T}_{l,n}:= \{ x \in M_{l,n}(F):\tr(\T{x}y)\in\mathfrak{o} \mbox{ for all } y \in M_{l,n}(\mathfrak{o}) \}$ .

We say that $f$ is a cusp form if in the expansion $(*)$ above for every $g \in G^n(F)$, we have $c(\mathbf{g};t,r) = 0$ unless $\begin{pmatrix} S_v & r_v \\ \transpose{r_v} & t_v \end{pmatrix}$ is positive definite for every $v \in \mathbf{a}$. The space of cusp forms will be denoted by $S^n_{k,S}(\b{\Gamma},\chi)$.

We now introduce the notion of Petersson inner product for Jacobi forms, following \cite{Z89}. Let $f$ and $g$ be Jacobi forms of weight $k$, one of which is a cusp form. Moreover, assume that both $f$ and $g$ are of level $\b{\Gamma}$. For $z = (\tau,w) \in \mathcal{H}_{n,l}$ we write $\tau = x + i y$ with $x,y \in Sym_n(F_{\a})$ and $w = u + i v$ with $u,v \in M_{l,n}(F_{\a})$. Let $dz := d(\tau,w):= \det(y)^{-(l+n+1)}dx dy du dv$ and set $\Delta_{S,k}(z):= \det(y)^k \mathbf{e_a}(-4 \pi \tr(\T{v}Sv y^{-1}))$. Then we define
\[
<f,g>_{\Gamma}:= \int_{A} f(z) \overline{g(z)} \Delta_{S,k}(z) dz,\,\,\,\, A:= \b{\Gamma} \setminus \mathcal{H}_{n,l} ,
\] 
and
\[
<f,g>:= vol(A)^{-1} \int_{A} f(z) \overline{g(z)} \Delta_{S,k}(z) dz,
\]
so that the latter is independent of the group $\b{\Gamma}$ as long as both $f$ and $g$ are in $M_{k,S}^n(\b{\Gamma},\chi)$. As it is explained in \cite{Z89}, the volume differential $dz$ is selected in such a way that $vol(A) = vol(\Gamma \setminus \mathbb{H}_n^{\mathbf{a}})$ where $\Gamma$ is the symplectic part of $\b{\Gamma}$.

\subsection{Adelic Siegel-Jacobi modular forms}

We keep writing $\b{G}:=\b{G}^{n,l}$ for the Jacobi group of degree $n$ and index $l$. For two ideals $\mathfrak{b}$ and $\mathfrak{c}$ of $F$, of which $\mathfrak{c}$ is integral, we recall that we have defined the open subgroups $K_{\h}[\mathfrak{b},\mathfrak{c}] \subset \b{G}_{\h}$, $D_{\h}[\mathfrak{b}^{-1},\mathfrak{bc}]\subset G^n_{\h}$ in Section \ref{sec:notation}.

\begin{lem}\label{strong_approx_Jac}The strong approximation theorem holds for the algebraic group $\b{G}$. In particular,
$$\b{G}(\mathbb{A}) = \b{G}(F) K_{\h}[\mathfrak{b},\mathfrak{c}] \b{G}_{\mathbf{a}}.$$
\end{lem}
\begin{proof} We give a sketch of the proof. We first observe that the strong approximation holds for the Heisenberg group. Indeed, its center $Z$ is isomorphic to the group $Sym_{l}$ of symmetric matrices, and we have $H^{n,l} / Z \cong M_{n,l} \times M_{n,l}$. Furthermore, the strong approximation holds for the symmetric matrices (as an additive group) and the same holds also for $M_{n,l} \times M_{n,l}$. From this it is easy to see that the strong approximation holds for $H^{n,l}$. Then, for the whole Jacobi group, it is enough to observe that the strong approximation holds for $\Sp_n$ with respect to the subgroup $D[\mathfrak{b}^{-1},\mathfrak{bc}]$ (see \cite{Sh95}), and hence the statement follows by observing that the Heisenberg group is, by definition, a normal subgroup of $\Jac$.  
\end{proof}

We now fix once and for all an additive character $\Psi :\mathbb{A} / F \rightarrow \mathbb{C}^{\times}$ as follows. We write $\Psi = \prod_{v \in \mathbf{h}} \Psi_v \prod_{v \in \mathbf{a}} \Psi_v$ and define 
$$\Psi_v(x_v):=\begin{cases}
e(-y_v), & v\in\h\\
e(x_v), & v\in\a ,
\end{cases}$$
where $y_v \in \Q$ is such that $\mathrm{Tr}_{F_v/\Q_p}(x_v) - y_v \in \Z_p$ for $p:= v \cap \Q$. 
Given a symmetric matrix $S \in Sym_l(F)$ we define a character $\psi_S : Sym_{l}(\A) /Sym_l(F) \rightarrow \C^{\times}$ by taking $\psi_S(\kappa) := \Psi(\tr(S\kappa).$  

We consider an adelic Hecke character $\chi : \mathbb{A}_F^{\times} / F^{\times} \rightarrow \mathbb{C}^{\times}$ of $F$ of finite order such that $\chi_v(x) = 1$ for all $x \in \mathfrak{o}_v^{\times}$ with $x-1 \in \mathfrak{c}_v$. We extend this character to a character of the group $K_0[\mathfrak{b},\mathfrak{c}]$ by setting $\chi(w) := \prod_{v | \mathfrak{c}} \chi_v(\det(a_{g}))^{-1}$ for $w = hg\in K_{0}[\mathfrak{b},\mathfrak{c}]$. 

Now, let $k \in \Z^{\mathbf{a}}$ and $S \in Sym_{l}(F)$ be such that $S \in \mathfrak{b}\mathfrak{d}^{-1}\mathcal{T}_l$ with $\mathcal{T}_l$ as in \eqref{index}. Moreover, let $K$ be an open subgroup of  $K[\mathfrak{b},\mathfrak{c}]$ for some $\mathfrak{b}$ and $\mathfrak{c}$.

\begin{defn} \label{definition automorphic Jacobi} An adelic Siegel-Jacobi modular form of degree $n$, weight $k$, index $S$ and character $\chi$, with respect to the congruence subgroup $K$ is a function $\mathbf{f}: \b{G}(\mathbb{A}) \rightarrow \mathbb{C}$ such that 
\begin{enumerate}
\item $\mathbf{f}\left((0,0,\kappa) \gamma \b{g} w \right) =  \chi(w) J_{k,S}(w,\mathbf{i}_0)^{-1} \psi_S(\kappa) \mathbf{f}(\b{g})$, for all $\kappa \in Sym_l(\mathbb{A})$, $\gamma \in \b{G}(F)$, $\b{g} \in \b{G}(\mathbb{A})$ and $w \in  K \cap K_{0}[\mathfrak{b},\mathfrak{c}]$;
\item for every $\b{g} \in \b{G}_{\mathbf{h}}$ the function $f_{\b{g}}$ on $\mathcal{H}_{n,l}$ defined by the relation
\[
(f_{\b{g}}|_{k,S}\b{y})(\mathbf{i}_0) := \mathbf{f}(\b{gy}) \quad\mbox{for all } \b{y} \in \b{G}_{\mathbf{a}}
\]
 is a Siegel-Jacobi modular form for the congruence group $\b{\Gamma}^{\b{g}} := \b{G}(F) \cap \b{g} K \b{g}^{-1}$. 
\end{enumerate}
\end{defn}

Note that the relation (1) is well defined. Indeed, thanks to the strong approximation for $Sym_l$ we may write $\kappa = \kappa_{F} \kappa_\mathbf{h} \kappa_{\mathbf{a}}$ with $\kappa_{F} \in Sym_l(F)$, $\kappa_{\mathbf{h}} \in \prod_{v \in \mathbf{h}} Sym_l(\mathfrak{b}_v^{-1})$ and $\kappa_{\mathbf{a}}\in \prod_{v \in \mathbf{a}}Sym_l(\mathbb{R})$. Furthermore, observe that $\psi_{S}(\kappa) = \prod_{v \in \mathbf{a}} \psi_{S,v}(\kappa_v)= J_{k,S}((0,0,\kappa),\mathbf{i}_0)^{-1}$ since $\psi_{S,\mathbf{h}}(\kappa_{\mathbf{h}}) = 1$ by our choice of the matrix $S$.

We denote the space of adelic Siegel-Jacobi modular forms by $\mathcal{M}_{k,S}^{n}(K,\chi)$.
As in the case of Siegel modular forms (see for example \cite[Lemma 10.8]{Sh96}) we can use  Lemma \ref{strong_approx_Jac} to establish a bijection between adelic Siegel-Jacobi forms and Siegel-Jacobi modular forms. Indeed, for any given $\b{g} \in \b{G}_{\mathbf{h}}$ we have the bijective map
\begin{equation}\label{bijection between classical and adelic}
\mathcal{M}_{k,S}^{n}(K,\chi) \rightarrow M_{k,S}^n(\b{\Gamma}^{\b{g}},\chi_{\b{g}})
\end{equation}
given by $\mathbf{f} \mapsto f_{\b{g}}$, with notation as in the Definition \ref{definition automorphic Jacobi} and $\chi_{\b{g}}$ the character on $\b{\Gamma}^{\b{g}}$ defined as $\chi(\gamma) := \chi(\b{g}^{-1}\gamma\b{g})$. Furthermore, we say that $\mathbf{f}$ is a cusp form, and we denote this space by $\mathcal{S}_{k,l}^{n}(K,\chi)$ if in the above notation $f_{\b{g}}$ is a cusp form for all $\b{g} \in \b{G}_{\mathbf{h}}$.  We will often use the bijection above with $\b{g}=1$. In this case, if we start with an adelic Siegel-Jacobi form $\mathbf{f}$, we will write $f$ for the Siegel-Jacobi modular form corresponding to $\mathbf{f}$.

We finish this section with a formula for Fourier expansion of adelic Siegel-Jacobi forms.
\begin{prop}\label{prop:adelic_Fourier_expansion}
Every Siegel-Jacobi form $\mathbf{f}\in\mathcal{M}_{k,S}^{n}(K[\mathfrak{b},\mathfrak{c}],\chi)$ admits Fourier expansion of the form
\begin{equation}\label{eq:adelic_Fourier_expansion}
\mathbf{f}\left((\l,\mu,0)\(\begin{matrix} q & \sigma\tilde{q}\\ & \tilde{q}\end{matrix}\)\right) =\sum_{\substack{t\in L\\ t\geq 0}}\sum_{r\in M} c(t,r;q,\lambda)\mathbf{e}_{\A}(\tr(t\sigma))\mathbf{e}_{\A}(\tr(\T{r}\lambda\sigma +\T{r}\mu)),
\end{equation}
where $\sigma\in Sym_n(\A), q\in\GL_n(\A), \l ,\mu\in M_{l,n}(\A)$ are such that $\lambda_v q_v \in M_{l,n}(\mathfrak{b}_v^{-1})$ for all $v \in \mathbf{h}$. Moreover, the coefficients $c(t,r;q,\lambda)$ satisfy the following properties:
\begin{enumerate}[leftmargin=0.7cm]
	\item $c(t,r;q,\lambda)=\Psi_{\mathbf{a}}(\tr(S[\lambda]\sigma ))\mathbf{e}_{\mathbf{a}}(\tr(S[\l](iq\T{q})))(\det q)_{\a}^k\mathbf{e}_{\a}(i\tr(\T{q}tq+\T{q}\T{r}\lambda q))c_0(t,r;q,\l)$, where $c_0(t,r;q,\l)$ is a complex number that depends only on $\mathbf{f},t,r,q_{\h}$ and $\l_{\h}$.
	\item $c(t,r;aq,\lambda a^{-1})=\chi(\det a)c(\T{a}ta,ra;q,\lambda)$ for every $a\in\GL_n(F)$.
	\item $c(t,r;q,\lambda)\neq 0$ only if $(\T{q}tq)_v\in (\f{bd}^{-1}\mathcal{T}_n)_v$ and $e_v(\tr (\T{q}_v\T{r}_v(M_{l,n}(\f{b}_v^{-1})))=1$ for every $v\in\h$.
\end{enumerate}
\end{prop}
\begin{proof}
First of all, note that it is enough to provide a formula for $\mathbf{f}$ at $(\l,\mu,\kappa)g$ with $\kappa =0$ (thanks to the relation (1)) and $g$ of the form as in the hypothesis. 

Let $X_{l,n}:=\{ \nu\in M_{l,n}(\A):\nu_v\in M_{l,n}(\f{b}^{-1}_v)\mbox{ for all } v\in\h\}$ and $X:=\{ x\in X_{n,n}\colon x=\T{x}\}$. As it was observed in \cite[Lemma 9.6]{Sh96}, we can write $\sigma =s+qx\T{q}$ and $\l s+\mu =m+\nu\T{q}$ with $s\in Sym_n(F), x\in X, m\in M_{l,n}(F)$ and $\nu\in X_{l,n}$. Then:
\begin{align*}
\mathbf{f}((\l,\mu,0)\(\begin{matrix} q & \sigma\tilde{q}\\ & \tilde{q}\end{matrix}\)) &=\mathbf{f}(\(\begin{matrix} 1 & s\\ & 1\end{matrix}\) (\l, \l s+\mu, \lambda s \transpose{\lambda})\(\begin{matrix} q & qx\\ & \tilde{q}\end{matrix}\))\\
&=\mathbf{f}((0,m,0)(\l,\nu\T{q},\lambda s \transpose{\lambda})_{\a} (\l,0,0)_{\h}(0,\nu\T{q},\kappa)_{\h}\(\begin{matrix} q & qx\\ & \tilde{q}\end{matrix}\))\\
&=\mathbf{f}((\l,\nu\T{q},\lambda s \transpose{\lambda})_{\a}(\l,0,0)_{\h}\diag[q,\tilde{q}](0,\nu ,\kappa)_{\h}\(\begin{matrix} 1_n & x\\ & 1_n\end{matrix}\)_{\a})\\
&=\psi_S(\kappa_{\h})\(f_{\b{p}}|_{k,S} (\l,\nu\T{q},\lambda s \transpose{\lambda})_{\a}\(\begin{matrix} q & qx\\ & \tilde{q}\end{matrix}\)_{\a}\) (\b{i}_0),
\end{align*}
where we take $\kappa:=\lambda s \transpose{\lambda} - (\l q\T{\nu} +\nu\T{q}\T{\l}), \b{p}:=(\l,0,0)_{\h}\diag[q,\tilde{q}]_{\h}$ and $f_{\b{p}}$ is as in Definition \ref{definition automorphic Jacobi}. 

Since $f_{\b{p}}\in M_{k,S}^n(\b{G}(F) \cap \b{p} K[\mathfrak{b},\mathfrak{c}] \b{p}^{-1}, \chi)$, it is invariant under the translations $\tau\mapsto\tau +b$ and $w\mapsto w+\mu$ for every $b\in\mathcal{L}:=Sym_n(F) \cap q_{\h} X\T{q}_{\h}$ and $\mu\in\mathcal{L}_{l,n}:=M_{l,n}(F) \cap (X_{l,n}\T{q}_{\h})$. Indeed, for each such $b$ and $\mu$ the finite parts of the adelic elements 
\[
(0,0,\lambda b \transpose{\lambda}) \left(\begin{matrix} 1 & b  \\ 0 & 1 \end{matrix}\right) =
(\lambda , 0 ,0) \diag[q, \tilde{q}] (0, - \lambda b \tilde{q}, 0)   \left(\begin{matrix} 1 & q^{-1} b \tilde{q} \\ 0 & 1 \end{matrix}\right) \diag[q^{-1}, \transpose{q}] (-\lambda, 0,0) 
\]
and
\[
(0,\mu, \lambda q \mu \tilde{q} + \mu \transpose{\lambda}) = 
(\lambda , 0 ,0) \diag[q, \tilde{q}] (0, \mu \tilde{q}, 0)  \diag[q^{-1}, \transpose{q}] (-\lambda, 0,0) 
\]
are in the finite part of the group $\b{p} K[\mathfrak{b},\mathfrak{c}] \b{p}^{-1}$. Hence, $f_{\b{p}}$ has a Fourier expansion 
$$f_{\b{p}}(\tau ,w)=\sum_{\substack{t\in L\\ t\geq 0}}\sum_{r\in M} c(\b{p};t,r)\mathbf{e}_{\a}(\tr (t\tau+\T{r}w)),$$
where 
$$L=\{ x\in Sym_n(F):\mathbf{e}_{\a}(\tr(x\mathcal{L}))=1\} ,$$
$$M=\{ x\in M_{l,n}(F):\mathbf{e}_{\a}(\tr(\T{x}\mathcal{L}_{l,n}))=1\} .$$
In particular, $c(\b{p};t,r)\neq 0$ only if at every $v\in\h$ and for every $x\in X_v, x_{l,n}\in (X_{l,n})_v$ we have $e(\tr (\T{q_v}t_vq_vx))=1$ and $e(\tr (\T{q_v}\T{r_v}(x_{l,n})))=1$. Further, if we put $\b{r}:=(\l,\nu\T{q},\lambda s \transpose{\lambda})_{\a}\(\begin{smallmatrix} q & qx\\ & \tilde{q}\end{smallmatrix}\)_{\a}$, 
we have 
\begin{align*}
\mathbf{f}&((\l,\mu,0)\(\begin{matrix} q & \sigma\tilde{q}\\ & \tilde{q}\end{matrix}\)) =\psi_S(\kappa_{\mathbf{h}})J_{k,S}(\b{r},\b{i}_0)^{-1}f_{\b{p}}(\b{r}\b{i}_0)\\
&=\Psi_{\h}(\tr(S\kappa))\mathbf{e}_{\a}(\tr(S[\l]s) + \tr(S[\l](iq\T{q}+qx\T{q})))(\det q)_{\a}^k\\
&\hspace{0.4cm}\cdot f_{\b{p}}(iq\T{q}+qx\T{q},i\lambda q\T{q}+\lambda qx\T{q}+\nu\T{q})\\
&=\Psi_{\h}(\tr(S\kappa))\mathbf{e}_{\a}(\tr(S[\l](iq\T{q}+\sigma)))(\det q)_{\a}^kf_{\b{p}}(iq\T{q}+qx\T{q},i\lambda q\T{q}+\lambda qx\T{q}+\nu\T{q}),
\end{align*}
Now note that 
\begin{align*}
\Psi_{\h}(\tr(S\kappa)) &= \Psi_{\h}(\tr(S (\lambda s \transpose{\lambda} - (\l q\T{\nu} +\nu\T{q}\T{\l}))))= 
\Psi_{\h}(\tr(S (\lambda s \transpose{\lambda}))\\
&= \Psi_{\h}(\tr(S (\lambda \sigma \transpose{\lambda})) \Psi_{\h}(-\tr(S (\lambda q x \transpose{q} \transpose{\lambda})) =   \Psi_{\h}(\tr(S (\lambda \sigma \transpose{\lambda}))). 
\end{align*}

Moreover, since $\mathbf{e}_{\h}(\tr(tqx\T{q}))=1=\mathbf{e}_{\h}(\T{r}\lambda qx\T{q}+\T{r}\nu\T{q}))$ for $t\in L, r\in M$, we have
$$\mathbf{e}_{\A}(\tr(t\sigma))=\mathbf{e}_{\A}(\tr(ts+tqx\T{q}))=\mathbf{e}_{\A}(\tr(tqx\T{q}))=\mathbf{e}_{\a}(\tr(tqx\T{q}))$$ and
$$\mathbf{e}_{\A}(\tr(\T{r}(\lambda\sigma +\mu)))=\mathbf{e}_{\A}(\tr(\T{r}(m+\nu\T{q})+\T{r}\lambda qx\T{q}))=\mathbf{e}_{\a}(\tr(\T{r}\nu\T{q}+\T{r}\lambda qx\T{q})).$$
Hence,
$$f_{\b{p}}(\b{r}\b{i}_0)=\sum_{\substack{t\in L\\ t\geq 0}}\sum_{r\in M} c(\b{p};t,r) \mathbf{e}_{\a}(i\tr(tq\T{q}+\T{r}\lambda q\T{q})) \mathbf{e}_{\A}(\tr(t\sigma))\mathbf{e}_{\A}(\tr(\T{r}\lambda\sigma +\T{r}\mu)).$$
In this way we obtain Fourier expansion \eqref{eq:adelic_Fourier_expansion} that satisfies properties $(1)$ and $(3)$. The second property follows from the fact that $\mathbf{f}|_{k,S}\diag[a,\tilde{a}]=\chi(\det a)^{-1}\mathbf{f}$ for $a\in\GL_n(F)$.
\end{proof}

\section{Jacobi Eisenstein series} \label{Jacobi Eisenstein series}

In this section we introduce Klingen-type Jacobi Eisenstein series. We do this both from a classical and adelic point of view, and also explore the relation between the two in the spirit of the bijection \eqref{bijection between classical and adelic} between classical and adelic Siegel-Jacobi forms, which was established in the previous section. First systematic study of Eisenstein series from a classical point of view was undertaken by Ziegler in \cite{Z89}. Our contribution here is to extend his results to include non-trivial level, non-trivial nebentype and we also work over a general totally real field. Furthermore, we introduce the adelic point of view, which, to the best of our knowledge, a systematic study of which, has not appeared before in the literature in the Jacobi setting. \newline

For an integer $r\in\{ 0,1,\ldots ,n\}$, we let $P^{n,r}, \b{P}^{n,r}$ be Klingen parabolic subgroups of $G^n$ and $\b{G}^{n,l}$ respectively, as defined in Section \ref{sec:notation}. We define the map $\lambda^n_{r,l} : \b{G}^{n,l} \rightarrow F$ by
\[
\lambda^n_{r,l}((\lambda,\mu,\kappa)g) := \lambda_r^n(g),
\]
where $\lambda^{n}_{r}: \Sp_{n} \rightarrow F$ is the map defined as in \cite{Sh95} by
\[
\lambda^{n}_r \left(\begin{pmatrix} a_1 & a_2 & b_1 & b_2 \\
a_3 & a_4 & b_3 & b_4 \\  
c_1 & c_2 & d_1 & d_2 \\
c_3 & c_4 & d_3 & d_4
\end{pmatrix}\right) = \det(d_4),
\]
where the matrices $a_1,b_1,c_1,d_1$ are of size $r$ and the matrices $a_4,b_4,c_4,d_4$ of size $n-r$; we set $\lambda^{n}_n(g):=1$. We extend this map to the adeles so that $\lambda^{n}_{r,l}: \b{G}^{n,l}(\mathbb{A}) \rightarrow \mathbb{A}$.

Furthermore for $r>0$ we define the map 
$$\b{\omega}_r: \mathcal{H}_{n,l} \rightarrow \mathcal{H}_{r,l}$$
by $\b{\omega}_r(\tau,w) := (\tau_1,w_1)$, where $\tau_1$ denotes the $r \times r$ upper left corner of the matrix $\tau$ and $w_1$ is the $l \times r$ matrix obtained from the first $r$ columns of $w$. Note that $\tau_1 = \omega_r(\tau)$ for $\omega_r$ as in \cite{Sh95}; we extend this and write $\b{\omega}_r(w) := w_1$.

Finally, we define a (set theoretic) map 
$$\b{\pi}_r : H^{n,l} \times M_{2n} \rightarrow H^{r,l} \times M_{2r} , \quad\b{\pi}_r((\lambda,\mu,\kappa),g) := ((\lambda_1, \mu_1 ,\kappa), \pi_r(g)),$$ 
where $\lambda_1$ (resp $\mu_1$) is the $l \times r$ matrix obtained by taking the first $r$ columns of $\lambda$ (resp. $\mu$), and 
$\pi_r(g):= \(\begin{smallmatrix}
a_1(g) & b_1(g) \\ c_1(g) & d_1(g)\end{smallmatrix}\)$ 
is the map defined in \cite{Sh95} with $\pi_0(g):=1$. 

As we pointed out above, the maps $\lambda^{n}_{r}, \b{\omega}_r, \b{\pi}_r$ generalize the maps defined in \cite{Sh95}. In a similar manner their properties generalize the ones of the symplectic setting.

\begin{lem}
	Assume $r>0$. Then for all $\b{g} \in \b{P}^{n,r}(\mathbb{A})$we have
	\begin{equation}\label{eq:omega_r}
	\b{\omega}_r(\b{g} z) = \b{\pi}_r(\b{g})\b{\omega}_r(z)
	\end{equation}   
	and
	\begin{equation}\label{eq:J_at_P}
	J_{k,S}(\b{g},z) = (\lambda^n_{r,l}(\b{g})_{\mathbf{a}})^k J_{k,S}(\b{\pi}_r(\b{g}), \b{\omega}_r(z)).
	\end{equation}
\end{lem}
\begin{proof} 
	Write $z= (\tau,w)$ and $\b{g} = hg = (\lambda,\mu,\kappa)g$. Then, by \cite[(1.24)]{Sh95}, $\omega_r(g\tau) = \pi_r(g) \omega_r(\tau)$ and $j(g,\tau) = \lambda_r(g)_{\mathbf{a}} j(\pi_r(g),\omega_r(\tau))$.  Thus, to show \eqref{eq:omega_r} it suffices to establish the equality
	\[
	(w (c_g \tau + d_g)^{-1} + \lambda g \tau + \mu)_1 = w_1 (c_{\pi_r(g)} \omega_r(\tau) + d_{\pi_r(g)})^{-1} + \lambda_1 \pi_r(g) \omega_r(\tau) + \mu_1; 
	\] 
	or, after using the fact that $\pi_r(g)\omega_r(\tau) = \omega_r(g\tau)$ for $g\in P^{n,r}$, 
	\[
	(w (c_g \tau + d_g)^{-1})_1 = w_1 (c_{\pi_r(g)} \omega_r(\tau) + d_{\pi_r(g)})^{-1},\qquad (\lambda g \tau)_1 = \lambda_1 \omega_r(g \tau).
	\]
	
	Set $c:=c_g$, $d:=d_g$ and observe that for $\b{g} \in \b{P}^{n,r}(\mathbb{A})$,
	\[
	c \tau + d = \begin{pmatrix} c_1 & 0 \\ 0 & 0 \end{pmatrix} \begin{pmatrix}
	\tau_1 & \tau_2 \\ \transpose{\tau_2} & \tau_4
	\end{pmatrix} + \begin{pmatrix} d_1 & d_2 \\ 0 & d_4 \end{pmatrix} = \begin{pmatrix}
	c_1 \tau_1 + d_1 & * \\ 0 & d_4
	\end{pmatrix},
	\]
	where $c_1,\tau_1,d_1$ are $r \times r$ matrices. In particular,
	\[
	(c \tau + d)^{-1} = \begin{pmatrix}
	(c_1 \tau_1 + d_1)^{-1} & * \\ 0 & d_4^{-1}
	\end{pmatrix} ,
	\]
	and thus
	\begin{align*}
	(w (c \tau + d)^{-1})_1 &= ((w_1 \,w_2) 
	\begin{pmatrix}	(c_1 \tau_1 + d_1)^{-1} & * \\ 0 & d_4^{-1}	\end{pmatrix})_1 
	= (w_1 (c_1 \tau_1 + d_1)^{-1} \,\, *)_1\\
	&= w_1 (c_1 \tau_1 + d_1)^{-1}=w_1 (c_{\pi_r(g)} \tau_1 + d_{\pi_r(g)})^{-1}.
	\end{align*}
	
	Similarly,		 
	$$\lambda g \tau = (\lambda_1 \,\,0) \begin{pmatrix} \omega_r(g \tau) & * \\ * & * \end{pmatrix} = (\lambda_1 \omega_r(g \tau)\,\, *).$$
	
	We will now prove the equality \eqref{eq:J_at_P}. Because $\lambda_{r,l}^n(\mathbf{g})_{\mathbf{a}} = \lambda_r(g)_{\mathbf{a}}$ and $j(g,\tau) = \lambda_r(g)_{\b{a}} j(\pi_r(g), \omega_r(\tau))$, it is enough to show that
	\[
	\mathcal{J}_S(\b{g}, z) = \mathcal{J}_S(\b{\pi}_r(\b{g}),\b{\omega}_r(z)),
	\]
	that is,
	\begin{enumerate}
		\item $\tr(S[w](c_g \tau+d_g)^{-1} c_g) = \tr(S[w_1] (c_{\pi_r(g)} \tau_1 + d_{\pi_r(g)})^{-1} c_{\pi_r(g)})$
		\item $\tr(\T{\lambda}Sw (c_g \tau+d_g)^{-1}) = \tr(\T{\lambda_1}Sw_1 (c_{\pi_r(g)} \tau_1 + d_{\pi_r(g)})^{-1})$ and
		\item $\tr(S[\lambda] g \tau) = \tr(S[\lambda_1] \pi_r(g) \tau_1)$.
	\end{enumerate}
	
	Write $w = (w_1 \,\,w_2)$, so that 
	$$S[w] = \(\begin{smallmatrix}
	\T{w_1}\\ \T{w_2}\end{smallmatrix}\) S (w_1 \,\,w_2) = \(\begin{smallmatrix}
	\T{w_1} S\\ \T{w_2}S\end{smallmatrix}\) (w_1 \,\,w_2) = \begin{pmatrix}
	S[w_1] & * \\ * & * \end{pmatrix}.$$
	Moreover, as we have seen before, $(c_g \tau+d_g)^{-1} = \(\begin{smallmatrix}
	(c_{\pi_r(g)} \omega_r(\tau) + d_{\pi_r(g)})^{-1} & * \\ 0 & *
	\end{smallmatrix}\)$, $c = \(\begin{smallmatrix}
	c_{\pi_r(g)}  & 0 \\ 0 & 0
	\end{smallmatrix}\)$, so that 
	\[
	(c_g \tau+d_g)^{-1} c_g =
	\begin{pmatrix}
	(c_{\pi_r(g)} \omega_r(\tau) + d_{\pi_r(g)})^{-1}c_{\pi_r(g)}  & 0 \\ 0 & 0
	\end{pmatrix}.
	\] 
	Hence 
	\begin{align*}
	\tr(S[w](c_g \tau+d_g)^{-1} c_g) &= \tr\left(\begin{pmatrix}
	S[w_1] & * \\ * & * 
	\end{pmatrix} \begin{pmatrix}
	(c_{\pi_r(g)} \tau_1 + d_{\pi_r(g)})^{-1}c_{\pi_r(g)}  & 0 \\ 0 & 0
	\end{pmatrix} \right)\\ 
	&=\tr (S[w_1] (c_{\pi_r(g)} \tau_1 + d_{\pi_r(g)})^{-1} c_{\pi_r(g)}).
	\end{align*}
	
	Now write $\lambda = (\lambda_1 \,\,0)$. Then
	\begin{align*}
	\T{\lambda}Sw(c_g \tau+d_g)^{-1} &= \begin{pmatrix} S(\lambda_1,w_1) & * \\ 0 & 0 \end{pmatrix} \begin{pmatrix}
	(c_{\pi_r(g)} \omega_r(\tau) + d_{\pi_r(g)})^{-1} & * \\ 0 & *
	\end{pmatrix}\\
	&=\begin{pmatrix}
	S(\lambda_1,w_1)(c_{\pi_r(g)} \omega_r(\tau) + d_{\pi_r(g)})^{-1} & * \\ 0 & 0
	\end{pmatrix}.
	\end{align*}  
	In particular,  
	$$\tr(\T{\lambda}Sw (c_g \tau+d_g)^{-1}) = \tr(\T{\lambda_1}Sw_1 (c_{\pi_r(g)} \tau_1 + d_{\pi_r(g)})^{-1}).$$
	
	For the final equation it is enough to observe that 
	$S[\lambda] = \begin{pmatrix} S[\lambda_1] & 0 \\ 0 & 0 \end{pmatrix}$ and so 
	\[
	S[\lambda] g \tau = \begin{pmatrix} S[\lambda_1] & 0 \\ 0 & 0 \end{pmatrix} \begin{pmatrix} (g \tau)_1 & (g\tau)_2 \\ (g \tau)_3 & (g \tau)_4 \end{pmatrix} = \begin{pmatrix} S[\lambda_1] (g \tau)_1 & * \\ 0 & 0 \end{pmatrix}. 
	\]
	But $(g \tau)_1 = \omega_r(g \tau) = \pi_r(g) \omega_r(\tau)$ and hence 
	\[
	\tr(S[\lambda] g \tau) = \tr(S[\lambda_1] \pi_r(g) \tau_1).
	\]
\end{proof}
\subsection{Adelic Jacobi Eisenstein series of Klingen-type}\label{sec:adelic_ES}
We are now ready to define adelic Jacobi Eisenstein series of Klingen type. Fix a weight $k \in \Z^{\a}$ and consider a Hecke character $\chi$ such that for a fixed integral ideal $\f{c}$ of $F$ we have
\begin{enumerate}
	\item  $\chi_v(x) = 1$ for all $x \in \f{o}_v^{\times}$ with $x-1 \in \f{c}_v$, $v\in\h$,
	\item $\chi_{\a}(x_{\a}) = \sgn (x_{\a})^k:= \prod_{v \in \mathbf{a}} \left(\frac{x_v}{|x_v|}\right)^{k_v}$, for $x_{\a} \in \mathbb{A}_{\a}$;
\end{enumerate}
we will also write $\chi_{\f{c}}:=\prod_{v|\f{c}}\chi_v$. We fix a fractional ideal $\f{b}$ and an integral ideal $\f{e}$ such that $\f{c}\subset\f{e}$ and $\f{e}$ is prime to $\f{e}^{-1}\f{c}$. Further, for $r\in\{ 1,\ldots ,n\}$ we set
$$K:= K_{\h}[\f{b},\f{c}] (H_{\mathbf{a}}^{n,l} \rtimes D_{\infty}^{\a}),$$
\begin{multline*}
K^{n,r}:= \{ \b{x} = (\lambda,\mu,\kappa) x \in K: (a_1(x)-1_r)_v\in M_{r,r}(\f{e}_v),\\ (a_2(x))_v\in M_{r,n-r}(\f{e}_v), (b_1(x))_v\in M_{r,r}(\f{b}_v^{-1}\f{e}_v)\mbox{ for every } v|\f{e}\},
\end{multline*}
where $x = \spmatrix{a_x & b_x \\ c_x & d_x } =
\spmatrix{a_1(x) & a_2(x) & b_1(x) & b_2(x) \\
a_3(x) & a_4(x) & b_3(x) & b_4(x) \\
c_1(x) & c_2(x) & d_1(x) & d_2(x) \\
c_3(x) & c_4(x) & d_3(x) & d_4(x) }$, and  
$$K^r:= \{ \b{x}\in K^r[\f{b}^{-1}\f{e},\f{bc}]: (a_x-1_r)_v\in M_{r,r}(\f{e}_v)\mbox{ for every } v|\f{e}\}.$$
If $r=0$, we put $K^{n,0}:=K$.

For a cusp form $\mathbf{f} \in \mathcal{S}_{k,S}^r (K^r,\chi^{-1})$, $\mathbf{f}:=1$ if $r=0$, we define a $\mathbb{C}$-valued function $\phi(x,s;\mathbf{f})$ with $x \in \b{G}^n(\mathbb{A})$ and $s \in \mathbb{C}$ as follows. We set $\phi(x,s;\mathbf{f}) := 0$ if $ x \notin \b{P}^{n,r}(\mathbb{A})K^{n,r}$ and otherwise, if $x = \b{p}\b{w}$ with $\b{p} \in \b{P}^{n,r}(\mathbb{A})$ and $\b{w} \in K^{n,r}$, we set
$$
\phi(x,s;\mathbf{f}):= \chi(\lambda_{r,l}^n(\b{p}))^{-1} \chi_{\mathfrak{c}}(\det(d_w)))^{-1} J_{k,S}(\b{w},\mathbf{i}_0)^{-1} \mathbf{f}(\pi_r(\b{p})) |\lambda^n_{r,l}(\b{p})|_{\A}^{-2s},  
$$
where $\b{w} = h w$ with $w \in \Sp_n(\A)$. We recall here that if we write $p$ for the symplectic part of $\b{p}$ then $\lambda_{r,l}^n(\b{p})= \lambda^n_r(p)$. Moreover, since at archimedean places $x_{\a}\in \b{P}^{n,r}_{\a} K^{n,r}_{\a}=P^{n,r}_{\a}K^{n,r}_{\a}$ if and only if $x_{\a}\in P'_{\a}K^{n,r}_{\a}$ , where $P':=\bigcap_{r=0}^{n-1} P^{n,r}$ (\cite{Sh95}, Lemma 3.1), we always choose $\b{p} \in \b{P}^{n,r}(\mathbb{A})$ so that $\b{p}_{\a}=p_{\a} \in P'_{\a}$. We now check that $\phi(x,s;\mathbf{f})$ is well-defined, i.e. that it is independent of the choice of $p$ and $\b{w}$. 

Let $x=\b{p}_1 \b{w}_1=\b{p}_2 \b{w}_2$, set $\b{r}:= \b{p}_2^{-1} \b{p}_1 = \b{w}_2 \b{w}_1^{-1} \in \b{P}^{n,r}(\A) \cap K^{n,r}$ and assume that $(p_1)_{\a}, (p_2)_{\a}\in P'_{\a}$. Observe that  $\l_{r,l}^n(\b{r})_v=(\det d_{p_2,4})_v^{-1}(\det d_{p_1,4})_v\in\f{o}\ti_v$ for every $v\in\h$, and $|\l_{r,l}^n(\b{r})_v|_v=1$ for all $v\in\a$. Hence, $|\lambda^n_{r,l}(\b{p})|_{\A}^{-2s}$ is independent of choice of $\b{p}$ and $\b{w}$, and $\chi(\lambda^n_{r,l}(\b{p}))^{-1}=\chi_{\f{c}}(\lambda^n_{r,l}(\b{p}))^{-1}(\lambda^n_{r,l}(\b{p})_{\a})^{-k}$. Because
$$\mathbf{f}(\pi_r(\b{p}_1))=\mathbf{f}(\pi_r(\b{p}_2\b{r}))=\mathbf{f}(\pi_r(\b{p}_2)\pi_r(\b{r}))
=\mathbf{f}(\pi_r(\b{p}_2))\chi_{\f{c}}(\det a_{\pi_r(\b{r})})J_{k,S}(\pi_r(\b{r}),\b{i}_0)^{-1},$$
we have to prove that
\begin{multline*}
\chi_{\f{c}}(\lambda^n_{r,l}(\b{r}))^{-1}(\lambda^n_{r,l}(\b{r})_{\a})^{-k} \chi_{\mathfrak{c}}(\det(d_{w_1}))^{-1}\chi_{\mathfrak{c}}(\det(d_{w_2}))\chi_{\f{c}}(\det a_{\pi_r(\b{r})})\\
= J_{k,S}(\pi(\b{r}),\b{i}_0) J_{k,S}(\b{w}_1,\mathbf{i}_0) J_{k,S}(\b{w}_2,\mathbf{i}_0)^{-1} 
\end{multline*}
First of all, since $\b{r}_{\a}\in P_{\a}'$,
$$(\lambda^n_{r,l}(\b{r})_{\a})^k J_{k,S}(\pi(\b{r}),\b{i}_0) J_{k,S}(\b{w}_1,\mathbf{i}_0) J_{k,S}(\b{w}_2,\mathbf{i}_0)^{-1}
=J_{k,S}(\b{r},\b{i}_0)J_{k,S}(\b{r},\b{w}_1\cdot\mathbf{i}_0)^{-1}=1.$$
Moreover, it is easy to check that
\begin{multline*}
\chi_{\f{c}}(\lambda^n_{r,l}(\b{r}))^{-1}\chi_{\mathfrak{c}}(\det(d_{w_1}))^{-1}\chi_{\mathfrak{c}}(\det(d_{w_2}))\chi_{\f{c}}(\det a_{\pi_r(\b{r})})\\
=\chi_{\f{c}}(\det d_{\pi_r(w_2)})\chi_{\f{c}}(\det d_{\pi_r(w_1)})^{-1}\chi_{\f{c}}(\det a_{\pi_r(\b{r})})=1.
\end{multline*}
This proves the statement above.

We define the Eisenstein series of Klingen type by
\begin{equation}\label{def:KlingenES}
E(x,s;\mathbf{f}):=E(x,s;\mathbf{f},\chi, K^{n,r}) := \sum_{\gamma \in\b{P}^{n,r}(F) \setminus \b{G}^n(F)} \phi(\gamma x,s;\mathbf{f}),\quad\Re(s) \gg 0.
\end{equation}
If $r=0$ and $\mathbf{f}=1$, then we say that $E(x,s):=E(x,s;1)$ is an Eisenstein series of Siegel type.

It is clear from the above calculations that this is well defined, and for $\gamma \in \b{P}^{n,r}(F)$, $\b{w}\in K^{n,r}_{\h}\times K_{\infty}$,
$$\phi(\gamma x \b{w},s;\mathbf{f}) = \chi_{\mathfrak{c}}(\det(d_w))^{-1} J_{k,S}(\b{w},\mathbf{i}_0)^{-1} \phi(x,s;\mathbf{f}).$$
In particular, for  $\kappa \in Sym_{l}(\A)$, $\gamma \in \b{G}^n(F)$, $x \in \b{G}^n(\A)$ and $\b{w}\in K^{n,r}_{\h}\times K_{\infty}$,
$$E((0,0,\kappa) \gamma x \b{w}, s; \mathbf{f}) = \psi_S(\kappa)\chi_{\mathfrak{c}}(\det(d_w))^{-1} J_{k,S}(\b{w},\mathbf{i}_0)^{-1} E(x,s;\mathbf{f}).$$

We will show in Proposition \ref{absolute convergence of Eisenstein series} below that the series above, evaluated at  $s= k/2$ for $k\in\Z$, $k > n+r+l+1$, is absolutely convergent and hence defines an adelic Siegel-Jacobi modular form of parallel weight $k \mathbf{a}:=(k,k,\ldots,k) \in \mathbb{Z}^{\mathbf{a}}$. \newline

We now investigate the relation of the adelic Eisenstein series \eqref{def:KlingenES} with the classical one. 

Write $K^{n,r}_{\mathbf{h}} = C_{\mathbf{h}}[\mathfrak{o},\mathfrak{b}^{-1},\mathfrak{b}^{-1}] \rtimes D^{n,r}_{\mathbf{h}}[\mathfrak{b}^{-1},\mathfrak{bc}]$. Then it follows from \cite[Lemma 3.2]{Sh95} and \cite[Lemma 1.3]{Sh83} that 
$$P^{n,r}(\mathbb{A}) = \bigsqcup_{x \in X} P^{n,r}(F) x (P^{n,r}(\mathbb{A}) \cap D^{n,r}_{\mathbf{h}}[\mathfrak{b}^{-1},\mathfrak{bc}]) P^{n,r}(\mathbb{A}_{\mathbf{a}}),$$  
where $X$ is a finite subset of $P^{n,r}(\mathbb{A})$ such that $\{\mathfrak{a}_r(x):x \in X\}$ forms a set of representatives for the ideal class group of $F$, where $\mathfrak{a}_r(x)$ is the ideal of $F$ defined in \cite[page 551]{Sh95} as the ideal corresponding to the idele $\lambda_r(x)$. In particular one may pick $x$'s of a very specific form, namely $\diag[1_{n-1},t^{-1},1_{n-1},t]$ with $t \in \mathbb{A}^{\times}_{\mathbf{h}}$.  Since $\b{P}^{n,r} = H_r^{n,l} \rtimes P^{n,r}$ and the strong approximation holds for $H_r^{n,l}$ by the same argument as in Lemma \ref{strong_approx_Jac}, we have that 
\[
\b{P}^{n,r}(\mathbb{A}) = \bigsqcup_{x' \in X'} \b{P}^{n,r}(F) x' (\b{P}^{n,r}(\mathbb{A}) \cap K^{n,r}_{\h}[\f{b},\f{c}]) \b{P}^{n,r}(\mathbb{A}_{\mathbf{a}}),
\]    
where $X'$ is the set $X$ extended trivially to $\b{G}^n$ by the canonical embedding $\Sp_n \hookrightarrow \Jac^n$. We can now establish that 
\begin{align*}
\b{P}^{n,r}(\mathbb{A}) K^{n,r} &= \bigsqcup_{x' \in X'} \b{P}^{n,r}(F) x' K^{n,r}_{\h}[\f{b},\f{c}] \b{P}^{n,r}(\mathbb{A}_{\mathbf{a}}) K^{n,r}(\mathbb{A}_{\mathbf{a}})\\
&=\bigsqcup_{x' \in X'} \b{P}^{n,r}(F) x' K^{n,r}_{\h}[\f{b},\f{c}] \b{G}^n(\mathbb{A}_{\mathbf{a}}).
\end{align*}

Indeed, we only need to establish that the union is disjoint. Assume that the cosets determined by $x_1,x_2 \in X'$ are not disjoint, that is $x_1 = a x_2 b c$ for some $a\in \b{P}^{n,r}(F)$, $b \in K_{\mathbf{h}}[\mathfrak{b},\mathfrak{c}]$ and $c \in  \b{P}^{n,r}(\mathbb{A}_{\mathbf{a}})K^{n,r}(\mathbb{A}_{\mathbf{a}})$. Since $x_1,x_2 \in \Jac^n_{\mathbf{h}}$, we have that $x_1 = a_{\h} x_2 b$. Moreover, since $a \in \b{P}^{n,r}(F)$ and $x_1,x_2$ are diagonal, $b\in \b{P}^{n,r}(\mathbb{A}) \cap K^{n,r}_{\h}[\f{b},\f{c}]$ and $c_{\mathbf{a}} \in \b{P}^{n,r}(\mathbb{R})$. But then this implies that $x_1 \in \b{P}^{n,r}(F) x_2 (\b{P}^{n,r}(\mathbb{A}) \cap K^{n,r}_{\h}[\f{b},\f{c}]) \b{P}^{n,r}(\mathbb{A}_{\mathbf{a}})$, and thus $x_1=x_2$. 

Take the set $X'$ to be of the particular form indicated above, that is let $x'\in X'$ be of the form $\diag[1_{n-1},t^{-1},1_{n-1},t] \in \Sp_{n}(\A) \hookrightarrow \Jac^n(\A)$ with $t \in \mathbb{A}^{\times}_{\mathbf{h}}$. Observe that for any such $x'$, $x' K^{n,r}_{\mathbf{h}}[\mathfrak{b},\mathfrak{c}] (\mathbb{A}_{\mathbf{a}})\b{G}^n(\mathbb{A}\mathbf{a}) \cap \Jac^n(F) \neq \emptyset$. Indeed, this follows from the fact that $\diag[1_{n-1},t^{-1},1_{n-1},t] D^{n,r}_{\h}[\mathfrak{b}^{-1},\mathfrak{bc}]\Sp_n(\mathbb{R}) \cap \Sp_n(F) \neq \emptyset$. In particular, we can conclude the analogue of \cite[Lemma 3.3]{Sh95} in the Jacobi setting:

\begin{lem}\label{Parabolic decomposition} Set $Y := \bigcup_{t \in \mathbb{A}_\mathbf{h}^{\times}} \diag[1_{n-1},t^{-1},1_{n-1},t]  K_{\h}[\f{b},\f{c}] \b{P}^{n,r}(\mathbb{A}_{\mathbf{a}})K^{n,r}(\mathbb{A}_{\mathbf{a}})$. Then there exists a subset $Z$ of $\Jac^n(F)\cap Y$ such that 
$$\b{P}^{n,r}(\mathbb{A}) K^{n,r} = \bigsqcup_{\zeta \in Z} \b{P}^{n,r}(F) \zeta K^{n,r}_{\h}[\f{b},\f{c}]\b{P}^{n,r}(\R)K^{n,r}(\mathbb{A}_{\mathbf{a}})	= \bigsqcup_{\zeta \in Z} \b{P}^{n,r}(F) \zeta K^{n,r}_{\h}[\f{b},\f{c}]\Jac^n(\R)$$
and
\begin{align*}
	\Jac^n(F)\cap \b{P}^{n,r}(\mathbb{A}) K^{n,r}&= \bigsqcup_{\zeta \in Z} \b{P}^{n,r}(F) \zeta \left( K^{n,r}_{\h}[\f{b},\f{c}]\b{P}^{n,r}(\mathbb{A}_{\mathbf{a}})K(\mathbb{A}_{\mathbf{a}}) \cap \Jac^n(F) \right)\\  
	&= \bigsqcup_{\zeta \in Z} \b{P}^{n,r}(F) \zeta \left( K^{n,r}_{\h}[\f{b},\f{c}]\Jac^n(\mathbb{A}_{\mathbf{a}}) \cap \Jac^n(F) \right) .
\end{align*}
\end{lem} 
\subsection{Classical Jacobi Eisenstein series of Klingen-type} 
We now associate a Siegel-Jacobi modular form to an adelic Eisenstein series defined in \eqref{def:KlingenES}. We set $\Gamma := \Jac^n(F)\cap K^{n,r}_{\mathbf{h}}[\mathfrak{b},\mathfrak{c}]\Jac^n(\mathbb{A}_{\mathbf{a}})$, and with $Z$ as in Lemma \ref{Parabolic decomposition} we define $R_{\zeta} := (\b{P}^{n,r}(F) \cap \zeta \Gamma \zeta^{-1}) \setminus \zeta \Gamma$, for $\zeta \in Z$. Then, again by the same lemma, it follows that a set of representatives for $\b{P}^{n,r}(F) \setminus (\Jac^n(F) \cap \b{P}^{n,r}(\mathbb{A}) K^{n,r})$ is given by $R:=\bigcup_{\zeta \in Z}R_{\zeta}$. In particular, we may write
$$E(x,s;\mathbf{f}) = \sum_{\gamma \in R} \phi(\gamma x,s; \mathbf{f}).$$

For any given $z \in \mathcal{H}_{n,l}$ there is an $y \in\b{G}^n_{\a}$ such that $y \cdot\mathbf{i}_0 = z$. Moreover, we can always pick $y$ such that the symmetric matrix in the Heisenberg part of $y$ is zero, i.e. $\kappa_y = 0$. A Siegel-Jacobi modular form that corresponds to $E(x,s;\mathbf{f})$ via the bijection \eqref{bijection between classical and adelic} with $\b{g} = 1$ is the Eisenstein series,

$$E(z,s;\mathbf{f}) = J_{k,S}(y,\mathbf{i}_0) \sum_{\gamma \in R} \phi(\gamma y,s;\mathbf{f}).$$
We will write it down in terms of $f$ and $z$ using the bijection \eqref{bijection between classical and adelic} again.
For some $\zeta \in Z$ and $\gamma \in R_{\zeta}$ we may write $\gamma y = \tau\b{w}$, where $\tau_{\mathbf{h}} = \diag[1_{n-1},t^{-1},1_{n-1},t]$ as in Lemma \ref{Parabolic decomposition}, $\tau_{\mathbf{a}} \in \cap_{r=0}^{n-1} P^{n,r}_{\mathbf{a}}$ and $\b{w} \in K^{n,r}$. This is because $H^{n,l}_{\a}\subset K^{n,r}_{\a}$ and, by \cite[Lemma 3.1]{Sh95}, $G^n(\A) =\cap_{r=0}^{n-1} P^{n,r}(\A)D_{\infty}^{\a}D_{\h}[\f{b}^{-1},\f{b}]$. Therefore 
\[
\phi(\tau\b{w},s;\mathbf{f}) = \chi_{\h}(t)^{-1} \chi_{\a}(\lambda_{r,l}^n(\tau)_{\a})^{-1} \chi_{\f{c}}(\det(d_w))^{-1} J_{k,S}(\b{w},\mathbf{i}_0)^{-1}\mathbf{f}(\pi_r(\tau_{\a})) |\lambda_{r,l}^n(\tau)|_{\A}^{-2s}.
\]

Observe further that, in case $r>0$,
\begin{enumerate}
	\item $\mathbf{f}(\b{\pi}_r(\tau_{\a}))=J_{k,S}(\b{\pi}_r(\tau_{\a}),\b{i}_0)^{-1} f(\b{\pi}_r(\tau_{\a}))\stackrel{\eqref{eq:omega_r},\eqref{eq:J_at_P}}{=} J_{k,S}(\tau_{\a},\b{i}_0)^{-1}(\lambda_{r,l}^n(\tau)_{\a})^k f(\b{\omega}_r(\gamma z))$;
	\item $|\lambda_{r,l}^n(\tau)_{\a}|_F=|\frac{j(\tau_{\a},\b{i})}{j(\pi_r(\tau_{\a}),\omega_r(\b{i}))}|_F=\(\frac{\delta(\pi_r(\tau_{\a}),\b{i})}{ \delta(\tau_{\a},\b{i})}\)^{1/2}=\(\frac{\delta(\b{\omega}_r(\gamma z))}{ \delta(\gamma z)}\)^{1/2}$;
	\item $|\lambda_{r,l}^n(\tau)|_{\A}=|t|_F^{-1}|\lambda_{r,l}^n(\tau)_{\a}|_F$;
	\item $J_{k,S}(\gamma,z)J_{k,S}(y,\b{i}_0)=J_{k,S}(\gamma y,\b{i}_0)= J_{k,S}(\tau,\b{w}\b{i}_0)J_{k,S}(\b{w},\b{i}_0)=J_{k,S}(\tau,\b{i}_0)J_{k,S}(\b{w},\b{i}_0)$.
\end{enumerate}
Moreover, since the product $\chi_{\h}(t)^{-1}\chi_{\f{c}}(\det(d_w))^{-1}$ depends only on the symplectic part of $\gamma$, we can follow the reasoning in \cite[Lemma 3.6]{Sh95} and denote it by $\chi[\gamma]$, which agrees with the definition of $\chi[\gamma]$ in \cite[(3.11)]{Sh95}. Taking all these into account we obtain
\begin{align}\label{eq:KlingenES}
	E(z,s;\mathbf{f}) &=\sum_{\gamma \in R} \chi[\gamma] |t|_F^{2s} \(\frac{\delta(\gamma z)}{ \delta(\b{\omega}_r(\gamma z))}\)^{s-k/2}f(\b{\omega}_r(\gamma z))J_{k,S}(\gamma,z)^{-1}\nonumber\\
	&= \sum_{\zeta \in Z} |\lambda_{r,l}^n(\zeta)|_F^{2s} \sum_{\gamma \in R_{\zeta}}\chi[\gamma] \(\frac{\delta(z)}{ \delta(\b{\omega}_r(z))}\)^{s-k/2}f(\b{\omega}_r(z))|_{k,S}\gamma.
\end{align}

Analogously, if $r=0$ (and $\mathbf{f}=1$), we obtain the Siegel type Jacobi Eisenstein series,
\begin{equation}\label{eq:SiegelES}
E(z,s)=\sum_{\zeta \in Z} |\lambda_{0,l}^n(\zeta)|_F^{2s} \sum_{\gamma \in R_{\zeta}}\chi[\gamma] \delta(z)^{s-k/2}|_{k,S}\gamma = \sum_{\zeta \in Z} N(\mathfrak{a}(\zeta))^{2s} \sum_{\gamma \in R_{\zeta}}\chi[\gamma] \delta(z)^{s-k/2}|_{k,S}\gamma.
\end{equation}

We finish this section with a result regarding the absolute convergence of the series.

\begin{prop}\label{absolute convergence of Eisenstein series} The Eisenstein series $E(z,s;\mathbf{f})$ is absolutely convergent for $\Re(2s) > n+r+l+1$. In particular for $k \mathbf{a} \in \mathbb{Z}^{\mathbf{a}}$ with $k > n+r+l+1$ the series $E(z,k/2;\mathbf{f})$ is a Siegel-Jacobi form of parallel weight $k$.
\end{prop} 
\begin{proof} This follows from the calculations of Ziegler in \cite[pages 204-207]{Z89}. The difference with his Theorem 2.5 is the different normalisation of our Eisentein series as well as the introduction of the complex parameter $s$, but it is easy to see that his calculations lead to the range of absolute convergence stated above. 
\end{proof}

Later in the paper we will explore analytic properties of the Klingen-type Eisenstein series, such as analytic continuation and possible poles regarding the parameter $s$. This will be done in section \ref{section of analytic properties of Eisenstein series}. Furthermore, in the last section of this paper we will study the analytic properties of $E(z,s;\mathbf{f})$ with respect to the variable $z$ for some particular values of $s$. Namely, we will try to establish whether this series, even if it fails to be holomorphic in $z$, still has some good algebraic properties. To do this, we will introduce in the last section the notion of nearly holomorphic Siegel-Jacobi forms, and we will see that for particular values of $s$ the Jacobi Eisenstein series are of this kind.


\section{The Doubling Method}

As it was discussed in the introduction of this paper one of the most fruitful methods for studying various $L$-functions attached to (classical, i.e. Siegel, Hermitian, orthogonal) automorphic forms is, what is often called, the doubling method. It is perhaps not surprising that the same method can be used to study also $L$-functions attached to Siegel-Jacobi forms. We will introduce the latter a bit later in the paper, after developing necessary background for the doubling method. Actually there are two, rather different, ways to use this method.

\begin{enumerate}

\item \textbf{Method I.} This is the original approach of Murase \cite{Mu89,Mu91}, where he used a homomorphism (actually an injection)
\[
\Jac^{n,l} \times \Jac^{n,l} \rightarrow \Sp_{l+2n}.
\]
As we indicated in the introduction one of the main advantages of this approach is the fact that analytic properties of the $L$-function can be read off from analytic properties of (classical) Siegel Eisenstein series of $\Sp_{2n+l}$, which are well-understood. On the other hand, it is not quite clear how one could translate the picture classically, i.e. pulling back the Siegel Eisenstein series to the Jacobi symmetric space, which makes the method less attractive for other applications (differential operators, algebraicity, study of Klingen-type Eisenstein series and others).  

\item \textbf{Method II.} The second approach, which we follow in this paper, was first employed by Arakawa \cite{Ar94}. It uses a homomorphism (shortly to be made explicit),
\[
\Jac^{m,l} \times \Jac^{n,l} \rightarrow \Jac^{n+m,l} .
\]
This seems to be a more natural approach and closer to the spirit of the doubling method, since one ``doubles'' the same ``kind'' of a group. Moreover, it is quite clear what happens on the corresponding symmetric spaces. However, this method calls for a study of analytic and algebraic properties of Siegel-type Jacobi Eisenstein series introduced in the previous section, a task that will be taken upon later in this paper.  
\end{enumerate}

In this section we will develop technical results which will be necessary to apply the doubling method. The main result here is Lemma \ref{key_lemma}, which will be used in the next section to study a particular pullback of a Siegel-type Eisenstein series. Our approach is modeled on the work of Shimura in \cite{Sh95} where the symplectic case is considered, and our results here generalize those of Shimura to the Jacobi setting.

We define first the map mentioned above. Let
$$\iota_A\colon\Jac^{m,l}\times\Jac^{n,l}\to\Jac^{m+n,l},$$ 
$$\iota_A((\l,\mu,\kappa)g)\times (\l',\mu',\kappa')g')):= ((\l\, \l'), (\mu\, \mu'),\kappa+\kappa';\iota_S(g\times g')),$$
where 
$$\iota_S:G^m\times G^n\hookrightarrow G^{m+n},\quad \iota_S\(
\(\begin{smallmatrix}
a  & b  \\ c  & d  \\
\end{smallmatrix}\)\times
\(\begin{smallmatrix}
a'  & b'  \\ c'  & d' \\
\end{smallmatrix}\)\) :=
\(\begin{smallmatrix}
a	&	 &  b  &    \\
	& a' &     & b' \\
c	&	 & d   &    \\
	& c' &     & d' \\
\end{smallmatrix}\) .$$
In what follows we will often write $\b{g}\times\b{g}'$ for $\iota_A(\b{g}\times\b{g}')$. Sometimes it will be useful to view elements of $\b{G}^{m+n,l}$ as elements of $G^{l+m+n}$ via the embedding in equation (\ref{embedding of Jacobi to Symplectic}).
Denote by $H^{n,l}_r$ the Heisenberg subgroup of $\b{P}^{n,r}$, that is, put 
$$H^{n,l}_r(F):=\{ ((\l\, 0_{l, n-r}),\mu,\kappa)\in H^{n,l}(F)\}.$$ 
We will now adapt a method presented in \cite{Sh95} to find good coset representatives for $\b{P}^{m+n}(F)\back\Jac^{m+n}(F)$. Let $n\leq m$ and define $\b{\tau}_r:=1_H\tau_r\in\Jac^{m+n}(F)$, where
$$\tau_r:=\(\begin{smallmatrix}
1_m & 		&	 & \\
	& 1_n	&	 &	\\
	& e_r	&1_m &	\\
\T{e_r}	&	&	 & 1_n
\end{smallmatrix}\) ,\qquad 
e_r:=\(\begin{smallmatrix}
1_r & 	\\
	& 0
\end{smallmatrix}\)\in M_{m,n}(F) .$$

\begin{lem}
	If $n\leq m$,
	$$\Jac^{m+n}(F)=\bigsqcup_{0\leq r\leq n} \b{P}^{m+n}(F)\b{\tau}_r\iota_A(\Jac^{m}(F)\times\Jac^{n}(F)).$$
\end{lem}
\begin{proof}
	Let $\Jac^{m+n}(F)=\bigsqcup_i \b{P}^{m+n}(F)\b{g}_i\iota_A(\Jac^{m}(F)\times\Jac^{n}(F))$ be a double coset decomposition. There exist unique $g_i\in G^{m+n}(F)$ and $h_i\in H^{m+n,l}(F)$ such that $\b{g}_i=g_ih_i$. Note also that $\iota_A(\Jac^{m}(F)\times\Jac^{n}(F))=H^{m+n,l}(F)\rtimes\iota_A(G^m(F)\times G^n(F))$. We have
	\begin{align*}
	\Jac^{m+n}(F)&=\bigsqcup_i \b{P}^{m+n}(F) g_ih_i H^{m+n,l}(F)\iota_A(G^m(F)\times G^n(F))\\
	&=\bigsqcup_i H^{m+n,l}_0(F) P^{m+n}(F) H^{m+n,l}(F) g_i\iota_S(G^m(F)\times G^n(F)))\\
	&=\bigsqcup_i H^{m+n,l}_0(F) H^{m+n,l}(F) P^{m+n}(F)g_i\iota_S(G^m(F)\times G^n(F)).
	\end{align*}
	Since $\Jac^{m+n}(F)=H^{m+n,l}(F)G^{m+n}(F)$ and $G^{m+n}(F)=\bigsqcup_{0\leq r\leq n} P^{m+n}(F) \tau_r\iota_S(G^m(F)\times G^n(F))$ by \cite[Lemma 4.2]{Sh95}, we can take $\{g_i\}_i =\{\tau_r:0\leq r\leq n\}$ and thus $\{\b{g}_i\}_i =\{\b{\tau}_r:0\leq r\leq n\}$.
\end{proof}

\begin{lem}\label{lem:coset_decomp}
	$$\b{P}^{m+n}(F)\b{\tau}_r(\Jac^{m}(F)\times\Jac^{n}(F))=\bigsqcup_{\b{\xi},\b{\beta},\b{\gamma}} \b{P}^{m+n}(F)\b{\tau}_r((\b{\xi}\times 1_H1_{2m-2r})\b{\beta}\times\b{\gamma}),$$
	where $\b{\xi}$ runs over $Sym_l(F)\back\Jac^{r}(F)$, $\b{\beta}$ over $\b{P}^{m,r}(F)\back\Jac^{m}(F)$, and $\b{\gamma}$ over $\b{P}^{n,r}(F)\back\Jac^{n}(F)$.
\end{lem}
\begin{proof}
	By previous lemma and Lemma 4.3 from \cite{Sh95}, 
	\begin{multline*}
	\b{P}^{m+n}(F)\b{\tau}_r\iota_A(\Jac^{m}(F)\times\Jac^{n}(F))\\
	=\bigsqcup_{\xi,\beta,\gamma} H^{m+n,l}_0(F)H^{m+n,l}(F) P^{m+n}(F)\tau_r\iota_S(\iota_S(\xi\times 1_{2m-2r})\beta\times\gamma)),
	\end{multline*}
	where $\xi,\beta,\gamma$ run over $G^r(F), P^{m,r}(F)\back G^m(F), P^{n,r}(F)\back G^n(F)$ respectively.
	Note that $$H^{m+n,l}_0(F)H^{m+n,l}(F)=\bigcup_{\substack{\l\in M_{l,m}(F)\\ \l'\in M_{l,n}(F)}} H^{m+n,l}_0(F) ((\l,0,0)1_{2m}\times (\l',0,0)1_{2n}),$$
	and for $g=\(\begin{smallmatrix}
	A & B \\  &  D \\
	\end{smallmatrix}\)\in P^{m+n}(F)$, 
	$$((\l,0,0)1_{2m}\times (\l',0,0)1_{2n})1_H g\in H^{m+n,l}_0(F) P^{m+n}(F)((\l\,\l')A,0,0)1_{2(m+n)}.$$ 
	Indeed, if we view it as an element of $G^{l+m+n}$, we obtain
	\begin{align*}
	&\(\begin{smallmatrix} 
	1_l & \l  & \l' &     	   &	 &	   \\
	& 1_m &	    &     	   &  	 &	   \\
	&     & 1_n &     	   &	 &	   \\
	&     &     & 1_l      &	 &	   \\
	&	  &		& -\T{\l}  & 1_m & 	   \\
	&	  &		& -\T{\l'} & 	 & 1_n \\
	\end{smallmatrix}\)
	\(\begin{smallmatrix} 
	1_l &   &     &   \\
	& A	&	  & B \\
	&   & 1_l &   \\
	&   &     & D \\
	\end{smallmatrix}\)\\
	&\hspace{2cm}=\(\begin{smallmatrix} 
	1_l &   &     &   \\
	& A	&	  & B \\
	&   & 1_l &   \\
	&   &     & D \\
	\end{smallmatrix}\)
	\(\begin{smallmatrix} 
	1_l &     	  &	\kappa 			   & (\l\,\l')B \\
	& 1_{m+n} &	\T{B}\T{(\l\,\l')} &  		    \\
	&     	  &  1_l			   &      	    \\
	&	      &					   & 1_{m+n}    \\
	\end{smallmatrix}\)
	\(\begin{smallmatrix} 
	1_l & (\l\,\l')A &     	   	 		   &	   	 \\
	& 1_{m+n}    &	    	  	       &  	     \\
	&            & 1_l      		   &	   	 \\
	&	  		 & -\T{A}\T{(\l\,\l')} & 1_{m+n} \\
	\end{smallmatrix}\)\\
	&\hspace{2cm}=\(\begin{smallmatrix} 
	1_l &     	  &	\kappa 			    & (\l\,\l')B\T{A} \\
	& 1_{m+n} &	A\T{B}\T{(\l\,\l')} &  		    	  \\
	&     	  &  1_l			    &      	    	  \\
	&	      &					    & 1_{m+n}   	  \\
	\end{smallmatrix}\)
	\(\begin{smallmatrix} 
	1_l &   &     &   \\
	& A	&	  & B \\
	&   & 1_l &   \\
	&   &     & D \\
	\end{smallmatrix}\)
	\(\begin{smallmatrix} 
	1_l & (\l\,\l')A &     	   	 		   &	   	 \\
	& 1_{m+n}    &	    	  	       &  	     \\
	&            & 1_l      		   &	   	 \\
	&	  		 & -\T{A}\T{(\l\,\l')} & 1_{m+n} \\
	\end{smallmatrix}\) ,
	\end{align*}
	where $\kappa =(\l\,\l')B\T{A}\T{(\l\,\l')}$. 
		Moreover, because $\b{\tau}_r$ commutes with $((\l,0,0)1_{2m}\times (\l',0,0)1_{2n})$, we have
	\begin{multline*}
	\b{P}^{m+n}(F)\b{\tau}_r\iota_A(\Jac^{m}(F)\times\Jac^{n}(F))=\bigsqcup_{\xi,\beta,\gamma}\bigcup_{\substack{\l\in M_{l,m}(F) \l'\in M_{l,n}(F)}} H^{m+n,l}_0(F) P^{m+n}\b{\tau}_r\\
	\iota_A((\l,0,0)1_{2m}\times (\l',0,0)1_{2n})\iota_S(\iota_S(\xi\times 1_{2m-2r})\beta\times \gamma).
	\end{multline*}	
	Write $\l =(\l_1\,\l_2)$ and $\l' =(\l_1'\,\l_2')$ as concatenation of matrices $\l_1\in M_{l,r}(F),  \l_2\in M_{l,m-r}(F), \l_1'\in M_{l,r}(F),  \l_2'\in M_{l,n-r}(F)$.
	Because $H^{m+n,l}_0(F)$ and $P^{m+n}(F)$ commute (as follows from the above computation) and 
	\begin{multline*}
	H^{m+n,l}_0(F)\b{\tau}_r=\b{\tau}_r\{(\mu'\T{e}_r,\mu,\kappa)1_{2m}\times (\mu e_r,\mu',\kappa')1_{2n}:\\
	\mu\in M_{l,n}(F),\mu'\in M_{l,n}(F),\kappa, \kappa'\in Sym_l(F)\},
	\end{multline*}
	we can include $(0,(\l_1'\, 0),0)1_{2m}\times ((\l_1'\, 0),0,0)1_{2n}$
	in the set above for each $\l'$, and so we are left with
	$$(\l,(-\l_1'\, 0),0)\iota_S(\xi\times 1_{2m-2r})\beta\times ((0\, \l_2'),0,0)\gamma.$$ 
	In fact, 
	$$(\l,(-\l_1'\, 0),0)\iota_S(\xi\times 1_{2m-2r})\beta =((\l_1,-\l_1',0)\xi\times 1_H 1_{2m-2r})((0\,\l_2),0,0)\beta .$$
	Therefore we can exchange the representatives 
	$$\b{\tau}_r\iota_A(\iota_A((\l_1,-\l_1',0)\xi\times 1_H 1_{2m-2r})((0\,\l_2),0,0)\beta\times ((0\, \l_2'),0,0)\gamma)$$
	with $\b{\tau}_r\iota_A((\iota_A(\b{\xi}\times 1_H1_{2m-2r})\b{\beta}\times\b{\gamma})$, where $\b{\xi},\b{\beta},\b{\gamma}$ are as in the hypothesis. Reversing the process described above, it is easy to see that the cosets are distinct.
\end{proof}
We are now ready to prove the main result of this section. The following lemma is the generalization of \cite[Lemma 4.4]{Sh95}. 

\begin{lem}\label{key_lemma}
	Let $\f{e},\f{b}, \f{c}$ be as in Section \ref{sec:adelic_ES}, and $\b{\sigma}$ an element of $\Jac^{m+n}_{\b{h}}$ given by 
	$$\b{\sigma}_v:=\begin{cases}
	1_H\diag [1_m,\t_v^{-1}1_n,1_m,\t_v1_n] & \mbox{if } v\nmid\f{e},\\
	1_H\diag [1_m,\t_v^{-1}1_n,1_m,\t_v1_n]\b{\tau}_n & \mbox{if } v|\f{e},
	\end{cases}$$
	where $\t$ is an element of $F\ti_{\b{h}}$ such that $\t\f{o}=\f{b}$. Let $\b{D}^{m+n}:=K^{m+n}[\f{b},\f{c}]\subset\Jac^{m+n}(\A)$. Assume that $n\leq m$. Then
	\begin{multline*}
	\b{P}^{m+n}(F)\b{\tau}_n(\Jac^{m}(F)\times\Jac^{n}(F))\cap (\b{P}^{m+n}(\A) \b{D}^{m+n}\b{\sigma})\\
	=\bigsqcup_{\b{\xi}\in\b{X},\b{\beta}\in\b{B}} \b{P}^{m+n}(F)\b{\tau}_n ((1_H\iota_S(\b{\xi}\times 1_{2m'})\b{\beta}\times 1_H 1_{2n}),
	\end{multline*}
	where $m'=m-n$, $\b{B}$ is a subset of $\Jac^{m}(F)\cap Y$ as in Lemma \ref{Parabolic decomposition}, which represents $\b{P}^{m,n}(F)\back (\Jac^{m}(F)\cap \b{P}^{m,n}(\A)\b{D}^m)$, and $\b{X}=\Jac^{n}(F)\cap \Jac^{n}_{\a}\prod_{v\in\h} \b{X}_v$ with
	$$\b{X}_v=\begin{cases}
	\{ (\l,\mu,\kappa)x\in C_v[\f{o},\f{b}^{-1},\f{b}^{-1}]D_v^n[\f{b}^{-1}\f{c},\f{bc}]: a_x-1\in M_{n,n}(\f{e}_v)\} & \mbox{if } v|\f{e},\\
	C_v[\f{o},\f{b}^{-1},\f{b}^{-1}]D_v^n[\f{b}^{-1}\f{c},\f{b}]W_vC_v[\f{o},\f{b}^{-1},\f{b}^{-1}] D_v^n[\f{b}^{-1},\f{bc}] & \mbox{if } v|\f{e}^{-1}\f{c},\\
	C_v[\f{o},\f{b}^{-1},\f{b}^{-1}]G^n(F_v)C_v[\f{o},\f{b}^{-1},\f{b}^{-1}] & \mbox{if } v\nmid\f{c},\\
	\end{cases}$$
	$$W_v=\{ \diag[q,\tilde{q}]: q\in\GL_n(F_v)\cap M_{n,n}(\f{c}_v)\} ;$$
	if $m=n$, we take $\b{B}=\{ 1_H1_{2m}\}$.
\end{lem}

\begin{rem}Before we proceed to the proof of the lemma we should stress a significant difference between this result and the symplectic case. In \cite[Lemma 4.4]{Sh95}, at the places $v$ which do not divide $\mathfrak{c}$, one obtains that the set $X_v$ (with the notation there) is the entire symplectic group $G^n(F_v) = \Sp_n(F_v)$. However, this is not the case here as the set $\b{X}_v$ above is not equal to the group $\b{G}^n(F_v)$. This is one of the main differences between the Jacobi and the symplectic group regarding their Hecke theory at the ``good places''. It will become even more apparent later in this paper when we will consider the theory of Hecke operators. 
\end{rem}

\begin{proof}[Proof of Lemma \ref{key_lemma}]
We will divide the proof into two parts: the case where $v$ does not divide $\mathfrak{c}$ (a good place) and when it does (a bad place). We first consider the case of $v$ being good. \newline
We first obtain a description of the set $C_v[\f{o},\f{b}^{-1},\f{b}^{-1}]G^n(F_v)C_v[\f{o},\f{b}^{-1},\f{b}^{-1}]$. First note that a set of representatives for $G^n(F_v) / D_v[\mathfrak{b}^{-1},\mathfrak{b}]$ consists of
\[
m(g,h,\sigma) := \left(\begin{matrix}  g^{-1} h & g^{-1} \sigma \transpose{h}^{-1} \\  0 & \transpose{g} \transpose{h}^{-1}
\end{matrix} \right)
\]
where $(g,h) \in \GL_n(\f{o}_v) \back W / (\GL_n(\f{o}_v) \times 1_n)$, $\sigma \in Sym_n(F_v)/ g Sym_n(\mathfrak{b}_v^{-1}) \transpose{g}$ and $W = \{ (g,h ) \in B \times B: g L + h L = L \}$, where $L = M_{n,1}(\mathfrak{o}_v)$, and $B = GL_n(F_v) \cap M_n(\mathfrak{o}_v)$. In particular, if we write $\b{D}^{m+n}_v = C_v D_v$, then 
\begin{align}\label{good}\nonumber
C_vG^{n}(F_v)C_v &= \bigcup_{g,h,\sigma} C_v m(g,h,\sigma)D_vC_v = \bigcup_{g,h,\sigma} C_v m(g,h,\sigma)C_vD_v\\
&=\bigcup_{\substack{g,h,\sigma\\\lambda ,\mu}}  C_v (  \lambda h^{-1} g, -\lambda h^{-1} \sigma \transpose{g}^{-1} + \mu \transpose{h} \transpose{g}^{-1}, *)  m(g,h,\sigma)D_v .
\end{align}

Consider now the set $\b{P}^{m+n}(F_v) \b{D}^{m+n}_v$ and write $\b{P}^{m+n}(F_v) = H^{0}(F_v)P^{m+n}(F_v)$. Since 
\[
\(\begin{matrix} a_p & b_p \\
                        0 & d_p 
                       \end{matrix} \) (\lambda, \mu, *) = (\lambda a_p^{-1}, \lambda a_p^{-1} b_p d_p^{-1} + \mu d_p^{-1}, * ) \(\begin{matrix} a_p & b_p \\
                        0 & d_p 
                       \end{matrix} \) ,
\]
we can conclude that 
\[
\b{P}^{m+n}(F_v) \b{D}^{m+n}_v\! =\! \{ (\lambda,\mu,\kappa) g\! :\! \lambda\! \in\! M_{l,n+m}(\mathfrak{o}_v)a_p^{-1}\! , \mu \! \in\! M_{l,n+m}(F_v), g = pk \! \in\! \Sp_{n+m}(F_v) \} .
\]
Note that this is well defined. Indeed, if we write $g = p_1 k_1 = p_2 k_2$ then $p_1^{-1}p_2 \in D_v$ and in particular $a_{p_1}^{-1} a_{p_2} \in M_{n+m}(\mathfrak{o}_v) \cap \GL_{n+m}(F_v)$, and similarly $a_{p_2}^{-1} a_{p_1} \in M_{n+m}(\mathfrak{o}_v) \cap \GL_{n+m}(F_v)$; that is,   $a_{p_1}^{-1} a_{p_2} \in \GL_{n+m}(\mathfrak{o}_v)$.  \newline

Consider now $\b{\alpha}=\iota_A(\b{\xi}\times 1_H1_{2m'})\b{\beta}$ with $\b{\xi}\in Sym_l(F)\back\Jac^n(F)$, $\b{\beta}\in \b{P}^{m,n}(F)\back\Jac^{m}(F)$, and write $\b{\xi}=(\l_1,\mu_1,0)\xi, \b{\beta}=((0\, \l_2),0,0)\beta$, where $\l_2 \in M_{r, m-n}(F)$. Then 
\begin{align*}
 \b{\alpha} &= \iota_A((\l_1,\mu_1,0)\xi \times 1_H1_{2m'}) ((0\, \l_2),0,0)\beta = ((\l_1\,0), (\mu_1\,0), 0) (\xi \times 1_{2m'}) ((0\, \l_2),0,0)\beta\\
&= ((\l_1\,0), (\mu_1\,0), 0)((0\, \l_2),0,0) (\xi \times 1_{2m'})\beta = ((\l_1\,\l_2), (\mu_1\,0), 0) (\xi \times 1_{2m'})\beta,
\end{align*}
and so 
\[
\iota_A (\b{\alpha} \times 1_H 1_{2n}) = ((\l_1\,\l_2\, 0), (\mu_1\,0\,0), 0) ( (\xi \times 1_{2m'})\beta \times 1_{2n}).
\]

Now we see that 
\begin{align*}
 \b{\tau}_n \iota_A (\b{\alpha} \times 1_H 1_{2n}) \b{\sigma}^{-1} &= ((\l_1\, \l_2\,(-\mu_1)),(\mu_1\,0\,0),0) \tau_n ( (\xi \times 1_{2m'})\beta \times 1_{2n})\b{\sigma}^{-1}\\
&=((\l_1\, \l_2\,(-\mu_1)),(\mu_1\,0\,0),0) \tau_n ( (\xi \times 1_{2m'})\beta \times 1_{2n})\sigma^{-1}.
\end{align*}

Put $g:=\tau_n( (\xi \times 1_{m-n})\beta \times 1_{2n})\sigma^{-1}$ and write $g = pk \in P^{m+n} D^{m+n}$. Then by \cite[Lemma 4.4]{Sh95} we may take $\beta$ to be of the form $hw$, where $h = \diag[1_{m-1}, t^{-1}, 1_{m-1}, t]$ and $w$ is in the congruence subgroup $D^m$. Moreover, we may take 
\[
\xi = \left(\begin{matrix}  g^{-1} h & g^{-1} \sigma \transpose{h}^{-1} \\ 0 & \transpose{g} \transpose{h}^{-1} 
\end{matrix} \right) d,
\] 
where $g,h,\sigma$ are in the sets as above, and $d\in D^n$. In particular,  
\[
(\xi \times 1_{2m'})\beta \times 1_{2n} = \left (\begin{matrix} A &  0 & B & 0 \\0  & 1_n & 0 & 0 \\0 & 0 & D & 0 \\0 & 0 & 0 & 1_n 
\end{matrix} \right)d_1,
\]
where $d_1$ is some element in $D^{n+m}$,
\[
A :=\left (\begin{matrix} g^{-1} h  &  0 \\0  & \tilde{h} 
\end{matrix} \right),\,\,
B : = \left (\begin{matrix} g^{-1} \sigma \transpose{h}^{-1} & 0\\ 0 & 0 \end{matrix} \right),\,\,
D := \left (\begin{matrix} \transpose{g} \transpose{h}^{-1}  &  0\\ 0 & \tilde{h}^{-1} \end{matrix} \right)
\]
and $\tilde{h} = \diag[1_{m-n-1}, t ]$. In this way we obtain
\[
\tau_n ( (\xi \times 1_{2m'})\beta \times 1_{2n}) \sigma^{-1}= \left (\begin{matrix} A & 0 & B & 0 \\
				0 &   \theta_v1_n & 0 & 0 \\
				0 & \theta_v e_n & D & 0 \\
	\transpose{e}_n A &  0 &  \transpose{e}_n B& \theta_v^{-1}1_n 
 \end{matrix} \right) d'
\]
for some $d'$ in the congruence subgroup $D^n$. Furthermore, if we write
\[
\left (\begin{matrix} A &0  & B & 0 \\
 0 &   \theta_v1_n & 0 &  0\\
 0 & \theta_v e_n & D & 0 \\
\transpose{e}_n A & 0 &  \transpose{e}_n B& \theta_v^{-1}1_n 
 \end{matrix} \right) = pk
\]
for some $ p  \in P^{n+m}(F_v)$ and $k = \left( \begin{matrix} k_1 & k_2 \\ k_3 & k_4 \end{matrix} \right) \in D_v^{n+m}[\mathfrak{b}^{-1},\mathfrak{b}\mathfrak{c}]$, then we can conclude that 
\[
\transpose{a}^{-1}_p k_3 = \left( \begin{matrix}  0 & \theta_v e_n \\ \transpose{e}_n A & 0 \end{matrix} \right)\quad
\mbox{and}\quad
\transpose{a}^{-1}_p k_4 =  \left( \begin{matrix} D  &0  \\  \transpose{e}_n B & \theta_v^{-1} 1_n \end{matrix} \right) .
\]
Since the matrix $[k_3\,\, k_4]$ extends to an element in the congruence subgroup $D_v^{n+m}[\mathfrak{b}^{-1},\mathfrak{b}\mathfrak{c}]$, it follows that 
\[
\theta_v^{-1} k_3 \Lambda + k_4 \Lambda = \Lambda, 
\]
where now $\Lambda = M_{n+m,l}(\mathfrak{o})$. That is, for any given $\ell \in \Lambda$ there exist $\ell_1, \ell_2 \in \Lambda$ such that $\theta_v^{-1} k_3 \ell_1 + k_4 \ell_2 = \ell$. Write $\Lambda = \transpose{[\Lambda_1, \Lambda_2, \Lambda_3]}$ with $\Lambda_1,\Lambda_3 \in M_{l,n}$ and $\Lambda_2 \in M_{l,m-n}$. Then the relation
$\transpose{a}^{-1}_p \theta_v^{-1}k_3 \Lambda + \transpose{a}^{-1}_p k_4 \Lambda = \transpose{a}^{-1}_p \Lambda$, which can be also written as
\[
 \left( \begin{matrix}  0 & e_n \\ \theta_v^{-1} \transpose{e}_n A & 0 \end{matrix} \right) \Lambda  + \left( \begin{matrix} D  &  0\\ \transpose{e}_n B & \theta^{-1}_{v}1_n \end{matrix} \right) \Lambda =  \transpose{a}^{-1}_p \Lambda,
\] 
means that the set $\transpose{a}^{-1}_p \Lambda$ can be described as
\[
 \left( \begin{matrix}  0 &  e_n \\ \theta_v^{-1}\transpose{e}_n A & 0 \end{matrix} \right)   \transpose{[\ell_1, \ell_2, \ell_3]}+ \left( \begin{matrix} D  & 0 \\ \transpose{e}_n B & \theta_v^{-1}1_n \end{matrix} \right) \transpose{[\ell'_1, \ell'_2, \ell'_3]},
\]
where $\ell_1, \ell_1' \in \Lambda_1$, $\ell_3, \ell_3' \in \Lambda_3$, $\ell_2,\ell'_2 \in \Lambda_2$ and, recall, $e_n ={1_n\choose  0}\in M_{m,n}$. Therefore, since $\transpose{e_n} A = \left( \begin{matrix} g^{-1}h\;\; 0 \end{matrix} \right)$ and $\transpose{e_n} B = \left( \begin{matrix} g^{-1}\sigma \transpose{h}^{-1}\;\; 0\end{matrix} \right)$, we get
\[
\left( \begin{matrix}  0 & e_n \\ \theta_v^{-1}\transpose{e}_n A &0  \end{matrix} \right)    \transpose{[\ell_1, \ell_2, \ell_3]}= \left( \begin{matrix} \transpose{\ell}_3 \\ 0 \\ \theta_v^{-1}g^{-1}h \transpose{\ell_1}  \end{matrix} \right)
\]
and
\[
\left( \begin{matrix} D  & 0 \\ \transpose{e}_n B & \theta_v^{-1}1_n \end{matrix} \right)  \transpose{[\ell'_1, \ell'_2, \ell'_3]} = \left( \begin{matrix} \transpose{g}\transpose{h}^{-1} \transpose{\ell'_1} \\ \tilde{h} \transpose{\ell'_2} \\ g^{-1} \sigma \transpose{h}^{-1} \transpose{\ell'}_1 + \theta_v^{-1}\transpose{\ell'}_3 \end{matrix} \right).
\]
Hence,
\[
\transpose{a}^{-1}_p \Lambda  
= \left( \begin{matrix} \transpose{\ell}_3 +\transpose{g}\transpose{h}^{-1} \transpose{\ell'_1}   \\ \tilde{h} \transpose{\ell'_2} \\ g^{-1}h  \theta_v^{-1} \transpose{\ell_1} + g^{-1} \sigma \transpose{h}^{-1} \transpose{\ell}'_1 + \theta_v^{-1}\transpose{\ell}'_3 \end{matrix} \right) ,
\]
and after taking a transposition
\[
\transpose{\Lambda} a_p^{-1} = \left( \begin{matrix} \ell_3 +\ell'_1 h^{-1} g   & \ell'_2 \transpose{\tilde{h}} &  \theta_v^{-1}\ell_1 \transpose{h} \transpose{g}^{-1} +  \ell'_1  h^{-1} \sigma \transpose{g}^{-1}+ \theta_v^{-1}\ell'_3 \end{matrix} \right) .
\]

In particular, we see that the element
$$\b{\tau}_n \iota_A (\b{\alpha} \times 1_H 1_{2n}) \b{\sigma}^{-1} = ((\l_1\, \l_2\,(-\mu_1)),(\mu_1\,0\,0),0) \tau_n ( (\xi \times 1_{m-n})\beta \times 1_{2n})\sigma^{-1}$$
belongs to $\mathbf{P}^{n+m}(F_v) \mathbf{D}^{m+n}_v$  
if and only if $\lambda_1$ is of the form 
$\ell_3 +\ell'_1 h^{-1} g$, and $\mu_1$ is of the form  $-(\theta_v^{-1}\ell_1 \transpose{h} \transpose{g}^{-1} +  \ell'_1  h^{-1} \sigma \transpose{g}^{-1}+ \theta_v^{-1}\ell'_3)$. This together with \eqref{good} concludes the proof of the lemma in the case of good places.

Now assume that $v$ is a place in the support of $\mathfrak{c}$. First we consider the case when $v|\mathfrak{e}^{-1}\mathfrak{c}$. As above, we start with a description of the set 
\[
C_v[\f{o},\f{b}^{-1},\f{b}^{-1}]D_v^n[\f{b}^{-1}\f{c},\f{b}]W_vC_v[\f{o},\f{b}^{-1},\f{b}^{-1}] D_v^n[\f{b}^{-1},\f{bc}],
\]
where $W_v=\{ \diag[q,\tilde{q}]: q\in\GL_n(F_v)\cap M_{n,n}(\f{c}_v)\} .$ As it was shown in \cite[page 567]{Sh95}, 
\[
D_v^n[\f{b}^{-1}\f{c},\f{b}] \diag[q, \tilde{q}] D_v^n[\f{b}^{-1},\f{bc}] = \bigcup_{f,g} \left(\begin{matrix} f & g \tilde{f} \\ 0 & \tilde{f} \end{matrix} \right)D_v^n[\f{b}^{-1},\f{bc}] ,
\]
where $f \in\GL_n(\f{o}_v) \back\GL_n(\f{o}_v) q \GL_n(\f{o}_v)$ and $g \in Sym_n(\mathfrak{b}^{-1}_v\mathfrak{c}_v)/ \transpose{f} Sym_n(\mathfrak{b}^{-1}_v)f$. Set $C_v := C_v[\f{o},\f{b}^{-1},\f{b}^{-1}]$. Then:
\begin{align}\label{bad}\nonumber
C_vD_v^n&[\f{b}^{-1}\f{c},\f{b}]W_v C_v D_v^n[\f{b}^{-1},\f{bc}] = C_v D_v^n[\f{b}^{-1}\f{c},\f{b}]W_v D_v^n[\f{b}^{-1},\f{bc}] C_v \\\nonumber
&=\bigcup_{q}  \bigcup_{f_q,g_q} C_v \left(\begin{matrix} f_q & g_q \tilde{f}_q \\  0 & \tilde{f}_q \end{matrix} \right)D_v^n[\f{b}^{-1},\f{bc}] C_v = \bigcup_{q}  \bigcup_{f_q,g_q} C_v \left(\begin{matrix} f_q & g_q \tilde{f}_q \\0  & \tilde{f}_q \end{matrix} \right)C_vD_v^n[\f{b}^{-1},\f{bc}] \\
&=\bigcup_{q}  \bigcup_{f_q,g_q,\lambda,\mu} C_v (\lambda f_q^{-1}, -\lambda f_q^{-1}g_q + \mu \transpose{f}_q, *] \left(\begin{matrix} f_q & g_q \tilde{f}_q \\0  & \tilde{f}_q \end{matrix} \right)D_v^n[\f{b}^{-1},\f{bc}] ,
\end{align}
where $f_q \in \GL_n(\f{o}_v)\back\GL_n(\f{o}_v) q \GL_n(\f{o}_v)$ and $g_q \in Sym_n(\mathfrak{b}^{-1}_v\mathfrak{c}_v)/ \transpose{f}_q Sym_n(\mathfrak{b}^{-1}_v)f_q$. 

Further we argue as in the case of good places. In particular, we may write as before
\[
 \b{\tau}_n \iota_A (\b{\alpha} \times 1_H 1_{2n}) \b{\sigma}^{-1} = ((\l_1\, \l_2\,(-\mu_1)),(\mu_1\,0\,0),0) \tau_n ( (\xi \times 1_{m-n})\beta \times 1_{2n})\sigma^{-1}
\]
with $\b{\xi}=(\l_1,\mu_1,0)\xi\in Sym_l(F)\back\Jac^n(F)$, $\b{\beta}=((0\, \l_2),0,0)\beta\in \b{P}^{m,n}(F)\back\Jac^{m,l}(F)$. Moreover, using \cite[Lemma 4.4]{Sh95} again, we may take $\xi = \left(\begin{matrix} f_q & g_q \tilde{f}_q \\ 0& \tilde{f}_q \end{matrix} \right)d$ for some $q \in M_n(\mathfrak{c}_v) \cap \GL_n(F_v)$,  $f_q \in \GL_n(\f{o}_v) \back \GL_n(\f{o}_v) q \GL_n(\f{o}_v)$, $g_q \in Sym_n(\mathfrak{b}^{-1}_v\mathfrak{c}_v)/ \transpose{f}_q Sym_n(\mathfrak{b}^{-1}_v)f_q$ and $d \in D_v[\mathfrak{b}^{-1},\mathfrak{b}\mathfrak{c}]$. Then we obtain 

\[
\tau_n ( (\xi \times 1_{2m'})\beta \times 1_{2n}) \sigma^{-1}= \left (\begin{matrix} A & 0 & B &0 \\0 &   \theta_v1_n &0 & 0 \\ 0& \theta_v e_n & D &0  \\ \transpose{e}_n A &0  &  \transpose{e}_n B& \theta_v^{-1}1_n  
\end{matrix} \right) d'
\]
for some $d'\in D_v^{m+n}$, where this time 
\[
A :=\left (\begin{matrix}  f_q  &  0\\  0& \tilde{f_q} 
\end{matrix} \right),\,\,
B : = \left (\begin{matrix}  g_q \transpose{f_q}^{-1}  &0 \\ 
0 & 0 \end{matrix} \right),\,\,
D := \left (\begin{matrix}  \transpose{f}_q^{-1}  & 0 \\ 
0& \tilde{h}^{-1} \end{matrix} \right)\]

As before, write $\(\begin{smallmatrix} A & 0& B & 0\\
0 &   \theta_v1_n &0  & 0 \\
0 & \theta_v e_n & D &0  \\
 \transpose{e}_n A &0  &  \transpose{e}_n B& \theta_v^{-1}1_n 
 \end{smallmatrix} \)$ as a product of an element in $P^{m+n}$ and $D^{m+n}$. Then, after the same computations and with notation as above, we obtain   
\[
\transpose{\Lambda} a_p^{-1} = \left( \begin{matrix} \ell_3 +\ell'_1 f_q^{-1}    & \ell'_2 \transpose{\tilde{h}} &  \theta_v^{-1}\ell_1 \transpose{f}_q  +  \ell'_1  f_q^{-1} g_q + \theta_v^{-1}\ell'_3 \end{matrix} \right)
\]
 In particular, we see that the element 
$$\b{\tau}_n \iota_A (\b{\alpha} \times 1_H 1_{2n}) \b{\sigma}^{-1} = ((\l_1\, \l_2\,(-\mu_1)),(\mu_1\,0\,0),0) \tau_n ( (\xi \times 1_{m-n})\beta \times 1_{2n})\sigma^{-1}$$ 
belongs to $\mathbf{P}^{n+m}(F_v) \mathbf{D}^{m+n}_v$  if and only if $\lambda_1$ is of the form 
$\ell_3 +\ell'_1 f_q^{-1} $, and $\mu_1$ is of the form  $-(\theta_v^{-1}\ell_1 \transpose{f_q}  +  \ell'_1  f_q^{-1} g_q + \theta_v^{-1}\ell'_3)$. This requirement matches the decomposition \eqref{bad}, and thus finishes the proof of the second case.
\newline

Finally, we consider the case of $v|\mathfrak{e}$. In this situation we also argue as before, but note that now  
\[
\tau_n ( (\xi \times 1_{2m'})\beta \times 1_{2n}) \sigma^{-1}= \left (\begin{matrix}A &  0& B &0  \\0 & \theta_v1_n & 0 &0  \\
 0& \theta_v e_n & D &0  \\ 
\transpose{e}_n A &  0&  \transpose{e}_n B& \theta_v^{-1}1_n 
\end{matrix} \right) d' ,
\]
where
\[ d'\in D_v^{m+n},\,\,
A :=\left (\begin{matrix}  1_n  &0  \\  0& 1_n  
\end{matrix} \right),\,\,
B : = \left (\begin{matrix} 0 &0  \\ 0 & 0
\end{matrix} \right),\,\,
D := \left (\begin{matrix}  1_n  & 0 \\ 0& \tilde{h}^{-1} 
\end{matrix} \right) .\]

Hence, doing exactly the same computations as before, we see that the element 
$$\b{\tau}_n \iota_A (\b{\alpha} \times 1_H 1_{2n}) \b{\sigma}^{-1} = ((\l_1\, \l_2\,(-\mu_1)),(\mu_1\,0\,0),0) \tau_n ( (\xi \times 1_{m-n})\beta \times 1_{2n})\sigma^{-1}$$ 
belongs to $\mathbf{P}^{n+m}(F_v) \mathbf{D}^{m+n}_v$  if and only if $\lambda_1$ is of the form 
$\ell_3 +\ell'_1  $, and $\mu_1$ is of the form  $-(\theta_v^{-1}\ell_1   +  \ell'_1  + \theta_v^{-1}\ell'_3)$, which gives the set we claimed in the lemma.\end{proof}

\section{Diagonal Restriction of Eisenstein Series}\label{sec:diag_restr}

The map $\b{G}^{m,l} \times \b{G}^{n,l} \rightarrow \b{G}^{m+n,l}$ introduced in the previous section induces an embedding 
\[
\mathcal{H}_{m,l} \times \mathcal{H}_{n,l} \hookrightarrow \mathcal{H}_{n+m,l},\,\,\,z_1 \times z_2 \mapsto \diag[z_1,z_2],
\]
defined by 
\[
(\tau_1,w_1) \times  (\tau_2,w_2) \mapsto  (\diag[\tau_1,\tau_2], (w_1 \, w_2)).
\]

The aim of this section is to obtain the main identity (\ref{main_inner_product}), that is, to compute the Petersson inner product of a cuspidal Siegel-Jacobi modular form against a pull-backed Siegel-type Eisenstein series.   This identity should be seen as a generalization of the identity \cite[equation (4.11)]{Sh95} from the Siegel to the Jacobi setting. 

\subsection{The factor of automorphy}

We start with a study of the behavior of the factor of automorphy under diagonal restriction. 
First we compute $J_{k,S}(\b{\tau}_r, z)$ for $0 \leq r \leq n$; similar calculations have also been done in \cite[page 191]{Ar94}.

\begin{lem} Let $z = \diag[z_1,z_2]$ be as above, and $\b{\tau}_r$ as in the previous section. Then
\begin{align*}
J_{k,S}(\b{\tau}_r, z)&=
 \mathbf{e_a}(-\tr(S[\b{\omega}_r(w_2)\omega_r(\tau_2)^{-1} - \b{\omega}_r(w_1)] (\omega_r(\tau_2)^{-1} -\omega_r(\tau_1) )^{-1}  ))\\
 &\hspace{0.4cm}\cdot J_{k,S}(\b{\eta}_r, \b{\omega}_r(z_2))) \det(\omega_r(\tau_1) - \omega_r(\tau_2)^{-1})^k,
\end{align*}
where, recall, we write $\b{\omega}_r(z_i) = \b{\omega}_r(\tau_i,w_i) = (\omega_r(\tau_i), \b{\omega}_r(w_i))$ for $i=1,2$.
\end{lem}
\begin{proof}
By definition
\[
J_{k,S}(\b{\tau}_r, z) = j(\tau_r, \diag [\tau_1,\tau_2]) \mathbf{e_a}(\tr(S[w_1 \, w_2] \lambda(\tau_r,\diag[\tau_1,\tau_2])^{-1} f_r)),
\]
where $\tau_r = \(\begin{smallmatrix}
1_{N} &  \\ f_r & 1_N
\end{smallmatrix}\)$, $f_r = \(\begin{smallmatrix}
 & e_r \\ \transpose{e}_r &  
\end{smallmatrix}\)$ and $e_r = \(\begin{smallmatrix}
1_r &  \\  & 0 
\end{smallmatrix}\)$, with $N :=m+n$. Further 
\begin{align*}
\lambda(\tau_r,\diag[\tau_1,\tau_2])^{-1}f_r &=(f_r \diag[\tau_1, \tau_2] + 1_N)^{-1}f_r = \spmatrix{
1_m & e_r \tau_2 \\ \transpose{e}_r \tau_1 & 1_n }^{-1}f_r\\
&= \spmatrix{
1_m & -e_r \tau_2 \\ -\transpose{e}_r \tau_1 & 1_n}
\spmatrix{
(1_m-e_r \tau_2 \transpose{e}_r \tau_1)^{-1} & 0 \\  0& (1_n-\transpose{e}_r \tau_1 e_r \tau_2)^{-1}} \spmatrix{
 0& e_r \\ \transpose{e}_r &  0}\\
&=
\spmatrix{
1_m & -e_r \tau_2 \\ -\transpose{e}_r \tau_1 & 1_n}\spmatrix{
 0& (1_m-e_r \tau_2 \transpose{e}_r \tau_1)^{-1}e_r \\
(1_n-\transpose{e}_r \tau_1 e_r \tau_2)^{-1} \transpose{e}_r &0 } ,
\end{align*}
where
\begin{align*}
(1_m-e_r \tau_2 \transpose{e}_r \tau_1)^{-1}e_r  &= \spmatrix{
1_r - \omega_r(\tau_2) \omega_r(\tau_1) & -\omega_r(\tau_2) b_{\tau_2} \\
0 & 1_{m-r}}
\spmatrix{1_r & 0 \\0  & 0 }\\
&= \spmatrix{
(1_r-\omega_r(\tau_2) \omega_r(\tau_1))^{-1} & 0 \\0  & 0 }
\end{align*}
and, similarly, 
\[
(1_n-\transpose{e}_r \tau_1 e_r \tau_2)^{-1}\transpose{e}_r = \spmatrix{
(1_r-\omega_r(\tau_1) \omega_r(\tau_2))^{-1} & 0 \\ 0 & 0 }
\]

Hence,
\begin{align*}
(f_r \diag[\tau_1, \tau_2] + 1_N)^{-1}f_r &=\\
&\hspace{-3cm}=\spmatrix{
1_m & -e_r \tau_2 \\ -\transpose{e}_r \tau_1 & 1_n }
\spmatrix{
 0& 0 & (1_r-\omega_r(\tau_2) \omega_r(\tau_1))^{-1} &  0\\
0 & 0 & 0 & 0 & 0\\
(1_r-\omega_r(\tau_1) \omega_r(\tau_2))^{-1}  & 0 &0  &0  \\
 0& 0 & 0 & 0 & 0 }\\
&\hspace{-3cm}=\spmatrix{
1_r & 0 & -\omega_r (\tau_2) & * \\ 
0 & 1_{m-r} & 0 &0  \\
-\omega_r(\tau_1) & * & 1_r & 0 \\
 0&  0&  0& 1_{n-r} }
\spmatrix{
 0& 0 & (1_r-\omega_r(\tau_2) \omega_r(\tau_1))^{-1} &  0\\
0 & 0 & 0& 0   \\
(1_r-\omega_r(\tau_1) \omega_r(\tau_2))^{-1} & 0 & 0 &0  \\
 0 & 0 & 0 & 0 }\\
&\hspace{-3cm}=\spmatrix{
-\omega_r(\tau_2) (1_r-\omega_r(\tau_1) \omega_r(\tau_2))^{-1} & 0 & (1_r-\omega_r(\tau_2) \omega_r(\tau_1))^{-1} & 0 \\
 0& 0 & 0 &0 \\
(1_r-\omega_r(\tau_1) \omega_r(\tau_2))^{-1} & 0 & -\omega_r(\tau_1) (1_r-\omega_r(\tau_2) \omega_r(\tau_1))^{-1} & 0 \\
 0&  0&  0& 0 } ,
\end{align*}
and thus we can compute
\begin{align*}
\tr & (S[w_1 \, w_2] \lambda(\tau_r,\diag[\tau_1,\tau_2])^{-1} f_r) =\\ 
&=\tr( -\T{\b{\omega}_r(w_1)} S \b{\omega}_r(w_1) \omega_r(\tau_2)(1_r-\omega_r(\tau_1) \omega_r(\tau_2))^{-1})\\
&\hspace{0.4cm} + \tr(\T{\b{\omega}_r(w_1)} S \b{\omega}_r(w_2) (1_r-\omega_r(\tau_1) \omega_r(\tau_2))^{-1})\\
&\hspace{0.4cm} + \tr(\T{\b{\omega}_r(w_2)}S \b{\omega}_r(w_1)(1_r-\omega_r(\tau_2) \omega_r(\tau_1))^{-1})\\ 
&\hspace{0.4cm} - \tr(\T{\b{\omega}_r(w_2)} S \b{\omega}_r(w_2)\omega_r(\tau_1) (1_r-\omega_r(\tau_2) \omega_r(\tau_1))^{-1} )\\
&= \tr( (-\T{\b{\omega}_r(w_1)} S \b{\omega}_r(w_1) \omega_r(\tau_2) +  \T{\b{\omega}_r(w_1)} S \b{\omega}_r(w_2)) (1-\omega_r(\tau_1) \omega_r(\tau_2))^{-1})\\
&\hspace{0.4cm} +\tr((\T{\b{\omega}_r(w_2)}S \b{\omega}_r(w_1) - \T{\b{\omega}_r(w_2)} S \b{\omega}_r(w_2)\omega_r(\tau_1)) (1-\omega_r(\tau_2) \omega_r(\tau_1))^{-1})\\
&=\tr( (-\T{\b{\omega}_r(w_1)} S \b{\omega}_r(w_1)  +  \T{\b{\omega}_r(w_1)} S \b{\omega}_r(w_2)\omega_r(\tau_2)^{-1}) (\omega_r(\tau_2)^{-1}-\omega_r(\tau_1) )^{-1})\\
&\hspace{0.4cm} -\tr((\transpose{(\b{\omega}_r(w_2)\omega_r(\tau_2)^{-1})} S \b{\omega}_r(w_2)\omega_r(\tau_2)^{-1}\omega_r(\tau_2)\omega_r(\tau_1)) (\omega_r(\tau_2)^{-1} +\omega_r(\tau_1))^{-1} )\\
&\hspace{0.4cm} +\tr((\transpose{(\b{\omega}_r(w_2)\omega_r(\tau_2)^{-1})}S \b{\omega}_r(w_1)\\
&=\tr(-S[\b{\omega}_r(w_2)\omega_r(\tau_2)^{-1} - \b{\omega}_r(w_1)](\omega_r(\tau_2)^{-1} -\omega_r(\tau_1))^{-1})\\
&\hspace{0.4cm} +\tr(\transpose{(\b{\omega}_r(w_2) \omega_r(\tau_2)^{-1})}S\b{\omega}_r(w_2)).
\end{align*}

In particular, we conclude that
\begin{multline*}
J_{k,S}(\b{\tau}_r, z)  = \mathbf{e_a}(-\tr(S[\b{\omega}_r(w_2)\omega_r(\tau_2)^{-1} - \b{\omega}_r(w_1)](\omega_r(\tau_2)^{-1} -\omega_r(\tau_1))^{-1}   )) )\\
\cdot\mathbf{e_a} (\tr(S[\b{\omega}_r(w_2)\omega_r(\tau_2)^{-1}] \omega_r(\tau_2))) j(\tau_r, \diag [\tau_1,\tau_2])^k.
\end{multline*}

But $j(\tau_r, \diag [\tau_1,\tau_2])  = \det(1_r-\omega_r(\tau_1)\omega_r(\tau_2)) = \det(\omega_r(\tau_1) + \b{\eta}_r \omega_r(\tau_2)) \det(-\omega_r(\tau_2))$, where $\b{\eta}_r=1_H\(\begin{smallmatrix} & -1_r\\ 1_r & \end{smallmatrix}\)$,
and so we have that
\begin{multline*}
\mathbf{e_a} (\tr(S[\b{\omega}_r(w_2)\omega_r(\tau_2)^{-1}] \omega_r(\tau_2))) j(\tau_r, \diag [\tau_1,\tau_2] )^k\\
= J_{k,S}(\b{\eta}_r, (\omega_r(\tau_2), \b{\omega}_r(w_2))) \det(\omega_r(\tau_1) - \omega_r(\tau_2)^{-1})^k.
\end{multline*}
That is, $J_{k,S}(\b{\tau}_r, z)$ is equal to
\begin{multline*}
 \mathbf{e_a}(-\tr(S[\b{\omega}_r(w_2)\omega_r(\tau_2)^{-1} - \b{\omega}_r(w_1)] (\omega_r(\tau_2)^{-1} -\omega_r(\tau_1) )^{-1}  ))\\
\cdot J_{k,S}(\b{\eta}_r, (\omega_r(\tau_2), \b{\omega}_r(w_2))) \det(\omega_r(\tau_1) - \omega_r(\tau_2)^{-1})^k.
\end{multline*}

\end{proof}

Now, with the notation of Lemma \ref{key_lemma}, we compute $J_{k,S}(\b{\tau_r} ((\b{\xi} \times 1_{2m-2r}) \b{\beta} \times \b{\gamma}), \diag[z_1\, z_2])$. 

\begin{lem}With notation as above,
\begin{multline}\label{diag.aut.fac.}
J_{k,S}(\b{\tau_r} ((\b{\xi} \times 1_{2m-2r}) \b{\beta} \times \b{\gamma}), \diag[z_1\, z_2])=\\
=J_{k,S}(\b{\xi}, \b{\omega}_r(\b{\beta}z_1))J_{k,S}(\b{\beta}, z_1) J_{k,S}(\b{\gamma},z_2) J_{k,S}(\b{\eta}_r, \b{\omega}_r(\b{\gamma} z_2)) \det(\omega_r(\tau'_1) - \omega_r(\tau'_2)^{-1})^k\\
\cdot\mathbf{e_a}(-\tr(S[\omega_r(w'_2)\omega_r(\tau'_2)^{-1} - \omega_r(w'_1)]  (\omega_r(\tau'_2)^{-1} -\omega_r(\tau'_1) )^{-1}  )).
\end{multline}
\end{lem}
\begin{proof}
It follows from the cocycle relation that $J_{k,S}(\b{\tau_r} ((\b{\xi} \times 1_{2m-2r}) \b{\beta} \times \b{\gamma}), \diag[z_1\, z_2])$ is equal to,
\[
J_{k,S}(\b{\tau_r} ,((\b{\xi} \times 1_{2m-2r}) \b{\beta} \times \b{\gamma}) \cdot \diag[z_1\, z_2]) \cdot
J_{k,S}( (\b{\xi} \times 1_{2m-2r}) \b{\beta} \times \b{\gamma}), \diag[z_1\, z_2]).
\]
Note that 
$$
((\b{\xi} \times 1_{2m-2r}) \b{\beta} \times \b{\gamma}) \cdot \diag[z_1\, z_2] = \diag[(\b{\xi} \times 1_{2m-2r}) \b{\beta} z_1, \b{\gamma} z_2],
$$
and so
\begin{align*}
J_{k,S}( (\b{\xi} \times 1_{2m-2r}) \b{\beta} \times \b{\gamma})&, \diag[z_1\, z_2])\\
&\hspace{-0.5cm} = J_{k,S}( (\b{\xi} \times 1_{2m-2r})  \times 1_{2n}, \diag[\b{\beta}z_1\, \b{\gamma} z_2]) \times J_{k,S}(\b{\beta} \times \b{\gamma}, \diag[z_1 \, z_2])\\
&\hspace{-0.5cm} = J_{k,S}((\b{\xi} \times 1_{2m-2r}), \b{\beta} z_1) J_{k,S}(1_{2n},\b{\gamma}z_2) J_{k,S}(\b{\beta}, z_1) J_{k,S}(\b{\gamma},z_2)\\
&\hspace{-0.5cm} =J_{k,S}((\b{\xi} \times 1_{2m-2r}), \b{\beta} z_1)  J_{k,S}(\b{\beta}, z_1) J_{k,S}(\b{\gamma},z_2).
\end{align*}
Putting the last few calculations together we get that $J_{k,S}(\b{\tau_r} ((\b{\xi} \times 1_{2m-2r}) \b{\beta} \times \b{\gamma}), \diag[z_1\, z_2])$ is equal to 
$$J_{k,S}(\b{\tau_r} ,\diag[(\b{\xi} \times 1_{2m-2r}) \b{\beta} z_1, \b{\gamma} z_2]) \cdot
J_{k,S}((\b{\xi} \times 1_{2m-2r}), \b{\beta} z_1)  J_{k,S}(\b{\beta}, z_1) J_{k,S}(\b{\gamma},z_2).$$

Since $\b{\xi} \times 1_{2m-2r} \in \mathbf{P}^{m,r}$,  
$$
J_{k,S}((\b{\xi} \times 1_{2m-2r}), \b{\beta} z_1) \stackrel{\eqref{eq:J_at_P}}{=} (\lambda^m_{r,l}(\b{\xi} \times 1_{2m-2r}))^k J_{k,S}(\b{\pi}_r(\b{\xi} \times 1_{2m-2r}),\b{\omega}_r(\b{\beta}z_1)) = J_{k,S}(\b{\xi}, \b{\omega}_r(\b{\beta}z_1)).
$$ 

Moreover, by our previous computations, 
\begin{multline*}
J_{k,S}(\b{\tau_r} ,\diag[(\b{\xi} \times 1_{2m-2r}) \b{\beta} z_1, \b{\gamma} z_2]) =\\
=\mathbf{e_a}(-\tr(S[\b{\omega}_r(w'_2)\omega_r(\tau'_2)^{-1} - \b{\omega}_r(w'_1) ]  (\omega_r(\tau'_2)^{-1} -\omega_r(\tau'_1) )^{-1}  ))\\
\cdot J_{k,S}(\b{\eta}_r, (\omega_r(\tau'_2), \b{\omega}_r(w'_2))) \det(\omega_r(\tau'_1) - \omega_r(\tau'_2)^{-1})^k,
\end{multline*}
where we have set $(\b{\xi} \times 1_{2m-2r}) \b{\beta} z_1 = (\tau_1',w_1')$ and $\b{\gamma}z_2 = (\tau_2', w_2')$. 

Hence, with the above notation,
\begin{multline}
J_{k,S}(\b{\tau_r} ((\b{\xi} \times 1_{2m-2r}) \b{\beta} \times \b{\gamma}), \diag[z_1\, z_2])=\\
=J_{k,S}(\b{\xi}, \b{\omega}_r(\b{\beta}z_1))J_{k,S}(\b{\beta}, z_1) J_{k,S}(\b{\gamma},z_2) J_{k,S}(\b{\eta}_r, \b{\omega}_r(\b{\gamma} z_2)) \det(\omega_r(\tau'_1) - \omega_r(\tau'_2)^{-1})^k\\
\cdot\mathbf{e_a}(-\tr(S[\omega_r(w'_2)\omega_r(\tau'_2)^{-1} - \omega_r(w'_1)]  (\omega_r(\tau'_2)^{-1} -\omega_r(\tau'_1) )^{-1}  )).
\end{multline}
\end{proof}
The considerations above and the identity
$$\delta(g\tau)=\delta(\tau)|j(g,\tau)|^{-2}\, \qquad\mbox{for } g\in G^n,\, \tau\in\mathbb{H}_n,$$
lead to the following formula:
\begin{multline}\label{diag_delta}
	\delta (\b{\tau_r} ((\b{\xi} \times 1_{2m-2r}) \b{\beta} \times \b{\gamma}), \diag[z_1\, z_2])=\delta (\tau_r ((\xi \times 1_{2m-2r}) \beta \times\gamma), \diag[\tau_1\, \tau_2])\\
	=\delta(\beta\tau_1)\delta(\gamma\tau_2)|j(\xi,\omega_r(\beta\tau_1))j(\eta_r,\omega_r(\gamma\tau_2))\det(\xi\omega_r(\beta\tau_1)-\omega_r(\gamma\tau_2)^{-1})|^{-2}.
\end{multline}
\subsection{Decomposing the Eisenstein series I; the non-full rank part}

Thanks to the strong approximation (Lemma \ref{strong_approx_Jac}) we can pick an element $\b{\rho} = 1_{H} \rho \in \Jac^{m+n}(F) \cap K^{m+n}[\f{b},\f{c}] \b{\sigma}$ such that $a_{\sigma_v \rho_v^{-1}} - 1 \in M_{m+n,m+n}(\f{c})_v$ for all $v | \f{c}$. If we now write $\b{\rho} = \b{w}\b{\sigma}$ with $\b{w} \in K^{m+n}[\f{b},\f{c}]$, then for $y \in \Jac_{\a}$ such that $ y \mathbf{i}_0 = z$,
\begin{align*}
E(y \b{\sigma}^{-1}) &= E(\b{\rho}^{-1} \b{w}  y) = E(\b{w}  y) = E(\b{w}_{\h}\b{w} _{\a} y) = \chi(\det (d_{\b{w}_{\h}}))^{-1} E(\b{w} _{\a} y)\\ 
 &= \chi(\det (d_{\b{w}_{\h}}))^{-1} (E|_{k,S}\b{w} _{\a} y)(\b{i}_0).
\end{align*}

But since $\b{\sigma}_{\a}$ is trivial, $\b{w} _{\a} = \b{\rho}_{\a}$ and, by the condition on $\rho$, $\chi(\det (d_{\b{w}_{\h}})) = \chi(\det (d_{\b{\sigma}_{\h}})^{-1}$. In particular, we see that the adelic Eisenstein series $E(x\b{\sigma}^{-1},s)$ corresponds to the classical series $(E |_{k,S} \b{\rho}) (z,s)$. \newline

Let $y, \b{\rho}$ be as above and put 
$$\varepsilon_r(z,s):=\sum_{\alpha\in A_r} p_{\alpha}(z), \qquad 
p_{\alpha}(z):=\phi(\alpha y\b{\sigma}^{-1},s)J_{k,S}(y,\b{i}_0),$$ 
where $A_r:=\b{P}^{m+n}(F)\back\b{P}^{m+n}(F)\b{\tau}_r\iota_A(\Jac^m(F)\times\Jac^n(F))$. Then 
$$(E |_{k,S} \b{\rho}) (z,s)=\sum_{0\leq r\leq n} \varepsilon_r(z,s),$$ 
and for a fixed $r$ each $\alpha\in A_r$ is of the form $\alpha (\b{\xi}, \b{\beta},\b{\gamma}):=\b{\tau}_r((\b{\xi}\times 1_H 1_{2(m-r)})\b{\beta}\times\b{\gamma})$ for some $\b{\xi},\b{\beta},\b{\gamma}$ as in Lemma \ref{lem:coset_decomp}.

 The following Lemma generalizes Lemma 2.2 in \cite{Sh95} to the Jacobi case.

\begin{lem}\label{lem:in_product} Let $f$ be a cuspidal Siegel-Jacobi form on $\H_{n,l}$ of weight $k \in \Z^{\a}$ and $g(z)$ a function on $\mathcal{H}_{n,l}$ depending only on $\b{\omega}_r(z)$ and $\Im(z) := (\Im(\tau), \Im(w))$ for some $r \in \N$ with $0 \leq r < n$. If for a congruence subgroup $\b{\Gamma}$ we have $g |_{k,S} \gamma = g$ for every $\gamma \in \b{P}^{n,r}(F) \cap \tau \b{\Gamma} \tau^{-1}$ with $\tau \in \b{G}^{n,l}(F)$, then
\[
<\sum_{ \gamma \in R} g |_{k,S} \gamma , f > = 0
\] 
for any set $R$ of representatives for  $\b{P}^{n,r}(F) \cap \tau \b{\Gamma} \tau^{-1} \setminus \tau \Gamma$.
\end{lem}

\begin{proof} The proof is almost identical to the one of \cite[Lemma 2.2]{Sh95}, and we only need to establish that 
\[
\int_{X \times U} f \left( \begin{pmatrix}  \tau_1 & \tau_2 \\
\T{\tau_2} & \tau_4
 \end{pmatrix} , (w_1 \,\, w_2 )\right) dx_2 dx_4 du_2 = 0
\]
for $\tau_j = x_j + i y_j$, $\tau_1\in\mathbb{H}_r^{\a}$, $w_j = u_j+ iv_j$, $w_1\in M_{l,r}(\C)^{\a}$. This can be shown by considering the Fourier expansion of $f$ at infinity. Namely, if 
$$f(\tau, w) = \sum_{T,R} c(T,R) \mathbf{e_a}(\tr (T \tau + \transpose{R} w))$$
and we put $T=\(\begin{smallmatrix}
t_1 & \T{t_2}/2\\ t_2/2 & t_4
\end{smallmatrix} \)$, $R=\(\begin{smallmatrix}
r_1 & r_2\end{smallmatrix} \)$ with $t_i, r_i$ of suitable size, then 
$$\tr (T \tau + \transpose{R} w)=\tr (\T{t_2}\T{x_2}/2 +t_2x_2/2+t_4x_4+\transpose{r}_2u_2) +M = \tr (t_2x_2)+\tr (t_4x_4)+\tr (\transpose{r}_2u_2)+M,$$
where $M$ is independent of $x_2, x_4, u_2$.
In this way 
\begin{multline*}
\int_{X \times U} f \left( \begin{pmatrix}  \tau_1 & \tau_2 \\
\T{\tau_2} & \tau_4
\end{pmatrix} , (w_1 \,\, w_2 )\right) dx_2 dx_4 du_2 \\
= \sum_{T,R} c(T,R)\mathbf{e_a}(M)\int_{X \times U}\mathbf{e_a}(\tr (t_2x_2)+\tr (t_4x_4)+\tr (\transpose{r}_2u_2))dx_2 dx_4 du_2=0
\end{multline*}
since $c(\(\begin{smallmatrix}
t_1 & \\  & 0
\end{smallmatrix}\), \transpose{\binom{r_1}{0}})=0$, since $f$ is a cusp form.
\end{proof}

\begin{prop}\label{prop_epsilon_r} Let $n\leq m$, $z_1 \in \mathcal{H}_{m,l}$ and $z_2 \in \mathcal{H}_{n,l}$. For a cusp form $f$ on $\H_{n,l}$ of weight $k$, $0\leq r<n$ and for $s$ large enough, we have 
$$<\varepsilon_r(\diag[z_1,z_2],s), f(z_2)> = 0.$$
\end{prop}

\begin{proof}
	Let $z=\diag[z_1,z_2]\in\H_{m+n,l}$ and fix $r\in\{ 0,1,\ldots ,n-1\}$. 
	Put 
	$$\b{D}':=\{ x\in K^{m+n}[\f{b},\f{c}]:\det(d_x)_v-1\in\f{c}_v\mbox{ for every } v|\f{c}\} .$$
	Let $\b{\Gamma}$ be a congruence subgroup of $\Jac^n(F)$ such that $\iota_A(1_H1_{2m}\times \b{\Gamma})\subset\b{\sigma}^{-1} \b{D}'\b{\sigma}$. By the definition of $\phi$, for any $d'\in K^{m+n}[\f{b},\f{c}]$
	$$\phi(xd',s)=\chi_{\f{c}}(\det (d_{d'}))^{-1} J_{k,S}(d',\b{i}_0)^{-1}\phi(x,s),$$
	and thus $p_{\alpha}|_k \alpha' =p_{\alpha\alpha'}$ for $\alpha'\in\Jac^{m+n}(F)\cap\b{\sigma}^{-1} \b{D}'\b{\sigma}$. Further,	write $\Jac^n(F)=\bigsqcup_{\b{\tau}\in T} \b{P}^{n,r}(F)\b{\tau}\b{\Gamma}$, so that
	$$\varepsilon_r=\sum_{\b{\xi},\b{\beta},\b{\gamma}} p_{\alpha(\b{\xi},\b{\beta},\b{\gamma})}=\sum_{\b{\xi},\b{\beta}}\sum_{\b{\tau}\in T}\sum_{\gamma\in R_{\b{\tau}}} p_{\alpha(\b{\xi},\b{\beta},\b{\tau})}|_k\iota_A(1_H1_{2m}\times\b{\tau}^{-1})|_k\iota_A(1_H1_{2m}\times\gamma),$$
	where $R_{\b{\tau}}:=(\b{P}^{n,r}(F)\cap\b{\tau}\b{\Gamma}\b{\tau}^{-1})\back\b{\tau}\b{\Gamma}$.
	We will check that for each $\b{\tau}\in T$,
	$$g_{\b{\tau}}:=\sum_{\b{\xi},\b{\beta}} p_{\alpha(\b{\xi},\b{\beta},\b{\tau})}|_k\iota_A(1_H1_{2m}\times\b{\tau}^{-1})$$
	satisfies the conditions of Lemma \ref{lem:in_product}.
	
	Fix $\b{\tau}\in T$ and take $\b{\eta}\in \b{P}^{n,r}(F)\cap\b{\tau}\b{\Gamma}\b{\tau}^{-1}$. We will show that 
\begin{equation}\label{eq:eta}
\sum_{\b{\xi},\b{\beta}} p_{\alpha(\b{\xi},\b{\beta},\b{\tau})}|_k\iota_A(1_H1_{2m}\times\b{\tau}^{-1}\b{\eta}\b{\tau})= \sum_{\b{\xi},\b{\beta}} p_{\alpha(\b{\xi},\b{\beta},\b{\tau})},
\end{equation}
	which in turn immediately implies
	$$\sum_{\b{\xi},\b{\beta}} p_{\alpha(\b{\xi},\b{\beta},\b{\tau})}|_k\iota_A(1_H1_{2m}\times\b{\tau}^{-1}\b{\eta})= \sum_{\b{\xi},\b{\beta}} p_{\alpha(\b{\xi},\b{\beta},\b{\tau})}|_k\iota_A(1_H1_{2m}\times\b{\tau}^{-1}).$$
	First of all, because $\b{\tau}^{-1}\b{\eta}\b{\tau}\in\b{\Gamma}$,
	$$p_{\alpha(\b{\xi},\b{\beta},\b{\tau})}|_k\iota_A(1_H1_{2m}\times\b{\tau}^{-1}\b{\eta}\b{\tau})=p_{\alpha(\b{\xi},\b{\beta},\b{\eta}\b{\tau})},$$
	where
\begin{align*}
\alpha(\b{\xi},\b{\beta},\b{\eta}\b{\tau})&=\b{\tau}_r((\b{\xi}\times 1_H 1_{2(m-r)})\b{\beta}\times\b{\eta}\b{\tau})\\
&=\b{\tau}_r(1_H1_{2m}\times\b{\eta})((\b{\xi}\times 1_H 1_{2(m-r)})\b{\beta}\times\b{\tau}).
\end{align*}
	Because $p_{\alpha}$ depends only on $\b{P}^{m+n}(F)\alpha$, in order to prove \eqref{eq:eta} it suffices to show that there exists $\b{\zeta}\in\Jac^r(F)$ such that 
\begin{equation}\label{eq:zeta}
	\alpha(\b{\xi},\b{\beta},\b{\eta}\b{\tau})\in\b{P}^{m+n}(F)\alpha(\b{\zeta}\b{\xi},\b{\beta},\b{\tau}).
\end{equation}
	Write $\b{\eta}=((\l_1'\, 0),\mu',\kappa')\eta$. By the same calculation as in the proof of Lemma \ref{lem:coset_decomp},
\begin{multline*}
\b{\tau}_r(1_H1_{2m}\times \b{\eta})\in\b{P}^{m+n}(F)\b{\tau}_r((-\mu'\T{e_r},(-\l_1'\, 0),0)1_{2m}\times 1_H\eta)\\
=\b{P}^{m+n}(F)\b{\tau}_r(1_H1_{2m}\times 1_H\eta)((-\mu'\T{e_r},(-\l_1'\, 0),0)1_{2m}\times 1_H1_{2n}).
\end{multline*} 
On the other hand, by \cite[Lemma 4.3]{Sh95}, there is $\zeta\in G^r(F)$ such that $\tau_r\iota_S(1_{2m}\times\eta)\in P^{m+n}(F)\tau_r\iota_S(\iota_S(\zeta\times 1_{2(m-r)})\times 1_{2n})$. Hence, \eqref{eq:zeta} holds for $\b{\zeta}=\zeta (-\mu'\binom{1_r}{0},-\l_1',0)$. This proves \eqref{eq:eta}, and thus also an invariance property for $g_{\b{\tau}}$.

It remains to show that $g_{\b{\tau}}(\diag[z_1,z_2],s)$ depends only on $s, z_1, \Im (z_2)$ and $\b{\omega}_r(z_2)$. Observe that whenever $\alpha y\b{\sigma}^{-1}=pw$ for some $p\in\b{P}^{n,0}(\A), w\in K^{n,0}$, then
\begin{align*}
\phi(\alpha y\b{\sigma}^{-1},s)J_{k,S}(y,\b{i}_0)&=
\chi(\det d_p)^{-1}\chi_{\f{c}}(\det (d_w)_{\f{c}})^{-1} J_{k,S}(w,\b{i}_0)^{-1}|\det d_p|_{\A}^{-2s} J_{k,S}(y,\b{i}_0)\\
&=\mu(\alpha_{\h}\b{\sigma}^{-1})\chi_{\a}(\det (d_{p})_{\a})^{-1} J_{k,S}(p,\b{i}_0)J_{k,S}(\alpha,z)^{-1}|\det d_p|_{\A}^{-2s},
\end{align*}
where we put $\mu(\alpha_{\h}\b{\sigma}^{-1}):=\chi_{\h}(\det (d_{p})_{\h})^{-1}\chi_{\f{c}}(\det (d_{w})_{\f{c}})^{-1}$. Moreover, because 
$$J_{k,S}(p,\b{i}_0)=\chi_{\a}(\det (d_{p})_{\a})|\det d_p|_{\a}^k\quad\mbox{ and }\quad |\det d_p|_{\A}^{-2s}=\delta(\alpha_{\a} z)^s N(\f{a}_0(\alpha\b{\sigma}^{-1}))^{2s},$$ 
we get
\begin{multline}\label{eq:ES_formula}
(E |_{k,S} \b{\rho}) (z,s)=\sum_{0\leq r\leq n}\sum_{\alpha\in A_r}\phi(\alpha y\b{\sigma}^{-1},s)J_{k,S}(y,\b{i}_0)\\
=\sum_r\sum_{\alpha} N(\mathfrak{a}_0(\alpha\b{\sigma}^{-1}))^{2s}\mu(\alpha_{\h}\b{\sigma}^{-1})J_{k,S}(\alpha_{\a},\diag[z_1,z_2])^{-1}\delta(\alpha_{\a}\diag[z_1,z_2])^{s-k/2}.
\end{multline}
From this and the formulas \eqref{diag.aut.fac.}, \eqref{diag_delta} we see that $g_{\b{\tau}}$ depends only on $s, z_1, \Im (z_2)$ and $\b{\omega}_r(z_2)$. This finishes the proof.
\end{proof}
\subsection{Decomposing the Eisenstein series II; the full rank part } We start with the following auxiliary lemma.

\begin{lem}\label{Gauss_integral}
	For a symmetric positive definite matrix $S\in Sym_l(\R)$, $X\in M_{l,n}(\R), A\in Sym_n(\C)$ and a scalar $a\in\C^{\times}$, we have the following formula:
	\begin{multline*}
	\int_{\R^{l\times n}} \exp(a\tr(-S[X]A+RXA)) dX\\
	=(\det A)^{-l/2}\({\pi\over a}\)^{nl/2}(\det S)^{-n/2}\exp \({a\over 4}\tr(S^{-1}[\T{R}]A)\).
	\end{multline*}
\end{lem}
\begin{proof}
	Write $A=UD\T{U}$, where $U$ is unitary, and $D=\( \begin{smallmatrix} d_1 & & \\ & \ddots & \\ & & d_n\end{smallmatrix}\)$ diagonal positive definite. Let $\tilde{X}=XU$ and write $\tilde{X}=(\tilde{x_1}\ldots \tilde{x_n}),\, \tilde{x_i}\in M_{l,1}(\R)$. Then:
	$$\int \exp(a\tr(-S[X]A+RXA)) dX=(\det U)^{-l}\int \exp(a\tr(-S[\tilde{X}]D)+a\tr(\T{U}R\tilde{X}D)) dX$$
	Further, substitute $J:=\T{U}R$ and write $J=\(\begin{smallmatrix} j_1\\\vdots \\j_n\end{smallmatrix}\),\, j_i\in M_{1,l}(\C)$, so that the integral is equal to
	\begin{align*}
	(\det U)^{-l}& \prod_{i=1}^n\int_{\R}\exp(ad_i(-S[\tilde{x}_i]+j_i\tilde{x}_i)) d\tilde{x}_i\\
	&= (\det U)^{-l} \prod_{i=1}^n\( \({\pi\over ad_i}\)^{l/2}(\det S)^{-1/2}\exp\({1\over 4ad_i} S^{-1}[\T{j_i}](ad_i)^2\)\)\\
	&=(\det A)^{-l/2}\({\pi\over a}\)^{nl/2}(\det S)^{-n/2}\exp \({a\over 4}\tr(S^{-1}[\T{R}]A)\).
	\end{align*}
	To compute the last integral we used a formula for an integral of a shifted $l$-dimensional Gaussian function.
\end{proof}
Now we can prove,
\begin{lem}[Reproducing Kernel]\label{repr_kernel} Let $f$ be a holomorphic function on $\mathcal{H}_{n,l}$ of weight $k\in\Z^{\a}$ such that $\Delta_{S,k}(z) f(z)^2$ is bounded. Then for $s \in \C^{\a}$  satisfying $\Re(s_{\nu})\geq 0$, $\Re(s_{\nu})+k_{\nu}-l/2>2n$ for each $\nu\in\a$, and for $(\zeta,\rho) \in \mathcal{H}_{n,l}$ we have
\begin{multline*}
	\tilde{c}_{S,k}(s) \det (\Im(\zeta))^{-s} f(\zeta,\rho)=\\
	\int_{\mathcal{H}_{n,l}}\hspace{-0.6cm}  f(\tau,w) \overline{\mathbf{e_a}(-\tr(S[w-\bar{\rho}](\tau-\bar{\zeta})^{-1})) \det(\tau-\bar{\zeta})^{-k}} |\!\det(\tau-\bar{\zeta})|^{-2s}\! \det(\Im(\tau))^{s} \Delta_{S,k}(z) d(\tau,w), 
\end{multline*}
where
$$\tilde{c}_{S,k}(s) =\prod_{\nu\in\a} \det(2S_{\nu})^{-n}(-1)^{n(l+k_{\nu}/2)}2^{n(n+3)/2-4s_{\nu}-nk_{\nu}}\pi^{n(n+1)/2}\frac{\Gamma_n(s_{\nu}+k_{\nu}-{l\over 2}-{n+1\over 2})}{\Gamma_n(s_{\nu}+k_{\nu}-{l\over 2})}$$ 
and $\Gamma_n(s):=\pi^{n(n-1)/4}\prod_{i=0}^{n-1}\Gamma (s-{i\over 2})$.
\end{lem}

\begin{proof} We remark that a very similar integral was computed in the proof of \cite[Lemma 2.8]{Ar94}. Since the above lemma is only implicit in the form stated above in \cite{Ar94}, we decided to provide the full proof for the sake of completeness. 
	
To compute the above integral we plug in $f$ in its Fourier expansion:
\begin{multline*}
	\int_{\mathcal{H}_{n,l}}\sum_{T,R} c(T,R)\mathbf{e_a}(\tr(T\tau +Rw))\overline{\mathbf{e_a}(-\tr(S[w-\bar{\rho}](\tau-\bar{\zeta})^{-1})) \det(\tau-\bar{\zeta})^{-k}}\\
	\cdot |\det(\tau-\bar{\zeta})|^{-2s}\det(\Im(\tau))^{s+k} \mathbf{\exp_a}(-4\pi\tr(S[\Im(w)])\Im(\tau)^{-1}) d(\tau,w).
\end{multline*}
We integrate first against the variables of $w=u+iv$. Note that
\begin{align*}
\mathbf{e_a}&(-\tr(S[w-\bar{\rho}](\tau-\bar{\zeta})^{-1}))\\
&\hspace{-0.2cm} =\mathbf{e_a}(-\tr((S[u]+\T{u}S(iv-\bar{\rho})+\T{(iv-\bar{\rho})}Su+S[iv-\bar{\rho}])(\tau-\bar{\zeta})^{-1})))\\
&\hspace{-0.2cm} =\mathbf{e_a}(-\tr((S[u](\tau-\bar{\zeta})^{-1})))\mathbf{e_a}(-2\tr(\T{(iv-\bar{\rho})}Su(\tau-\bar{\zeta})^{-1}))\mathbf{e_a}(-\tr(S[iv-\bar{\rho}](\tau-\bar{\zeta})^{-1})
\end{align*}
Put $A:=-i(\bar{\tau}-\zeta)^{-1}$. A part of the expression under the integral that contains a variable $u$ equals
\begin{multline*}
\mathbf{\exp_a}(2\pi i \tr((S[u](\bar{\tau}-\zeta)^{-1}-2\T{(iv+\rho)}Su(\bar{\tau}-\zeta)^{-1}+Ru)))\\
=\mathbf{\exp_a}(2\pi \tr(-S[u]A+(2\T{(iv+\rho)}S+A^{-1}iR)uA)).
\end{multline*}
Therefore, after setting $\b{R} := \R^{\mathbf{a}}$, we obtain by Lemma \ref{Gauss_integral},
\begin{align*}
\int_{\b{R}^{l\times n}}& \mathbf{\exp_a}(2\pi \tr(-S[u]A+(2\T{(iv+\rho)}S+A^{-1}iR)uA)) du\\
&\hspace{-0.6cm} =(\det A)^{-l/2} 2^{-nl/2}(\det S)^{-n/2}\mathbf{\exp_a}\! \(\! 2\pi\tr(S^{-1}[S(iv+\rho)+{i\over 2}\T{R}A^{-1}]A)\!\)\\
&\hspace{-0.6cm} =(\det A)^{-l/2} 2^{-nl/2}(\det S)^{-n/2}\mathbf{\exp_a}\! \(\! 2\pi\tr(S[iv+\rho]A+iR(iv+\rho)-{1\over 4}S^{-1}[\T{R}]A^{-1})\!\)\! ,
\end{align*}
where by $(\det A)^{-l/2} 2^{-nl/2}$ we understand $\prod_{\nu\in\a}((\det A_{\nu})^{-l/2} 2^{-nl/2})$; we take this convention for the rest of the proof.
After this integration a part that contains $v$ equals
\begin{align*}
\mathbf{\exp_a}&(-2\pi\tr(Rv))\mathbf{\exp_a}(-4\pi\tr(S[v]\Im(\tau)^{-1}))\mathbf{\exp_a}(-2\pi\tr(S[iv+\rho]A)\\
&\hspace{0.4cm} \cdot\mathbf{\exp_a} \(2\pi\tr(S[iv+\rho]A+R(-v+i\rho))\)\\ &=\mathbf{\exp_a}(-4\pi\tr(S[v]\Im(\tau)^{-1}+Rv))\mathbf{e_a} (\tr(R\rho))\\
&=\mathbf{\exp_a}(4\pi\tr((-S[v]+(-\Im(\tau))Rv))\Im(\tau)^{-1})\mathbf{e_a} (\tr(R\rho)).
\end{align*}
Using Lemma \ref{Gauss_integral} again,
\begin{multline*}
\int_{\b{R}^{l\times n}} \mathbf{\exp_a}(4\pi\tr((-S[v]+(-\Im(\tau))Rv)\Im(\tau)^{-1})) dv\\
=(\det \Im(\tau))^{l/2} 2^{-nl}(\det S)^{-n/2}\mathbf{\exp_a}(\pi\tr(S^{-1}[\T{R}]\Im(\tau))).
\end{multline*}
Now, joining all the pieces together, we get
\begin{align*}
\int_{\mathcal{H}_{n,l}}&  f(\tau,w) \overline{\mathbf{e_a}(-tr(S[w-\bar{\rho}](\tau-\bar{\zeta})^{-1})) \det(\tau-\bar{\zeta})^{-k}} |\det(\tau-\bar{\zeta})|^{-2s} \det(\Im(\tau))^{s}\\
&\hspace{0.4cm} \cdot\Delta_{S,k}(z) d(\tau,w)\\
&=2^{-nl/2}\det (2S)^{-n}\sum_R \mathbf{e_a} (\tr(R\rho)) \int_{\mathbb{H}^{\mathbf{a}}_n}   
(\det A)^{-l/2} \det(\bar{\tau}-\zeta)^{-k} |\det(\zeta-\bar{\tau})|^{-2s}\\
&\hspace{0.3cm}\cdot\det(\Im(\tau))^{s+k+l/2} \mathbf{\exp_a}(\pi\tr(S^{-1}[\T{R}]\Im(\tau)))\mathbf{\exp_a}\(-{\pi\over 2}\tr(S^{-1}[\T{R}]i(\bar{\tau}-\zeta))\)\\
&\hspace{0.3cm}\cdot\sum_T c(T,R)\mathbf{e_a}(\tr(T\tau)) d\tau\\
&=2^{-nl/2}\det (2S)^{-n}\sum_R \mathbf{e_a} (\tr\( R\rho +{1\over 4}S^{-1}[\T{R}]\zeta\))\\
&\hspace{0.3cm}\cdot\int_{\mathbb{H}^{\mathbf{a}}_n} \det(\zeta-\bar{\tau})^{l/2-k}(-1)^{n(k+l/2+l/4)}|\det(\zeta-\bar{\tau})|^{-2s}\det(\Im(\tau))^{s+k-l/2}\\
&\hspace{0.3cm}\cdot\mathbf{e_a}\(-{1\over 4}\tr(S^{-1}[\T{R}]\tau)\)\sum_T c(T,R)\mathbf{e_a}(\tr(T\tau))\det(\Im(\tau))^l d\tau ,
\end{align*}
By the ``classical'' reproducing kernel formula for holomorphic functions on the Siegel upper half space as stated for example in \cite[Lemma 4.7]{Sh95},
\begin{align*}
\int_{\mathbb{H}^{\mathbf{a}}_n} &\det(\zeta-\bar{\tau})^{l/2-k}(-1)^{n(k+l/2+l/4)}|\det(\zeta-\bar{\tau})|^{-2s}\det(\Im(\tau))^{s+k-l/2}\\
&\hspace{0.4cm}\cdot\mathbf{e_a}\(-{1\over 4}\tr(S^{-1}[\T{R}]\tau)\)\sum_T c(T,R)\mathbf{e_a}(\tr(T\tau))\det(\Im(\tau))^l d\tau\\
&=\frac{\tilde{c}_{S,k}(s)}{2^{-nl/2}\det (2S)^{-n}} \mathbf{e_a}\(-{1\over 4}\tr(S^{-1}[\T{R}]\zeta)\)\det(\Im(\zeta))^{-s}  \sum_T c(T,R)\mathbf{e_a}(\tr(T\zeta)),
\end{align*}
where $\tilde{c}_{S,k}(s)$ is as in the hypothesis. This in particular shows that 
\begin{multline*}
\int_{\mathcal{H}_{n,l}}  f(\tau,w) \overline{\mathbf{e_a}(-tr(S[w-\bar{\rho}](\tau-\bar{\zeta})^{-1})) \det(\tau-\bar{\zeta})^{-k}} |\det(\tau-\bar{\zeta})|^{-2s} \det(\Im(\tau))^{s}\\
\cdot \Delta_{S,k}(z) d(\tau,w)=
\tilde{c}_{S,k}(s)\det(\Im(\zeta))^{-s}\sum_{T,R}c(T,R)\mathbf{e_a} (\tr( R\rho ))\mathbf{e_a}(\tr(T\zeta)),
\end{multline*}
which concludes the proof.
\end{proof}

In order to proceed further we introduce the following notation, taken from \cite[equation (4.5)]{Sh95}. We have that $G^n(\mathbb{A}) = D^n[\mathfrak{b}^{-1}, \mathfrak{b}]W D^n[\mathfrak{b}^{-1}, \mathfrak{b}]$ with 
$$W = \left\{ \diag[q, \tilde{q}] : \,\,q \in GL_n(\mathbb{A}_{\mathbf{h}}) \cap \prod_{v \in \mathbf{h}} GL_n(\mathfrak{o}_v) \right\} ,$$
that is, any element $x \in G^n(\mathbb{A})$ may be written as $x = \gamma_1 \diag[q, \tilde{q}] \gamma_2$ with $\gamma_1,\gamma_2 \in D^n[\mathfrak{b}^{-1},\mathfrak{b}]$ and $q \in W$. We define $\ell_0(x)$ to be the ideal associated to $\det(q)$, $\ell_1(x) := \prod_{v \nmid\mathfrak{c}}\ell_0(x)_v$ and set $\ell(x)$ for the norm of the ideal $\ell_0(x)$. With this notation we have, 

\begin{lem} For $z_1 \in \mathcal{H}_{m,l}$ and $z_2 \in \mathcal{H}_{n,l}$,
\begin{align*}
\varepsilon_n(\diag[z_1, z_2],s) &= \sum_{\b{\beta} \in \mathbf{B}} \sum_{\b{\xi} \in\b{X}} N(\f{b})^{-2ns}N(\f{a}_0(\beta))^{2s}\ell(\xi)^{-2s}\chi_{\h}(\theta^n)\chi[\beta]\chi^{*} (\ell_1(\xi)) \chi_{\f{c}}(\det(d_{\xi}))^{-1}\\
&\hspace{0.5cm}\cdot J_{k,S}(\b{\xi}, \b{\omega}_n(\b{\beta}z_1))^{-1} J_{k,S}(\b{\beta}, z_1)^{-1}  J_{k,S}(\b{\eta}_n,  z_2)^{-1} \det(\omega_n(\tau'_1) - \tau_2^{-1})^{-k}\\
&\hspace{0.5cm}\cdot\mathbf{e_a}(\tr(S[w_2 \tau_2^{-1} - \b{\omega}_n(w'_1) ]  (\tau_2^{-1} -\omega_n(\tau'_1) )^{-1} ))(\delta(\beta\tau_1) \delta(\tau_2))^{s-k/2}\\
&\hspace{0.5cm}\cdot |j(\xi, \omega_n(\beta \tau_1))j(\eta_n,\tau_2) \det(\omega_n(\tau_1') - \tau_2^{-1})|^{-2s + k}, 
\end{align*}
where we have set $(\b{\xi} \times 1_{2m-2n}) \b{\beta} z_1 = (\tau_1',w_1')$.
\end{lem}

\begin{proof}
The statement follows from the explicit computation of the factors occurring in the formula \eqref{eq:ES_formula}. Recall that we have already computed the values of the automorphy factor and $\delta$ in \eqref{diag.aut.fac.}, \eqref{diag_delta}. Therefore it suffices to find $\mathfrak{a}_0(\alpha\b{\sigma}^{-1})$ and $\mu(\alpha_{\h}\b{\sigma}^{-1})$ for $\alpha=\b{\tau}_n\iota_A(\iota_A(\b{\xi}\times 1_H 1_{2(m-n)})\b{\beta}\times\b{\gamma})$ with $\b{\xi}\in\b{X}, \b{\beta}\in\b{B}$ as in Lemma \ref{key_lemma}. 
Observe though that neither $\mathfrak{a}_0$ nor $\mu$ depends on the elements from Heisenberg group. Moreover, because for any symplectic matrix $g$ we have $gH=Hg$, the symplectic factors of the representatives given in Lemma \ref{key_lemma} are exactly the same as the representatives provided in \cite[Lemma 4.4]{Sh95}. Hence, it is clear that the formulas for $\mathfrak{a}_0$ and $\mu$ have to be the same as the ones computed in \cite[Lemma 4.6]{Sh95}. That is:
$$\mathfrak{a}_0(\alpha\b{\sigma}^{-1})=\f{b}^{-n}\mathfrak{a}_0(\beta)\ell_0(\xi)^{-1},\quad
\mu(\alpha_{\h}\b{\sigma}^{-1})=\chi_{\h}(\theta^n)\chi[\beta]\chi^*(\ell_1(\xi))\chi_{\mathfrak{c}}(\det(d_{\xi}))^{-1}.$$

\end{proof}

We now consider an $f \in S_k(\mathbf{\Gamma},\chi^{-1})$ where $\b{\Gamma} := \b{G}^n \cap \b{D}$ with
\[
\b{D}:=\{ (\l,\mu,\kappa)x\in C[\f{o}, \f{b}^{-1}, \f{b}^{-1}]D[\f{b}^{-1}\f{e},\f{bc}]:(a_x-1_n)_v\in M_{n,n}(\f{e}_v)\mbox{ for every } v|\f{e}\}. 
\]
We set $\nu_{\mathfrak{e}} = 2$ if $\mathfrak{e} | 2$, and $1$ otherwise. Then by using the standard unfolding trick regarding the $z_2$ variable and setting $A := \b{\Gamma} \setminus \mathcal{H}_{n,l}$, we obtain
\begin{align*}
<&\varepsilon_n(\diag[z_1,z_2],s), f(z_2)>\\
&= \nu_{\f{e}} vol(A)^{-1}\sum_{\b{\beta} \in \mathbf{B}} \sum_{\b{\xi} \in\b{X}} N(\f{b})^{-2ns}N(\f{a}_0(\beta))^{2s}\ell(\xi)^{-2s}\chi_{\h}(\theta^n)\chi[\beta]\chi^{*} (\ell_1(\xi)) \chi_{\f{c}}(\det(d_{\xi}))^{-1}\\
&\cdot J_{k,S}(\b{\xi}, \b{\omega}_n(\b{\beta}z_1))^{-1} J_{k,S}(\b{\beta}, z_1)^{-1} \delta(\beta\tau_1)^{s-k/2}|j(\xi, \omega_n(\beta \tau_1))|^{-2s + k}\\
& \int_{\mathcal{H}_{n,l}} J_{k,S}(\b{\eta}_n,  z_2)^{-1} \det(\omega_n(\tau'_1) - \tau_2^{-1})^{-k}
\mathbf{e_a}(\tr(S[w_2 \tau_2^{-1} - \b{\omega}_n(w'_1) ]  (\tau_2^{-1} -\omega_n(\tau'_1) )^{-1} )) \\
&\cdot\delta(\tau_2)^{s-k/2}|j(\eta_n,\tau_2) \det(\omega_n(\tau_1') - \tau_2^{-1})|^{-2s + k}f(z_2) \Delta_{S,k}(\tau_2,w_2) d(\tau_2,w_2).
\end{align*}

The integral on the right of the above formula is equal to
\begin{align*}
\int_{\mathcal{H}_{n,l}}\hspace{-0.4cm} & J_{k,S}(\b{\eta}_n, z_2)^{-1}(-1)^{nk}\det(\tau_2^{-1}-\omega_n(\tau'_1) )^{-k}
\mathbf{e}(\tr(S[w_2 \tau_2^{-1} - \b{\omega}_n(w'_1) ]  (\tau_2^{-1} -\omega_n(\tau'_1) )^{-1}  )) \\
 &\hspace{0.5cm}\cdot\delta(\eta_n\tau_2)^{s-k/2}  |\det(\tau_2^{-1}-\omega_n(\tau_1'))|^{-2s + k}f(z_2) \Delta_{S,k}(\tau_2,w_2) d(\tau_2,w_2)\\
&\hspace{-0.2cm} =\int_{\mathcal{H}_{n,l}} J_{k,S}(\b{\eta}_n,z_2)^{-1}J_{k,S}(\b{\eta}_n, \b{\eta}_n z_2)^{-1} (-1)^{nk}\det(-\tau_2-\omega_n(\tau'_1) )^{-k}\\
&\hspace{0.3cm}\cdot\mathbf{e}_{\mathbf{a}}(\tr(S[-w_2  - \b{\omega}_n(w'_1) ]  (-\tau_2 -\omega_n(\tau'_1) )^{-1} ))  \\
&\hspace{0.3cm}\cdot\delta(\eta_n\eta_n\tau_2)^{s-k/2}  |\det( -\tau_2-\omega_n(\tau_1') )|^{-2s + k}f|_{k,S}\b{\eta}_n(z_2) \Delta_{S,k}(\tau_2,w_2) d(\tau_2,w_2)\\
&\hspace{-0.2cm} =\int_{\mathcal{H}_{n,l}} J_{k,S}(-1_H 1_{2n},z_2)^{-1}\det(\tau_2+\omega_n(\tau'_1) )^{-k}
\mathbf{e}_{\mathbf{a}}(-\tr(S[w_2 +\b{\omega}_n(w'_1) ]  (\tau_2 +\omega_n(\tau'_1) )^{-1} ))  \\
&\hspace{0.3cm}\cdot\delta(-\tau_2)^{s-k/2}  |\det( \tau_2+\omega_n(\tau_1') )|^{-2s + k}f|_{k,S}\b{\eta}_n(z_2) \Delta_{S,k}(\tau_2,w_2) d(\tau_2,w_2)\\
&\hspace{-0.2cm} =\int_{\mathcal{H}_{n,l}}\overline{f|_{k,S}\b{\eta}_n(z_2)}\det(\tau_2+\omega_n(\tau'_1) )^{-k}
\mathbf{e}_{\mathbf{a}}(-\tr(S[w_2 +\b{\omega}_n(w'_1) ]  (\tau_2 +\omega_n(\tau'_1) )^{-1} ))\\
&\hspace{0.3cm} \cdot (-1)^{n(s+k/2)}\delta(\tau_2)^{s-k/2}  |\det( \tau_2+\omega_n(\tau_1') )|^{-2(s-k/2)} \Delta_{S,k}(\tau_2,w_2) d(\tau_2,w_2) .
\end{align*}

By Lemma \ref{repr_kernel}, this is equal to 
\begin{equation}\label{complex_conj}
(-1)^{n(s+k/2)}\overline{\tilde{c}_{S,k}(\bar{s}-k/2)\delta(\xi\omega_n(\beta\tau_1))^{-\bar{s}+k/2} f|_{k,S}\b{\eta}_n(-\overline{\b{\xi\omega}_n(\b{\beta} z_1)})}.
\end{equation}
Put $\delta_{n,k}:=\prod_{v\in\a}\delta_{v,n,k}$, where $\delta_{v,n,k}$ is equal to $1$ if $nk_v$ even and $-1$ otherwise, and
let $c_{S,k}(s):=\delta_{n,k}\tilde{c}_{S,k}(s)$. Then, because $\overline{\Gamma(\bar{s})}=\Gamma(s)$, the quantity \eqref{complex_conj} equals
$$(-1)^{n(s+k/2)}c_{S,k}(s-k/2)\delta(\xi\omega_n(\beta\tau_1))^{-s+k/2} \overline{f|_{k,S}\b{\eta}_n(-\overline{\b{\xi\omega}_n(\b{\beta} z_1)})} .$$
Hence, if we set $f^c(z):=\overline{f(-\bar{z})}$, where $-\bar{z} := (-\bar{\tau}, -\bar{w})$ for $z = (\tau,w)$, then
\begin{align*}
N&(\f{b})^{2ns}\chi_{\h}(\theta)^{-n}(-1)^{n(s-k/2)}c_{S,k}(s-k/2)^{-1} vol(A) <\varepsilon_n(\diag[z_1,z_2],s), f(z_2)>\\
&=\nu_{\f{e}} \sum_{\b{\beta} \in \mathbf{B}} \sum_{\b{\xi} \in\b{X}} N(\f{a}_0(\beta))^{2s}\ell(\xi)^{-2s}\chi[\beta]\chi^{*} (\ell_1(\xi)) \chi_{\f{c}}(\det(d_{\xi}))^{-1}J_{k,S}(\b{\beta}, z_1)^{-1}\\
&\hspace{0.5cm} J_{k,S}(\b{\xi}, \b{\omega}_n(\b{\beta}z_1))^{-1}\delta(\beta\tau_1)^{s-k/2}|j(\xi, \omega_n(\beta \tau_1))|^{-2s + k}\delta(\xi\omega_n(\beta\tau_1))^{-s+k/2}\\
&\hspace{0.5cm} ((f|_{k,S}\b{\eta}_n)^c|_{k,S}\b{\xi})(\b{\omega}_n(\b{\beta} z_1))J_{k,S}(\b{\xi}, \b{\omega}_n(\b{\beta}z_1))\\
&=\sum_{\b{\beta} \in \mathbf{B}} N(\f{a}_0(\beta))^{2s}\chi[\beta] J_{k,S}(\b{\beta}, z_1)^{-1} \(\frac{\delta(\beta\tau_1)}{\delta(\omega_n(\beta\tau_1))}\)^{s-k/2}\\
&\hspace{0.5cm}\sum_{\b{\xi} \in\b{X}} \ell(\xi)^{-2s}\chi^{*} (\ell_1(\xi)) \chi_{\f{c}}(\det(d_{\xi}))^{-1}
((f|_{k,S}\b{\eta}_n)^c|_{k,S}\b{\xi})(\b{\omega}_n(\b{\beta} z_1)).
\end{align*}
It is not hard to see that $\b{\eta}_n^{-1}\b{X}=\b{Y}\b{\eta}_n^{-1}$, where $\b{Y}=\b{G}^n(F)\cap\b{G}^n_{\a}\prod_{v\in\h}\b{Y}_v$ with
$$\b{Y}_v=\begin{cases}
\{ (\l,\mu,\kappa)y\in C_v[\f{b}^{-1}, \f{o}, \f{b}^{-1}] D_v^n[\f{bc},\f{b}^{-1}\f{c}]: a_y-1\in M_{n,n}(\f{e}_v)\} & \mbox{if } v|\f{e},\\
C_v[\f{b}^{-1}, \f{o}, \f{b}^{-1}] D_v^n[\f{b},\f{b}^{-1}\f{c}]Z_vC_v[\f{b}^{-1},\f{o},\f{b}^{-1}] D_v^n[\f{bc},\f{b}^{-1}] & \mbox{if } v|\f{e}^{-1}\f{c},\\
C_v[\f{b}^{-1},\f{o},\f{b}^{-1}]G^n(F_v)C_v[\f{b}^{-1},\f{o},\f{b}^{-1}] & \mbox{if } v\nmid\f{c},\\
\end{cases}$$
$$Z_v=\{ \diag[\tilde{q},q]: q\in\GL_n(F_v)\cap M_{n,n}(\f{c}_v)\} .$$
Moreover, it follows from Proposition \ref{Behaviour under complex conjugation} which we prove later that $(f|_{k,S}\b{\eta}_n)^c=f^c|_{k,S}\b{\eta}_n^{-1}$. Set
\begin{equation}\label{eq:D(z,s,g)}
\mathcal{D}(z,s,g):=\sum_{\b{\xi} \in\b{Y}} \ell'(\xi)^{-s}\chi^{*} (\ell'_1(\xi)) \chi_{\f{c}}(\det(a_{\xi}))^{-1}
(g|_{k,S}\b{\xi})(z),
\end{equation}
where $\ell'(\xi):=\ell(\eta_n\xi\eta_n^{-1}), \ell'_1(\xi):=\ell'_1(\eta_n\xi\eta_n^{-1})$. Then, using Proposition \ref{prop_epsilon_r}, formula \eqref{eq:KlingenES} and the fact that $N(\f{a}(\beta))=|\lambda^m_{n,l}(\beta)|_F$, we obtain
\begin{align}\label{main_inner_product}
N(\f{b})^{2ns}&\chi_{\h}(\theta)^{-n}(-1)^{n(s-k/2)}c_{S,k}(s-k/2)^{-1}vol(A)<(E |_{k,S} \b{\rho}) (\diag[z_1,z_2],s),f(z_2)>\nonumber\\
&\hspace{-0.5cm}=\nu_{\f{e}} \sum_{\b{\beta} \in \mathbf{B}} N(\f{a}_0(\beta))^{2s}\chi[\beta] J_{k,S}(\b{\beta}, z_1)^{-1} \(\frac{\delta(\beta\tau_1)}{\delta(\omega_n(\beta\tau_1))}\)^{s-k/2}\hspace{-0.3cm} \mathcal{D}(\b{\omega}_n(\b{\beta} z_1),2s,f^c)|_{k,S}\b{\eta}_n^{-1} .
\end{align}

\section{Shintani's Hecke Algebras and the standard $L$-function attached to Siegel-Jacobi modular forms} \label{Hecke algebra}

In this section we define Hecke operators acting on the space of Siegel-Jacobi modular forms. These operators were studied in the higher index case first by Shintani (unpublished), Murase \cite{Mu89,Mu91} and Murase and Sugano \cite{MS}. As we have indicated in the introduction this was done in the case of trivial level, and one of our contributions in this section is to define such operators also for non-trivial level. Furthermore, in this section we introduce the standard Dirichlet series which can be attached to a Hecke eigenform. Our main result here is an Euler product representation for this series, which extends previous results in \cite{MS} from index one to higher indices.

We start by fixing some notation. For the usual fractional ideals $\mathfrak{b},\mathfrak{c},\mathfrak{e}$ let 
$$\b{D}:=\{ (\l,\mu,\kappa)x\in C[\f{o}, \f{b}^{-1}, \f{b}^{-1}]D[\f{b}^{-1}\f{e},\f{bc}]:(a_x-1_n)_v\in M_n(\f{e}_v)\mbox{ for every } v|\f{e}\} ,$$
$$\b{\Gamma} :=\b{G}^n(F)\cap \b{D},$$
$$Q(\f{e}):=\{ r\in\GL_n(\mathbb{A}_{\h})\cap\prod_{v\in\h} M_n(\f{o}_v): r_v=1_n\mbox{ for every } v|\f{e}\} ,$$
$$R(\f{e}):=\{ \diag [\tilde{r},r]:r\in Q(\f{e})\} .$$

For $r\in Q(\f{e})$ and $f\in M^n_{k,S}(\b{\Gamma},\psi)$ we define a linear operator $T_{r,\psi} : M^n_{k,S}(\b{\Gamma},\psi)\rightarrow M^n_{k,S}(\b{\Gamma},\psi)$ by
\begin{equation}\label{T_r}
f|T_{r,\psi}:=\sum_{\b{\alpha}\in\b{A}} \psi_{\f{c}}(\det (a_{\alpha})_{\f{c}})^{-1} f|_{k,S}\b{\alpha},
\end{equation}
where $\b{A}\subset\b{G}^n(F)$ is such that $\b{G}^n(F)\cap\b{D}\diag [\tilde{r},r]\b{D}=\coprod_{\b{\alpha}\in\b{A}}\b{\Gamma\alpha}$. Further, for an integral ideal $\f{a}$ of $F$ we put
$$f|T_{\psi}(\f{a}):=\sum_{\substack{r\in Q(\f{e})\\ \det (r)\f{o}=\f{a}}} f|T_{r,\psi},$$
where we sum over all those $r$ for which the cosets $E r E$ are distinct, where $E:= \prod_{v\in\h}\GL_n(\f{o}_v)$.

We also note here that if we let $\mathbf{f}|T_{r,\psi}$ be the adelic Siegel-Jacobi form associated to $f|T_{r,\psi}$ by the bijection given in (\ref{bijection between classical and adelic}) with $\mathbf{g} =1$, then
$$(\mathbf{f}|T_{r,\psi})(x)=\sum_{\b{\alpha}\in\b{A}} \psi_{\f{c}}(\det (a_{\alpha})_{\f{c}})^{-1} \mathbf{f}(x\b{\alpha}^{-1}),\qquad x\in\b{G}^n(\A),$$
where $\b{D}\diag [\tilde{r},r]\b{D}=\coprod_{\b{\alpha}\in\b{A}}\b{D\alpha}$ with $\b{A}\subset\b{G}_{\h}$. As above we may also define $\mathbf{f}|T_{\psi}(\mathfrak{a})$.

We now consider a nonzero $\mathbf{f} \in \mathcal{S}^n_{k,S}(\b{D},\psi)$ such that $\mathbf{f}|T_{\psi}(\f{a}) = \lambda(\f{a}) \mathbf{f}$ for all integral ideals $\f{a}$ of $F$. For a Hecke character $\chi$ of $F$ we define the series
\[
D(s,\mathbf{f},\chi) := \sum_{\f{a}} \lambda(\f{a}) \chi^{*}(\f{a}) N(\f{a})^{-s},\,\,\,\,\Re(s)\gg 0,
\]
where for a Hecke character $\chi$ we write $\chi^*$ for the corresponding ideal character. Of course, for a prime ideal $\f{q}$ that divides the conductor $\f{f}_{\chi}$ we set $\chi^*(\f{q}) = 0$. A similar argument to \cite[Lemma 2.2]{Ar94} extended to the totally real field case shows that the function $D(s,\mathbf{f},\chi)$ is absolutely convergent for $\Re(s) > 2n+l+1$. 

We now impose a condition on the matrix $S$. We follow \cite[page 142]{Mu89}. Consider any prime ideal $\f{p}$ of $F$ such that $(\f{p},\f{c}) =1$ and write $v$ for the corresponding finite place of $F$. We say that the lattice $L := \mathfrak{o}_v^{l} \subset F_{v}^l$ is an $\f{o}_{v}$-maximal lattice with respect to a symmetric matrix $2S$ if for every $\mathfrak{o}_{v}$ lattice $M$ of $F_{v}^l$ that contains $L$ and satisfies $S[x] \in \f{o}_{v}$ for all $x \in M$, we have $M=L$. For any uniformiser $\pi$ of $F_{v}$ we now set 
\[
L' := \{ x \in (2S)^{-1} L: \pi S[x] \in \f{o}_{v}\} \subset F_{v}^l.
\]
We say that the matrix $S$ satisfies the condition $M_{\f{p}}^+$ if $L$ is an $\f{o}_{v}$-maximal lattice with respect to the symmetric matrix $2S$ and $L=L'$. The main aim of this section is to prove the following theorem.

\begin{thm}\label{Euler Product Representation} Let $0 \neq \mathbf{f} \in \mathcal{S}^n_{k,S}(\b{D},\psi)$ be such that $\mathbf{f}|T_{\psi}(\f{a}) = \lambda(\f{a}) \mathbf{f}$ for all integral ideals $\f{a}$ of $F$.  Assume that the matrix $S$ satisfies the condition $M_{\f{p}}^+$ for every prime ideal $\f{p}$ with $(\f{p},\f{c})=1$. Then
\[
\f{L}(\chi,s) D(s+n+l/2,\mathbf{f},\chi) = L(s,\mathbf{f},\chi) := \prod_{\f{p}} L_{\f{p}}(\chi^{*}(\f{p})N(\f{p})^{-s})^{-s},
\]
where for every prime ideal $\f{p}$ of $F$  
\[
 L_{\f{p}}(X) = \begin{cases} \prod_{i=1}^n \left((1- \mu_{\f{p},i}X)(1-\mu^{-1}_{\f{p},i}X ) \right)\! ,\,\,\, \mu_{\f{p},i} \in \mathbb{C}^{\times} & \hbox{if } (\f{p}, \mathfrak{c})=1,\\
 \prod_{i=1}^n (1- \mu_{\f{p},i}X) \,\,\, \mu_{\f{p},i} \in \mathbb{C} & \hbox{if }(\f{p}, \mathfrak{e}^{-1}\mathfrak{c})\neq 1\\
 1 & \hbox{if } (\f{p}, \mathfrak{e})\neq 1.
 \end{cases}
\]
Moreover, $\f{L}(\chi,s) = \prod_{(\f{p},\f{c})=1} \f{L}_{\f{p}}(\chi,s)$, where
\[
\f{L}_{\f{p}}(\chi,s) := G_{\f{p}}(\chi,s) \cdot  \begin{cases}  \prod_{i=1}^{n} L_{\f{p}}(2s+2n-2i,\chi^{2}) &\mbox{if } l \in 2\mathbb{Z}\\ 
\prod_{i=1}^{n} L_{\f{p}}(2s+2n-2i+1,\chi^{2})  & \mbox{if }  l \not \in 2\mathbb{Z}\end{cases}
\]
and $G_{\f{p}}(\chi,s)$ is a ratio of Euler factors which for almost all $\f{p}$ is equal to one. (Below, in Theorem \ref{good Euler factor} we make $G_{\f{p}}(\chi,s)$ very precise.) In particular, the function  $L(s,\mathbf{f},\chi)$ is absolutely convergent for $\Re(s) > n + l/2 +1$.
\end{thm}

\begin{rem} It is worth to notice that the factor $G_{\f{p}}(\chi,s)$ does not appear in the works of \cite{MS} and \cite{Ar94}. It is because in the case of $l=1$ considered there, the condition $M_{\f{p}}^+$ is equivalent to the condition that the matrix $S$ is regular (see for example \cite[Remark 4.3]{Mu89}), which implies that the factor $G_{\f{p}}(\chi,s)$ is equal to one for all good primes.
\end{rem}

Before we proceed to the proof of the above theorem, we state an immediate corollary regarding the vanishing of the $L$-function defined above. 

\begin{cor} \label{non-vanishing of L-values} With notation and assumptions as in Theorem \ref{Euler Product Representation},
\[
L(s,\mathbf{f},\chi) \neq 0
\]
whenever $\Re(s) > n + l/2 +1$.
\end{cor}
\begin{proof}This follows from the fact that the function $L(s,\mathbf{f},\chi)$ is absolutely convergent for $\Re(s) > n + l/2 +1$ and has an Euler product representation. For the formal argument see \cite[Lemma 22.7]{Sh00}.
\end{proof}

The rest of this section is devoted to a proof of Theorem \ref{Euler Product Representation}. Note that if we fix a prime ideal $\f{p}$ of $F$ and consider the series
\[
D_{\f{p}}(s,\mathbf{f},\chi) := \sum_{j=0}^{\infty} \lambda(\f{p}^j) \chi^{*}(\f{p})^j N(\f{p})^{-js},\,\,\,\,\Re(s)\gg 0,
\]
then
\[
D(s,\mathbf{f},\chi)= \prod_{\f{p}} D_{\f{p}}(s,\mathbf{f},\chi) = \prod_{(\f{p},\f{f}_{\chi})=1} D_{\f{p}}(s,\mathbf{f},\chi),
\]
which means that it suffices to prove the theorem locally place by place. 

\textbf{Local Notation.} For the rest of this section we fix the following notation. We fix a finite place $v \in \mathbf{h}$ of $F$. We abuse the notation and write $F$ for $F_v$, $\mathfrak{o}$ for $\f{o}_v$, and just $\f{p}$ for the corresponding maximal ideal in $\f{o}_v$. Moreover, we  denote by $\pi \in \f{p}$ any uniformiser of this place. We further set $q := [\mathfrak{o}:\mathfrak{p}]$ and denote by $| \cdot |$ the absolute value of $F$ normalised so that $|\pi| = q^{-1}$. We also write $\b{G}, G, \b{D}, D$ for $\b{G}(F_v), G(F_v), \b{D}_v$ and $D_v$. Finally, in this part of the paper we denote by $\psi_S$ the $v$-component of the additive adelic character $\psi_S$ introduced in section $3$. 
\subsection{The good places}

We consider first a finite place $v$ which is not in the support of $\mathfrak{c} \mathfrak{f}_{\chi}$. We assume that the matrix $S_v$ satisfies condition $M_{\f{p}}^+$. As we have indicated at the beginning of this section we will extend the results of \cite{MS} from the case $l=1$ to any $l$, and also introduce the twisting by a finite character $\chi$. Here we use (more or less) the notation from \cite{Mu89,Mu91,MS}. \newline

We define a local Hecke algebra $\mathfrak{X}$ as in \cite[page 142]{Mu89}. That is, let $\mathfrak{X}$ be the $\mathbb{C}$-module consisting of $\mathbb{C}$-valued functions $\phi$ on $\b{G}$ which satisfy 
\[
\phi((00,\kappa) \b{d} \b{g} \b{d}' ) = \psi_S(\kappa) \phi(\b{g}), \,\,\,\,\b{d},\b{d}' \in \b{D}, \b{g} \in \b{G},\,\,\kappa \in Sym_l(F) 
\]
and have compact support modulo $\mathcal{Z} := Sym_l(F) \subset \b{G}$. As it is explained in \cite{Mu89}, one can give to this module the structure  of an algebra by defining multiplication through convolution of functions. Moreover, it is shown in \cite[Lemma 4.4]{Mu89} that the assumption $M_{\mathfrak{p}}^{+}$ implies that a function $\phi \in \mathfrak{X}$ has support in 
\[
\bigcup_{\alpha \in \Lambda^+} \b{D} d_n(\pi_{\alpha}) \b{D} \mathcal{Z},
\]
where $\Lambda^+ := \{ (a_1,a_2,\ldots,a_n) \in \mathbb{Z}^n: a_1 \geq a_2 \geq \ldots \geq a_n \geq 0 \}$, 
\[
d_n : GL_n \hookrightarrow G \subset \b{G},\qquad d_n(a) := \diag[a, \transpose{a}^{-1}] ,
\]
and $\pi_{\alpha} := \diag[\pi^{a_1},\pi^{a_2},\ldots,\pi^{a_n}] \in GL_n(F)$. \newline

Let
\[
T := T(F) := \{ d_n (\diag[t_1,\ldots, t_n]): t_i \in F^{\times} \} \in G
\]
and 
\[
X_0(T) := \{ \xi \in Hom(T, \mathbb{C}^{\times}): \xi \,\, \text{is trivial on} \,\,T(\mathfrak{o})\}.
\]
For a character $\xi \in X_0(T)$ and $\phi \in \mathfrak{X}$ set
\[
\lambda_{\xi}(\phi) := \sum_{\alpha \in \mathbb{Z}^{n}} \xi^{-1}(d_n(\pi_{\alpha})) \hat{\phi}(d_n(\pi_{\alpha})),
\]
where for a function $\phi \in \mathfrak{X}$, $\hat{\phi}(t)$ is defined as in \cite[equation (4.8)]{Mu89}, that is,
\[
\hat{\phi}(t):= \delta_{\mathbf{N}_0}(t)^{-1/2} \int_{\mathbf{N}_0} \phi(\mathbf{n}_0 t) d\mathbf{n}_0,
\]
where $\mathbf{N}_0 := V_0 N_0 \subset \mathbf{G}$, $N_0$ is the unipotent radical of the Siegel parabolic $P_0$ of $\Sp_n$, $V_0 := \{(0, \mu, 0) : \mu \in M_{l,n}\}$, and $\delta_{\mathbf{N}_0}$ and the Haar measure $d\mathbf{n}_0$ are normalized as in \cite[page 144]{Mu89}.

For an $\alpha \in \Lambda^+$ we define $\phi_{\alpha} \in \f{X}$ by
\[
\phi_\alpha(\b{g}) := \begin{cases} \psi_S(\kappa) & \hbox{if }\b{g} = (0,0,\kappa)\b{d}d_n(\pi_{\alpha}) \b{d}' \in \mathcal{Z}\b{D} d_n(\pi_{\alpha}) \b{D},\\
0 & \hbox{otherwise},  
\end{cases}
\]
and for a finite unramified character $\chi$ of $F^{\times}$ we define the function $\nu_{s,\chi}$ on $\mathbf{G}$, $s \in \mathbb{C}$, by the conditions 
\[
\nu_{s,\chi}((0,0,\kappa) \b{d} g \b{d}') = \psi_S(-\kappa) \nu_{s,\chi}(g),\,\,\, g \in G,\,\,\b{d},\b{d}' \in \mathbf{D}
\]
and
\[
\nu_{s,\chi}(\pi_{\alpha}) := \chi(\pi_v)^{\ell(\alpha) } q^{-\ell(\alpha)s},
\]
where $\ell(\alpha) = \sum_{i=1}^na_i$. It is shown in \cite{MS} that these two conditions uniquely determine the function $\nu_{s,\chi}$.  Now, given a character $\xi \in X_0(T)$ and an unramified character $\chi$ of $F^{\times}$, we introduce the series
\[
B(\xi,\chi,s) := \sum_{\alpha \in \Lambda_n^{+}} \lambda_{\xi}(\phi_{\alpha}) \chi(\pi)^{\ell(\alpha)} q^{-\ell(\alpha)s}.
\]
Given a $\xi \in X_0(T)$ we define the function $\phi_\xi$ on $\mathbf{G}$ following \cite[equation (4.11)]{Mu89} by
\[
\phi_\xi((0,0,\kappa) \mathbf{n}_0 t (\lambda,0,0) \mathbf{d}) = \psi_S(\kappa) (\xi \delta^{1/2}_{\mathbf{n}_0})(t) \Phi_L(\lambda), \,\,\,\mathbf{d} \in \mathbf{D},\,t \in T,\,\,\mathbf{n}_0 \in \mathbf{N}_0,
\]
where $\Phi_L$ is the characteristic function of $L := M_{l,n}(\mathfrak{o})$. The following lemma (\cite{Ar94}, Lemma 5.2) gives an important integral representation of the series $B(\xi,\chi,s)$. 
\begin{lem}[Murase] For $\xi \in X_0(T)$ and a finite unramified character $\chi$ of $F^{\times}$ we have
\[
B(\xi,\chi,s) = \int_{\mathcal{Z} \setminus \mathbf{G}} \nu_{\chi,s}(g) \phi_{\xi}(g) dg. 
\]
\end{lem}

\begin{rem} The original lemma in \cite{Ar94} is stated without a twist by $\chi$, but it is easy to see that the arguments there extend easily to include also the case of twisting by an unramified character.
\end{rem}

For a finite unramified character $\chi$ and a character $\xi = (\xi_1, \ldots, \xi_n)\in X_0(T)$, where  $\xi_i$ are unramified characters of $F^{\times}$, we define the local $L$-function
\[
L(\xi,\chi,s) := \prod_{i=1}^{n} (1-\xi_i(\pi) \chi(\pi) q^{-s})^{-1} (1-\xi_i^{-1}(\pi) \chi(\pi) q^{-s})^{-1}.
\]
In order to state the main theorem of this section we need to introduce a bit more notation. We write $\alpha_S(s,\chi)$ for the Siegel series attached to the symmetric matrix $S$ and to the character $\chi$, as defined for example in \cite[Chapter III]{Sh96}. Moreover, by \cite[Theorem 13.6]{Sh96}, we have
\begin{equation}\label{Siegel_series}
\alpha_S (s,\chi) = \left(L(s,\chi) \prod_{i=1}^{[l/2]} L(2s-2i,\chi^{2})\right)^{-1} g_S(s,\chi) 
\end{equation}
for some analytic function $g_S(s,\chi)$ of the form $g_S(s,\chi) = G(\chi(\pi)q^{-s})$ for some polynomial $G(X) \in \mathbb{Z}[X]$ of constant term one. Moreover if $S$ is regular, that is, $\det(2S) = \mathfrak{o}^{\times}$ for $l$ even and $\det(2S) = 2\mathfrak{o}^{\times}$ for $l$ odd, then $g_S(s,\chi) = 1$. \newline

The following theorem generalizes a result due to Murase and Sugano \cite{MS}, where the case of $l=1$ and $\chi$ trivial is considered.

\begin{thm}\label{good Euler factor}
With the notation as above,
$$L(\xi,\chi,s) = \frac{g_S(s+n+l/2,\chi)}{g_S(s+l/2,\chi)} \Lambda(\chi,s)\int_{\b{Z}\setminus \b{G}}  \nu_{\chi,s+n+l/2}(\b{g})\phi_{\xi}(\b{g})d\b{g}\Lambda(\chi,s),$$
where
$$\Lambda(\chi,s):=\begin{cases}  \prod_{i=1}^{n} L(2s+2n-2i,\chi^{2}) &\mbox{if } l \in 2\mathbb{Z},  \\ 
\prod_{i=1}^{n} L(2s+2n-2i+1,\chi^{2})  & \mbox{if }  l \not \in 2\mathbb{Z}.\end{cases} $$
In particular,
$$L(\xi,\chi,s) =  B(\xi,\chi,s+n+l/2) \frac{g_S(s+n+l/2,\chi)}{g_S(s+l/2,\chi)}\Lambda(\chi,s).$$
\end{thm}

The rest of this subsection is devoted to a proof of this theorem. First we extend some calculations of Murase and Sugano \cite{MS}. Denote by $\sigma_{n_1,n_2}$ the characteristic function of $M_{n_1,n_2}(\mathfrak{o})$ and let
\begin{multline*}
F(s,\chi,\b{g}) := F(s,\chi, h g ) :=\\
\int_{GL_{2n+l}(F_v)} \sigma_{2n+l,4n+2l} \left( \left( y \left(\begin{matrix} 1_l & 0 \\ 0 & g \end{matrix} \right), y \alpha(h) \right) \right) \chi(\det(y)) |\det(y)|^{s+n + l/2} d^*y, 
\end{multline*}
where for $h = (\lambda,\mu,\kappa) \in H$ we set
\[
\alpha(h) := \left( \begin{matrix}   \kappa - \lambda \transpose{\mu} & - \lambda & -\mu \\                           
\transpose{\mu} & 1_n & 0 \\
\transpose{\lambda} & 0 & 1_n  \end{matrix} \right).
\]
Define also
\[
\mathcal{F}(s,\chi,\b{g}) := \int_{\mathcal{Z}} F(s,\chi,(0,0,\kappa) \b{g}) \psi_S(\kappa) d \kappa .
\]
We now recall a theorem of Murase in \cite[Theorem 2.12]{Mu91}.
\begin{thm}[Murase] We have the equality:
\begin{align*}
L(\xi,\chi,s) &= \alpha_S(s+l/2,\chi)^{-1} L(s+l/2,\chi)^{-1} \frac{\prod_{i=1}^{n} L(2s+2n+l-2i,\chi^2)}{\prod_{i=1}^{2n+l-1} L(s+n+l/2-i,\chi)} \\
&\hspace{0.4cm}\cdot\int_{\mathcal{Z}\setminus \b{G}} \mathcal{F}(s,\chi,\b{g}) \phi_{\xi}(\b{g})d\b{g}.
\end{align*}
\end{thm}

The following lemma extends a result of Murase and Sugano in \cite[Lemma 6.8]{MS}
from the case of index one ($l=1$) to any index. 
\begin{lem} \label{additive series} We have the following equality:
\begin{align*}
\mathcal{F}(s,\chi,\b{g}) &=\left(\prod_{i=1}^l L(s+n+l/2-i+1,\chi)\right)  \alpha_S(s+n+l/2,\chi)\\
&\hspace{0.4cm}\cdot\left(\prod_{i=1}^{2n} L(s+n-l/2-i+1,\chi)\right)\nu_{s+n+l/2,\chi}(\b{g}). 
\end{align*}
\end{lem}
\begin{proof}
We recall first a result of Shimura.  By \cite[Lemma 3.13]{Sh96}, for any $g \in M_m(F)$,
\begin{equation}\label{integral from Shimura}
\int_{GL_m(F)} \sigma_{m,2m} (yg,y) \chi(\det(y)) |\det(y)|^{s} d^{*}y = \prod_{i=1}^m L(s-i+1,\chi) \chi(\nu_0(g)) \nu(g)^{-s},
\end{equation}
where $\nu_0(g)$ and $\nu(g)$ denote the denominator ideal of $g$ and its norm respectively, as defined for example in \cite[page 19]{Sh96}. 

By \cite[Proposition 2.3 ]{Mu91},
\[
\mathcal{F}(s,\chi,(0,0,\kappa) \b{d} \b{g} \b{d}') = \psi_S(-\kappa) \mathcal{F}(s,\chi,\b{g} )
\]
for all $\kappa \in \mathcal{Z}$ and $\b{d},\b{d}' \in \b{D}$. That is, thanks to \cite[Lemma 4.4]{Mu89}, for a fixed $s$ the function $\mathcal{F}(s,\chi,\b{g} )$ is supported on $\bigcup_{m \in \Lambda_n^+} \mathcal{Z} \b{D} \pi_m \b{D}$. Hence, it is enough to prove the equality of the Lemma for $\b{g} = \pi_m$ for an $m \in \Lambda^+_n$. We have
\begin{align*}
\mathcal{F}(s,\chi,\pi_m) &= 
\int_{GL_{2n+l}(F)} \sigma_{2n+l,4n+2l} \( y \spmatrix{ 1_l &  \\  & \pi_m } , y \spmatrix{  \kappa &  \\   & 1_{2n} }\) \chi(\det(y)) |\det(y)|^{s+n + l/2} d^*y\\
&\hspace{0.4cm}\cdot\int_{\mathcal{Z}} \psi_S(\kappa) d \kappa 
\end{align*}

Write $y = k \(\begin{smallmatrix}  a & b \\  & d \end{smallmatrix}\)$, where $ k \in \GL_{2n+l}(\f{o})$, $a \in \GL_{l}(F)$, $d \in \GL_{2n}(F)$ and $b \in M_{l,2n}(F)$. Then $\mathcal{F}(s,\chi,\pi_m) = I_1 \cdot I_2 \cdot I_3$, where
\[
I_1 = \int_{\mathcal{Z}} \psi_S(\kappa) \int_{GL_l(F)} \sigma_{l,l}(a) \sigma_{l,l}(a \kappa) \chi(\det(a)) |\det(a)|^{s+n + l/2} d^*a,
\]
\[
I_2 = \int_{M_{l,2n}(F)} \sigma_{l,2n}(b \pi_m) \sigma_{l,2n}(b) db 
\]
and 
\[
I_3 = \int_{\GL_{2n}(F)} \sigma_{2n,2n}(d) \sigma_{2n,2n}(d \pi_m) \chi(\det(d))|\det(d)|^{s+n+l/2}|\det(d)|^{-l}d^*d. 
\]

We compute first the integral $I_1$. By the equation \eqref{integral from Shimura},
\begin{multline*}
\int_{\GL_l(F)} \sigma_{l,l}(a) \sigma_{l,l}(a \kappa) \chi(\det(a)) |\det(a)|^{s+n + l/2} d^*a\\ 
=\prod_{i=1}^l L(s+n+l/2-i+1,\chi) \chi(\nu_0(\kappa)) \nu(\kappa)^{-s-n-l/2} ,
\end{multline*}
and hence 
\[
I_1 = \prod_{i=1}^l L(s+n+l/2-i+1,\chi) \int_{\mathcal{Z}} \psi_S(\kappa)\chi(\nu_0(\kappa)) \nu(\kappa)^{-s-n-l/2} d \kappa .
\]
But the last integral is nothing else than the Siegel series $\alpha_S(s+n+l/2,\chi)$, and thus 
\[
I_1 = \prod_{i=1}^l L(s+n+l/2-i+1,\chi) \alpha_S(s+n+l/2,\chi).
\]

Finally, it is easy to see that $I_2 = q^{-(m_1+\ldots+m_n)l}$, and that by the equation \eqref{integral from Shimura} again,
$$I_3 = \prod_{i=1}^{2n} L(s+n-l/2-i+1,\chi) \chi(\nu_0(\pi_m)) \nu(\pi_m)^{-s-n-l/2} .$$
\end{proof}

\begin{proof}[Proof of Theorem \ref{good Euler factor}]
By Lemma \ref{additive series}, 
\begin{align*}
L(\xi,\chi,s) &= \alpha_S(s+l/2,\chi)^{-1} L(s+l/2,\chi)^{-1} \left(\prod_{i=1}^{2n+l-1} L(s+n+l/2-i,\chi)\right)^{-1}\\ 
&\hspace{0.4cm}\cdot\prod_{i=1}^{n} L(2s+2n+l-2i,\chi^2) 
\prod_{i=1}^l L(s+n+l/2-i+1,\chi) \alpha_S(s+n+l/2,\chi)\\ 
&\hspace{0.4cm}\cdot\prod_{i=1}^{2n} L(s+n-l/2-i+1,\chi) \int_{\b{Z}\setminus \b{G}}  \nu_{s+n+l/2,\chi}(\b{g})\phi_{\xi}(\b{g})d\b{g}\\
&=\alpha_S(s+l/2,\chi)^{-1} L(s+l/2,\chi)^{-1} \prod_{i=1}^{n} L(2s+2n+l-2i,\chi^2)\\  
&\hspace{0.4cm}\cdot L(s+n+l/2,\chi) \alpha_S(s+n+l/2,\chi)
\int_{\b{Z}\setminus \b{G}}  \nu_{s+n+l/2,\chi}(\b{g})\phi_{\xi}(\b{g})d\b{g}\\
&=\frac{\alpha_S(s+n+l/2,\chi)}{\alpha_S(s+l/2,\chi)} \frac{ L(s+n+l/2,\chi)}{L(s+l/2,\chi)} \prod_{i=1}^{n} L(2s+2n+l-2i,\chi^2) \\
&\hspace{0.4cm}\cdot \int_{\b{Z}\setminus \b{G}}  \nu_{s+n+l/2,\chi}(\b{g})\phi_{\xi}(\b{g})d\b{g}.
\end{align*}
If we now plug in the expression \eqref{Siegel_series} for the Siegel series, we obtain
\begin{align*}
L(\xi,\chi,s) &= \frac{g_S(s+n+l/2,\chi)}{g_S(s+l/2,\chi)} \frac{\prod_{i=1}^{[l/2]} L(2s+l-2i,\chi^{2})}{\prod_{i=1}^{[l/2]} L(2s+2n+l-2i,\chi^{2})} \prod_{i=1}^{n} L(2s+2n+l-2i,\chi^2) \\
&\hspace{0.4cm}\cdot \int_{\b{Z}\setminus \b{G}}  \nu_{s+n+l/2,\chi}(\b{g})\phi_{\xi}(\b{g})d\b{g}\\
&=\frac{g_S(s+n+l/2,\chi)}{g_S(s+l/2,\chi)} \prod_{i=[l/2]+1}^{[n+l/2]}\hspace{-0.3cm} L(2s+2n+l-2i,\chi^{2})
\int_{\b{Z}\setminus \b{G}} \hspace{-0.3cm}\nu_{s+n+l/2,\chi}(\b{g})\phi_{\xi}(\b{g})d\b{g} ,
\end{align*}
which finishes the proof.
\end{proof}

Given a cusp form $0 \neq \mathbf{f} \in S_{k,S}^n(\b{D},\psi)$ we can define an action of an element $\phi$ in the Hecke algebra 
$\mathfrak{X}$ by 
\[
(\mathbf{f} \star \phi) (\b{g}) = \int_{\mathcal{Z} \setminus \b{G}} \mathbf{f}(\b{g} x^{-1}) \phi(x) dx.
\]
If now $\mathbf{f}$ is a common eigenform for all $\phi\in\f{X}$, that is, $\mathbf{f} \star \phi = \lambda_{\mathbf{f}}(\phi) \mathbf{f}$ for all $\phi$, then we obtain a $\mathbb{C}$-algebra homomorphism $\lambda_{\mathbf{f}} : \mathfrak{X} \rightarrow \mathbb{C}$. 
Thanks to \cite[Theorem 4.15]{Mu89} we know that this homomorphism is of the form
\[
\lambda_{\mathbf{f}}(\phi) = \lambda_{\xi_{\mathbf{f}}}(\phi)
\]
for some character $\xi_{\mathbf{f}} \in X_0(T)$, and thus, as it is explained in \cite[Lemma 5.4]{Ar94}, 
\[
\mathbf{f} \star \phi_{\alpha} = \mathbf{f} | T_{\pi^{-1}_\alpha,\psi_S}\quad\mbox{for every }\alpha \in \Lambda_n^+. 
\]
Note here that since $\b{D} d_n(\pi_{\alpha}) \b{D} = \b{D} d_n(\pi^{-1}_{\alpha}) \b{D}$, we obtain  
\[
B(\xi_{\mathbf{f}},\chi,s) = D_{\mathfrak{p}}(s,\mathbf{f},\chi).
\]
In this way we can conclude Theorem \ref{Euler Product Representation} in the case when $v$ is a good prime by taking $\mu_{\f{p},i} :=\xi_i(\pi)$ if $\xi_{\mathbf{f}} = (\xi_1,\ldots,\xi_n)$.
\subsection{The bad places}

We now consider the case of $(\mathfrak{p},\mathfrak{c}) \neq 1$. If $(\mathfrak{p},\mathfrak{e}) \neq 1$, then there is nothing to show, because in this case each Hecke operator is just the identity. Hence we consider the case of $(\mathfrak{p},\mathfrak{e}^{-1}\f{c}) \neq 1$. In this section we set $E:= \GL_n(\mathfrak{o})$ and $\mathcal{S} := S(\mathfrak{b}^{-1}) :=  Sym_n(F) \cap M_n(\mathfrak{b}_v^{-1})$. \newline

First we work out the decomposition of the double cosets $\b{D} \diag[\tilde{\xi}, \xi] \b{D}$. Recall that we write $\b{D} = CD$ with $C = C_v[\mathfrak{o},\mathfrak{b}^{-1},\mathfrak{b}^{-1}] \subset H$ and $D = D_v[\mathfrak{b}^{-1},\mathfrak{bc}]\subset G$. By \cite[Lemma 19.2]{Sh00} we know that
\[
D \diag[\tilde{\xi}, \xi] D = \bigsqcup_{d,b} D \left( \begin{matrix}  \tilde{d} & \tilde{d}b \\  & d \end{matrix} \right), 
\]
where $d \in E \setminus E\xi E$ and $b \in \mathcal{S}/\transpose{d}\mathcal{S}d$, and thus 
\[
\b{D} \diag[\tilde{\xi}, \xi] \b{D} = CD \diag[\tilde{\xi}, \xi] DC = \bigsqcup_{d,b} \b{D} \left( \begin{matrix}  \tilde{d} & \tilde{d}b \\  & d \end{matrix} \right) C.
\]
Observe that for elements $(\lambda,\mu,\kappa) \in C$ and $\left( \begin{matrix}  \tilde{d} & \tilde{d}b \\  & d \end{matrix} \right)$ as above we have
\[
\left( \begin{matrix}  \tilde{d} & \tilde{d}b \\  & d \end{matrix} \right)(\lambda,\mu,\kappa)  = (\lambda \transpose{d}, (-\lambda b + \mu) d^{-1}, \kappa + \lambda \transpose{d} \transpose{d}^{-1}\transpose{(-\lambda b + \mu)} -\lambda \transpose{\mu} )\left( \begin{matrix}  \tilde{d} & \tilde{d}b \\  & d \end{matrix} \right).
\] 
In particular, 
\begin{equation}\label{double_coset_dec}
 \b{D} \diag[\tilde{\xi}, \xi] \b{D} = \bigsqcup_{d,b,\mu} \b{D} (0,\mu,0)  \left( \begin{matrix}  \tilde{d} & \tilde{d}b \\  & d \end{matrix} \right),
\end{equation}
where $d \in E \setminus E\xi E$, $b \in \mathcal{S} / \transpose{d} \mathcal{S} d$ and $\mu \in M_{l,n}(\mathfrak{b}^{-1}_v) d^{-1} / M_{l,n}(\mathfrak{b}^{-1}_v)$.

We will show that the set $\b{D} X \b{D}$, with $X = \{ \diag(\tilde{\xi} , \xi) : \xi \in M_n(\mathfrak{o}_v) \cap GL_n(F_v) \}$ is closed under multiplication. For $ \b{D} \diag[\tilde{\xi}_i, \xi_i] \b{D} = \bigsqcup_{d_i,b_i,\mu_i} (0,\mu_i,0)  \left( \begin{matrix}  \tilde{d}_i & \tilde{d}_ib_i \\  & d_i \end{matrix} \right)$, $i=1,2$, we have
\begin{align*}
\b{D} &\diag[\tilde{\xi}_1, \xi_1] \b{D}\diag[\tilde{\xi}_2, \xi_2] \b{D}\\
&= \bigsqcup_{d_1,b_1,\mu_1,d_2,b_2,\mu_2} \b{D} (0,\mu_1,0)  \left( \begin{matrix}  \tilde{d}_1 & \tilde{d}_1b_1 \\  & d_1 \end{matrix} \right)(0,\mu_2,0)  \left( \begin{matrix}  \tilde{d}_2 & \tilde{d}_2b_2 \\  & d_2 \end{matrix} \right)\\
&=\bigsqcup_{d_1,b_1,\mu_1,d_2,b_2,\mu_2} \b{D} (0,\mu_1,0)  \left( \begin{matrix}  \tilde{d}_1 & \tilde{d}_1b_1 \\  & d_1 \end{matrix} \right) \left( \begin{matrix}  \tilde{d}_2 & \tilde{d}_2b_2 \\  & d_2 \end{matrix} \right)(0,\mu_2d_2,0)\\
&=\bigsqcup_{d_1,b_1,\mu_1,d_2,b_2,\mu_2} \b{D} (0,\mu_1,0)  \diag[\tilde{d}_1 \tilde{d}_2, d_1 d_2 ]\left( \begin{matrix}  1 &  b_2 + \transpose{d}_2 b_1 d_2 \\  & 1 \end{matrix} \right)(0,\mu_2d_2,0)\\
&=\bigsqcup_{d_1,b_1,\mu_1,d_2,b_2,\mu_2} \b{D}\diag[\tilde{d}_1 \tilde{d}_2, d_1 d_2 ](0,\mu_1 d_1 d_2,0) \left( \begin{matrix}  1 &  b_2 + \transpose{d}_2 b_1 d_2 \\  & 1 \end{matrix} \right)(0,\mu_2d_2,0).
\end{align*}
Hence, because $(0,\mu_1 d_1 d_2,0), (0,\mu_2d_2,0) \in C$, $\spmatrix{  1 &  b_2 + \transpose{d}_2 b_1 d_2 \\  & 1 } \in D$ and $\tilde{d}_1 \tilde{d}_2 = \widetilde{d_1 d_2}$, we have shown that
\[
\b{D} \diag[\tilde{\xi}_1, \xi_1] \b{D}\diag[\tilde{\xi}_2, \xi_2] \b{D} \subset \b{D} X \b{D}.
\]

We define the Hecke algebra $\mathfrak{X} := \mathfrak{X}_v$ for $v | \mathfrak{e}^{-1}\mathfrak{c}$ to be the algebra generated by the double cosets $\b{D} X \b{D}$. \newline

In order to define the Satake parameters associated to an eigenform of this algebra we need to define an injective algebra homomorphism $\omega : \mathfrak{X} \rightarrow \mathbb{Q}[t_1,\ldots, t_n ]$. We will do this by reducing everything to the theory of $\GL_n$, very much in the spirit of Shimura in \cite[Theorem 19.8]{Sh00}. \newline

Given an element 
\[
\b{D} \diag[\tilde{\xi}, \xi] \b{D} = \bigsqcup_{d,b,\mu} (0,\mu,0)  \left( \begin{matrix}  \tilde{d} & \tilde{d}b \\  & d \end{matrix} \right),
\]
where $d \in E \setminus E\xi E$, $b \in \mathcal{S} / \transpose{d} \mathcal{S} d$ and $\mu \in M_{l,n}(\mathfrak{b}^{-1}_v) d^{-1} / M_{l,n}(\mathfrak{b}^{-1}_v)$, we set 
\[
\omega_0 \left( (0,\mu,0)  \left( \begin{matrix}  \tilde{d} & \tilde{d}b \\  & d \end{matrix} \right)\right):= \omega_0(E d),
\]
where $\omega_0$  is the classical map of the spherical Hecke algebra of $\GL_n$ defined as $\omega_0(Ed) := \prod_{i=1}^{n} (\xi^{-i} t_i)^{e_i}$ if an upper triangular representative of  $Ed$ has the diagonal entries $\pi^{e_1}, \pi^{e_2}, \ldots, \pi^{e_n}$ with $e_i \in \mathbb{Z}$. Further, let
\[
\omega(\b{D} \diag[\tilde{\xi}, \xi] \b{D}):= \sum_{d,b,\mu} \omega_0 \left( (0,\mu,0)  \left( \begin{matrix}  \tilde{d} & \tilde{d}b \\  & d \end{matrix} \right)\right) .
\] 
An identical argument to the one in \cite[Proposition 16.14]{Sh96} shows that $\omega : \mathfrak{X} \rightarrow \mathbb{Q}[[t_1^{\pm},t_2^{\pm},\ldots,t_n^{\pm}]]$ is an injective algebra homomorphism.

For a finite unramified character $\chi$ and for $s \in \mathbb{C}$ consider the formal series 
\[
B(\chi,s):=\sum_{\xi \in E \setminus B / E} (\b{D} \diag[\tilde{\xi}, \xi ] \b{D}) \chi(\det(\xi)) N(\det(\xi))^{-s},
\]
where $B := \GL_n(F) \cap M_n(\f{o})$. Then, if we define 
\[
\omega(B(\chi,s)):=\sum_{\xi \in E \setminus B / E} \omega(\b{D} \diag[\tilde{\xi}, \xi ] \b{D}) \chi(\det(\xi)) N(\det(\xi))^{-s},
\]
we have that
\[
\omega(B(\chi,s)) = \sum_{d \in E \setminus B} \omega_0(Ed) |\det(d)|^{-n-l} \chi(\det(d)) N(\det(d))^{-s}.
\]
Hence, by an argument similar to the one in \cite[Theorem 19.8]{Sh00}, we get
\[
\omega(B(\chi,s)) = \prod_{i=1}^n (1-q^{n+l} t_i \chi(\pi) q^{-s})^{-1} \in \mathbb{Q}[[t_1,\ldots,t_n]].
\]
Now \cite[Lemma 19.9]{Sh00} states that if we have a $\mathbb{Q}$-linear homomorphism $\lambda\colon\mathfrak{X} \rightarrow \mathbb{C}$ which maps the identity element to 1, then there exist Satake parameters $\mu_1,\ldots, \mu_n \in \mathbb{C}$ such that 
\[
\sum_{\xi \in E \setminus B / E} \lambda(\b{D} \diag[\tilde{\xi}, \xi ] \b{D}) \chi(\det(\xi)) N(\det(\xi))^{-s} = \prod_{i=1}^n (1-q^{n+l}\mu_i \chi(\pi) q^{-s})^{-1}
\]
or, equivalently,
\[
\sum_{\xi \in E \setminus B / E} \lambda(\b{D} \diag[\tilde{\xi}, \xi ] \b{D}) \chi(\det(\xi)) N(\det(\xi))^{-(s+n+l/2)} = \prod_{i=1}^n (1-q^{-l/2}\mu_i \chi(\pi) q^{-s})^{-1}
\]
as an equality of formal series in $\mathbb{C}[[q^{-s}]]$. Hence, if we take as $\lambda$ the homomorphism obtained from the eigenform $\mathbf{f}$ and let $\mu_{\f{p},i} := \mu_i q^{-l/2}$, we establish the rest of Theorem \ref{Euler Product Representation}, as in this case   
$$D_{\f{p}}(s,\mathbf{f},\chi) = \sum_{\xi \in E \setminus B / E} \lambda(\b{D} \diag[\tilde{\xi}, \xi ] \b{D}) \chi(\det(\xi)) N(\det(\xi))^{-s}.$$

\subsection{A $\psi$-twisted $L$-function}\label{The twisted L-function}
To an eigenform $\mathbf{f} \in S_{k,S}^n(\b{D},\psi)$ we can associate yet another $L$-function. It appears naturally in the doubling method when the form $\mathbf{f}$ has a non-trivial nebentype. For a character $\chi$ of conductor $\f{f}$ we define
\begin{multline*}
L_{\psi}(s,\mathbf{f},\chi) := \prod_{\f{p}} L_{\f{p}}(\chi^*(\mathfrak{p}) (\psi/\psi_{\mathfrak{c}})(\pi_{\mathfrak{p}}) N(\f{p})^{-s} )\\
=\left(\prod_{(\f{p},\f{c})=1} L_{\f{p}}((\chi\psi)^*(\mathfrak{p}) N(\f{p})^{-s} ) \right) \left(\prod_{\f{p} | \f{c}} L_{\f{p}}(\chi^*(\mathfrak{p})  N(\f{p})^{-s} ) \right),
\end{multline*}
where $\psi_{\f{c}} = \prod_{v | \f{c}} \psi_v$, $\pi_{\f{p}} \in \f{o}_{\f{p}}$ is a uniformizer of the ring of integers $\f{o}_{\f{p}}$, and the factors $L_{\f{p}}(X)$ are as in Theorem \ref{Euler Product Representation}.  We also define the series
\[
D_{\psi}(s,\mathbf{f},\chi) := \sum_{\f{a}} \lambda(\f{a}) \chi^{*}(\f{a}) \psi(\f{a}')N(\f{a})^{-s},
\]
where for an ideal $\f{a}$ with prime decomposition $\prod_{\f{p}}\f{p}^{n_{\f{p}}}$ we put $\f{a}' :=\prod_{(\f{p},\f{c})=1}\f{p}^{n_{\f{p}}}$. Then:
\[
D_{\psi}(s,\mathbf{f},\chi) = \prod_{(\f{p},\f{c})=1}\!\!\! D_{\f{p}}(s, \mathbf{f}, \chi\psi )\, \prod_{\f{p} | \f{c}} D_{\f{p}}(s,\mathbf{f},\chi).
\]
In particular, by Theorem \ref{Euler Product Representation},
\[
\f{L}_{\psi}(\chi,s) D_{\psi}(s+n+l/2,\mathbf{f},\chi) = L_{\psi}(s,\mathbf{f},\chi),
\]
where 
$\f{L}_{\psi}(\chi,s) = \prod_{(\f{p},\f{c})=1} \f{L}_{\f{p}}(\chi\psi,s)$, and
\[
\f{L}_{\f{p}}(\chi\psi,s) := G_{\f{p}}(\chi\psi,s) \begin{cases}  \prod_{i=1}^{n} L_{\f{p}}(2s+2n-2i,(\chi\psi)^{2}) &\mbox{if } l \in 2\mathbb{Z} \\ 
\prod_{i=1}^{n} L_{\f{p}}(2s+2n-2i+1,(\chi\psi)^{2})  & \mbox{if }  l \not \in 2\mathbb{Z}\end{cases} .
\]
Finally, for any given integral ideal $\f{x}$ we define the function
\[
L_{\psi,\f{x}}(s,\mathbf{f},\chi) := \prod_{(\f{p},\f{x})=1} L_{\f{p}}(\chi^*(\mathfrak{p}) (\psi/\psi_{\mathfrak{c}})(\pi_{\mathfrak{p}}) N(\f{p})^{-s} ),
\]
that is, we remove the Euler factors at the primes which divide $\f{x}$.
\subsection{The global Hecke algebra} \label{Normal Operators} Now let $\mathfrak{X} := \bigotimes_v \mathfrak{X}_v$ be the global Hecke algebra. Since every local Hecke algebra $\mathfrak{X}_v$ can be embedded in a power series ring (for the good places this has been established in \cite[Theorem 4.14]{Mu89} and for the bad places above), and thus is commutative, we can conclude that the global Hecke algebra $\mathfrak{X}$ is also commutative.  Moreover, if $T_{r,\psi}$ is the Hecke operator where $r_v = 1_n$ at $v | \f{c}$, then 
\[
< f|T_{r,\psi},g> = <f,g|T_{r,\psi}>.
\]
Indeed, this follows from the fact that $<f |_{S,k} \alpha, g|_{S,k}\alpha>=<f,g>$ for any $\alpha \in G^n$ and that for any $r$ as above we have
\[
\b{D}\diag [\tilde{r},r]\b{D} = CD\diag [\tilde{r},r] DC =  CD\diag [\transpose{r},r^{-1}]CD =\b{D}\diag [\transpose{r},r^{-1}]\b{D},
\]
where the second equality follows from \cite[Remark on page 89]{Sh96}. In particular, it follows that the Hecke operators $T(\f{a})$ with $(\f{a},\f{c})=1$ are normal, and thus can be simultaneously diagonalized.

We finish this section by obtaining a result which will be useful for our later considerations. We first recall that we have defined $f^c(z) = \overline{f(-\overline{z})}$. Now set $\epsilon := \diag[1_n, -1_n]$ and define
\begin{equation}\label{epsilon_automorphism}
\b{\epsilon}((\lambda,\mu,\kappa) \gamma) \b{\epsilon}:= (\lambda,-\mu,-\kappa) \epsilon \gamma \epsilon .
\end{equation}
We will check that this is a group automorphism of the Jacobi group. For any $\b{\gamma}_1 = (\lambda_1,\mu_1,\kappa_1) g_1^{-1}$ and $\b{\gamma}_2 = (\lambda_2,\mu_2,\kappa_2) g_2$, where $g_1 = \left(\begin{smallmatrix} a_1 & b_1 \\ c_1 & d_1\end{smallmatrix} \right)$ and  $g_2 = \left(\begin{smallmatrix} a_2 & b_2 \\ c_2 & d_2 \end{smallmatrix} \right)$ we have 
\begin{align*}
\b{\epsilon} \left(  \b{\gamma}_1 \b{\gamma}_2\right) \b{\epsilon} &= \b{\epsilon} ( (\lambda_1 + \lambda_2 a_1 + \mu_2 c_1, \mu_1 + \lambda_2 b_1 + \mu_2 d_1, \kappa_1 + \kappa_2 -\lambda_2 \transpose{\mu_2} +\lambda_1 \transpose{(\lambda_2 b_1 + \mu_2 d_1)}\\
& \hspace{0.4cm} + (\lambda_2 a_1 + \mu_2 c_1) \transpose{(\lambda_2 b_1 + \mu_2 d_1)} + (\lambda_2 b_1 + \mu_2 d_1) \transpose{\lambda_1})  g_1^{-1} g_2) \b{\epsilon}\\
&=(\lambda_1 + \lambda_2 a_1 + \mu_2 c_1, -(\mu_1 + \lambda_2 b_1 + \mu_2 d_1),
-( \kappa_1 + \kappa_2 +\lambda_1 \transpose{(\lambda_2 b_1 + \mu_2 d_1)}\\
&\hspace{0.4cm} + (\lambda_2 a_1 + \mu_2 c_1) \transpose{(\lambda_2 b_1 + \mu_2 d_1)} -\lambda_2 \transpose{\mu_2}  + (\lambda_2 b_1 + \mu_2 d_1) \transpose{\lambda_1})) \epsilon g_1^{-1} g_2 \epsilon.
\end{align*}

On the other hand,
\[
\b{\epsilon} (  \b{\gamma}_1) \b{\epsilon} = (\lambda_1, - \mu_1 ,-\kappa_1) \epsilon g_1^{-1}\epsilon\quad\mbox{ and }\quad
\b{\epsilon} (  \b{\gamma}_2) \b{\epsilon} = (\lambda_2, - \mu_2 ,-\kappa_2) \epsilon g_2\epsilon,
\]
that is,
\[
(\b{\epsilon} (  \b{\gamma}_1) \b{\epsilon})( \b{\epsilon} (  \b{\gamma}_2) \b{\epsilon}) = (\lambda_1, - \mu_1 ,-\kappa_1) \epsilon g_1^{-1}\epsilon (\lambda_2, - \mu_2 ,-\kappa_2) \epsilon g_2\epsilon .
\]
Now note that $(\epsilon g_1^{-1} \epsilon)^{-1} = \epsilon g_1 \epsilon = \left(\begin{smallmatrix} a_1 & -b_1 \\ -c_1 & d_1\end{smallmatrix} \right)$, and so
\begin{align*}
(\lambda_1, - \mu_1 ,-\kappa_1&) \epsilon g_1^{-1}\epsilon (\lambda_2, - \mu_2 ,-\kappa_2) \epsilon g_2\epsilon\\
&=( (\lambda_1 + \lambda_2 a_1 + (-\mu_2)(- c_1), (-\mu_1) + \lambda_2 (-b_1) + (-\mu_2) d_1,\\
&\hspace{0.4cm} (-\kappa_1) + (-\kappa_2) + (\lambda_2 a_1 + (-\mu_2) (-c_1)) \transpose{(\lambda_2 (-b_1) + (-\mu_2) d_1)}-\lambda_2 \transpose{(-\mu_2)}\\
&\hspace{0.4cm} + \lambda_1 \transpose{(\lambda_2 (-b_1) + (-\mu_2) d_1)} +(\lambda_2 (-b_1) + (-\mu_2) d_1) \transpose{\lambda_1})  \epsilon g_1^{-1} g_2 \epsilon,
\end{align*}
which shows that the map is a group automorphism of the group $\b{G}^n$. 

\begin{prop}\label{Behaviour under complex conjugation} Let $\b{\gamma} = (\lambda,\mu,\kappa) \gamma \in \b{G}$. Then
\[
(f|_{k,S} \b{\gamma})^c = f^c|_{k,S}\b{\epsilon} \b{\gamma} \b{\epsilon} .
\]
Moreover, if $f$ is an eigenform with $f| T_{\psi}(\f{a}) = \lambda(\f{a}) f$ for all fractional ideals $\f{a}$ prime to $\f{c}$, then so is $f^c$. In particular, $f^c |T_{\psi}(\f{a}) = \lambda(\f{a}) f^c$ and $L_{\psi,\f{c}}(s,f,\chi) = L_{\psi,\f{c}}(s,f^c,\chi)$. 
\end{prop}

\begin{proof} Write $\b{\gamma} = (\lambda,\mu,\kappa) \left(\begin{smallmatrix} a& b \\ c & d \end{smallmatrix}\right)$, so that $\b{\epsilon} \b{\gamma} \b{\epsilon}  = (\lambda, -\mu,-\kappa ) \left(\begin{smallmatrix} a& -b \\- c & d \end{smallmatrix}\right)$. Then
\begin{align*}
(f&|_{k,S} \b{\gamma})(z) = 
\det(c\tau+d)^{-k} f((a\tau+b)(c\tau+d)^{-1}, w (c\tau+d)^{-1} + \lambda (a \tau+b)(c\tau+d)^{-1} + \mu)\\
&\hspace{-0.4cm}\cdot\mathbf{e_a}(-\tr(S\kappa) + \tr(S[w](c\tau+d)^{-1}c) - 2 \tr(S(\lambda,w)(c\tau+d)^{-1}) - \tr(S[\lambda] (a\tau+b)(c\tau+d)^{-1}) )^{-1} ,
\end{align*}
and so 
\begin{align*}
(&f|_{k,S} \b{\gamma})^c(z)=\overline{\det(c(-\overline{\tau}+d)^{-k}}
\overline{ \mathbf{e_a}(-\tr(S\kappa) + \tr(S[-\overline{w}](c(-\overline{\tau})+d)^{-1}c)}\\
& \hspace{0.4cm}\overline{- 2 \tr(S(\lambda,-\overline{w})(c(-\overline{\tau})+d)^{-1})- \tr(S[\lambda] (a(-\overline{\tau})+b)(c(-\overline{\tau}+d)^{-1}) )^{-1} } \\
&\hspace{0.2cm}\cdot\overline{f((a(-\overline{\tau})+b)(c(-\overline{\tau})+d)^{-1}, -\overline{w} (c(-\overline{\tau})+d)^{-1} + \lambda (a (-\overline{\tau})+b)(c(-\overline{\tau})+d)^{-1} + \mu)}\\
&=\overline{f((a(-\overline{\tau})+b)(c(-\overline{\tau})+d)^{-1}, -\overline{w} (c(-\overline{\tau})+d)^{-1} + \lambda (a (-\overline{\tau})+b)(c(-\overline{\tau})+d)^{-1} + \mu)}\\
&\hspace{0.4cm}\cdot\det(-c\tau+d)^{-k} \mathbf{e_a}(\tr(S\kappa) -\tr(S[w](-c\tau+d)^{-1}c) + 2 \tr(S(\lambda,-w)(-c\tau+d)^{-1})\\
&\hspace{0.4cm} + \tr(S[\lambda] (-a\tau+b)(-c\tau+d)^{-1}) )^{-1}.
\end{align*}
On the other hand
\begin{align*}
f^c&|_{k,S}\b{\epsilon} \b{\gamma} \b{\epsilon}=\det(-c\tau+d)^{-k} \mathbf{e_a}(\tr(S\kappa ) + \tr(S[w](-c\tau+d)^{-1}(-c))\\
&\hspace{0.4cm} - 2 \tr(S(\lambda,w)(-c\tau+d)^{-1}) - \tr(S[\lambda] (a\tau-b)(-c\tau+d)^{-1}) )^{-1}\\
&\hspace{0.4cm}\cdot\overline{f(-\overline{(a\tau-b)(-c\tau+d)^{-1}}, -\overline{(w (-c\tau+d)^{-1} + \lambda (a \tau-b)(-c\tau+d)^{-1} - \mu))}}\\
&= \overline{f((a(-\overline{\tau})+b)(c(-\overline{\tau})+d)^{-1}, -\overline{w} (c(-\overline{\tau})+d)^{-1} + \lambda (a (-\overline{\tau})+b)(c(-\overline{\tau})+d)^{-1} + \mu))}\\
&\hspace{0.4cm}\det(-c\tau+d)^{-k} \mathbf{e_a}( \tr(S \kappa) + \tr(S[w](-c\tau+d)^{-1}(-c)) - 2 \tr(S(\lambda,w)(-c\tau+d)^{-1})\\
&\hspace{0.4cm} - \tr(S[\lambda] (a\tau-b)(-c\tau+d)^{-1}) )^{-1}, 
\end{align*}
which establishes the first statement of the proposition. 

Now assume that $f$ is an eigenform of $T(\f{a})$ with eigenvalues $\lambda(\f{a})$ for all integral ideals $\f{a}$. Because the map \eqref{epsilon_automorphism} is a group automorphism, we see that for any $r\in Q(\f{e})$ if $\b{G}^n(F) \cap \b{D}\diag[\tilde{r},r]\b{D} = \coprod_{\b{\gamma}} \b{\Gamma} \b{\gamma}$, then also
$\b{G}^n(F) \cap \b{D}\diag[\tilde{r},r]\b{D} = \coprod_{\b{\gamma}} \b{\Gamma} \b{\epsilon} \b{\gamma} \b{\epsilon}.$
This means that $f^c|T_{r,\psi}= (f|T_{r,\psi})^c$. In particular,
\[
f^c|T_{\psi}(\f{a})= (f|T_{\psi}(\f{a}))^c = (\lambda(\f{a}) f)^c = \overline{\lambda(\f{a})} f^c
\] 
for all integral ideals $\f{a}$. 
However, since $0 \neq f$, then $<f,f> \neq 0$ and thus the equality
\[
\lambda(\f{a})<f,f> = <f|T_{\psi}(\f{a}),f> = <f,f|T_{\psi}(\f{a})> = <f,f> \overline{\lambda(\f{a})}
\] 
implies that the eigenvalues $\lambda(\f{a})$ are totally real. The last statement regarding the $L$-functions is now obvious. 
\end{proof}
\section{Analytic properties of Siegel-type Jacobi Eisenstein series} \label{section of analytic properties of Eisenstein series}

In the previous section we introduced the standard $L$-function attached to a Siegel-Jacobi eigenfunction. Our first aim is to study its analytic properties using the identity \eqref{main_inner_product}. However, in order to do this we need to establish first the analytic properties of the Siegel-type Jacobi Eisenstein series with respect to the parameter $s$. This is the subject of this section. More precisely, we will establish the analytic continuation and detect possible poles of this Eisenstein series. The main idea of our method goes back to B\"{o}cherer \cite{B83}, which was further extended by Heim in \cite{Heim}, and its aim is to relate Jacobi Eisenstein series of Siegel type to symplectic Eisenstein series (of Siegel type). We extend their results to include level, character and - more importantly - we deal also with the case of totally real field. This last generalization requires development of some new techniques in case the class number is not trivial. In this section the Jacobi Eisenstein series is denoted by a bold $\b{E}$, and the symplectic by a normal $E$.\newline

We start with the following lemma, which gives us good representatives for the sets $(\b{P}^n \cap \zeta \b{\Gamma} \zeta^{-1}) \setminus \zeta \b{\Gamma}$, where $\zeta \in \Sp_n(F)$, and $\b{\Gamma}$ is a congruent subgroup of the form $H \rtimes \Gamma_0(\mathfrak{b},\mathfrak{c})$.  

\begin{lem}\label{lem:representatives_for_ES} A set of representatives for the left cosets $(\b{P}^n \cap \zeta \b{\Gamma} \zeta^{-1}) \setminus \zeta \b{\Gamma}$ is given by 
\[
(\lambda, 0 ,0) \gamma,\,\,\,\, \lambda \in M_{l,n}(\mathfrak{o}),\,\,\,\gamma \in P \cap \zeta  \Gamma_0(\mathfrak{b},\mathfrak{c}) \zeta^{-1} \setminus \zeta \Gamma_0(\mathfrak{b},\mathfrak{c}).
\]
\end{lem}
\begin{proof}
First note that $\zeta \b{\Gamma} = \zeta (H \rtimes \Gamma_0(\mathfrak{b},\mathfrak{c})) = H \rtimes \zeta \Gamma_0(\mathfrak{b},\mathfrak{c})$ and, similarly, $\b{P}^n \cap \zeta \b{\Gamma} \zeta^{-1} 
= \b{P}^n \cap (H \rtimes \zeta \Gamma_0(\mathfrak{b},\mathfrak{c}) \zeta^{-1})$, which is nothing else than the set $(H^n_0 \cap H) \rtimes (P \cap \zeta \Gamma_0(\mathfrak{b},\mathfrak{c}) \zeta^{-1})$. Now, since 
\[
(P \cap \zeta \Gamma_0(\mathfrak{b},\mathfrak{c}) \zeta^{-1}) H = H (P \cap \zeta \Gamma_0(\mathfrak{b},\mathfrak{c}) \zeta^{-1}),
\]
a set of representatives for the cosets is given by a product of representatives for $(H^n_0 \cap H) \setminus H$ and for $(P \cap \zeta \Gamma_0(\mathfrak{b},\mathfrak{c}) \zeta^{-1}) \setminus \zeta \Gamma_0(\mathfrak{b},\mathfrak{c}).$ This is precisely the statement of the lemma.
\end{proof}

Now recall the expression \eqref{eq:SiegelES} for an Eisenstein series of Siegel type:
\[
\b{E}(z,s) =\sum_{\zeta \in Z} N(\mathfrak{a}(\zeta))^{2s}\sum_{\gamma \in Q_{\zeta}}\chi[\gamma] \delta(z)^{s-k/2}|_{k,S} \gamma,
\]
where $Q_{\zeta} = (P \cap \zeta \Gamma_0(\mathfrak{b},\mathfrak{c}) \zeta^{-1}) \setminus \zeta \Gamma_0(\mathfrak{b},\mathfrak{c})$.

We set $\b{E}_{\zeta} (z,s) := \sum_{\gamma \in Q_{\zeta}}\chi[\gamma] \delta(z)^{s-k/2}|_{k,S} \gamma.$ Clearly, the analytic continuation of $\b{E}(z,s)$ and its set of possible poles would follow by establishing such a result for all the  $\b{E}_{\zeta} (z,s)$, as $\zeta \in Z$. \newline

If we write $\gamma = hg$ and $z=(\tau,w)$, then
\[
\b{E}_{\zeta}(z,s) = \sum_{\gamma \in Q_{\zeta}}\chi[\gamma] \delta(z)^{s-k/2}|_{k,S} \gamma = \sum_{\gamma \in Q_{\zeta}}\chi[\gamma] J_{k,S}(\gamma,z)^{-1} \delta(g \tau)^{s-k/2}.
\]
Further, by Lemma \ref{lem:representatives_for_ES}, 
\begin{align*}
\b{E}_{\zeta}(z,s) &= \sum_{g \in Q_{\zeta}} \chi[g] j(g,\tau)^{-k} \delta(g \tau)^{s-k/2} \mathbf{e_a}(-\tr(S[w](c_g \tau + d_g)^{-1} c_g))\\
&\hspace{0.4cm}\cdot\sum_{\lambda \in M_{l,n}(\mathfrak{o})} \mathbf{e_a}( 2\tr(\T{\l}Sw (c_g\tau +d_g)^{-1}) +\tr(S[\lambda] g \cdot \tau)).
\end{align*}

For a lattice $L$ in $M_{l,n}(F)$ we define the Jacobi theta series
\[
\Theta_{S,L}(z) = \Theta_{S,L}(\tau,w) := \sum_{\lambda \in L} \mathbf{e_a}( 2\tr(\T{\l}Sw) +\tr(S[\lambda] \tau)).
\]
Recall (Lemma \ref{Parabolic decomposition}) that the elements $\zeta$ may be selected in the form $\diag[1_{n-1}, a_\zeta, 1_{n-1}, a_\zeta^{-1}]$. In particular, for an element $ g \in Q_{\zeta}$ of the form $g = \zeta g_1$,
\[
c_g\tau +d_g = (c_{\zeta} (g_1 \tau) + d_{\zeta})(c_{g_1}\tau + d_{g_1}) = \(\begin{smallmatrix} 1_{n-1} & \\ & a_{\zeta}^{-1}\end{smallmatrix}\) (c_{g_1}\tau + d_{g_1})
\]
and 
\[
g \cdot \tau = \zeta g_1 \cdot \tau = \(\begin{smallmatrix} 1_{n-1} & \\ & a_{\zeta}\end{smallmatrix}\) (g_{1} \cdot \tau)\(\begin{smallmatrix} 1_{n-1} & \\ & a_{\zeta}\end{smallmatrix}\). 
\]
That is, we may write 
$$\sum_{\lambda \in M_{l,n}(\mathfrak{o})} \mathbf{e_a}( 2\tr(\T{\l}Sw (c_g\tau +d_g)^{-1}) +\tr(S[\lambda] g \cdot \tau))=
\Theta_{S,\Lambda_{a_\zeta}} (g_1 \cdot \tau, w (c_{g_1} \tau + d_{g_1})^{-1}),$$
where $\Lambda_{a_\zeta} := M_{l,n}(\mathfrak{o})\(\begin{smallmatrix} 1_{n-1} & \\ & a_{\zeta}\end{smallmatrix}\)$ and $g = \zeta g_1$. Moreover, because $c_g = \(\begin{smallmatrix} 1_{n-1} & \\ & a_{\zeta}^{-1}\end{smallmatrix}\) c_{g_1}$,
\[
\mathbf{e_a}(\tr(S[w](c_g \tau + d_g)^{-1} c_g)) = \mathbf{e_a}(\tr(S[w](c_{g_1} \tau + d_{g_1})^{-1} c_{g_1})).
\]
Hence,
\[
\b{E}_{\zeta}(z,s) = \sum_{g \in Q_{\zeta}} \chi[g] j(g,\tau)^{-k} \delta(g \tau)^{s-k/2} \mathbf{e_a}(-\tr(S[w](c_{g_1} \tau + d_{g_1})^{-1} c_{g_1})) \Theta_{S,\Lambda_{a_\zeta}} (g_1 z).
\]

We now set $\Gamma^{\theta} := \Sp_{n}(F) \cap D^{\theta}$, where $D^{\theta} := D[\mathfrak{b}^{-1},\mathfrak{b}]$,  if $l$ is even, and $D^{\theta} := D[\mathfrak{b}^{-1},\mathfrak{b}] \cap D[2 \mathfrak{d}^{-1}, 2 \mathfrak{d}]$ if $l$ is odd. For $\gamma \in \Gamma^{\theta}$, $\tau \in \mathbb{H}^{\mathbf{a}}$ let $j(\gamma,\tau)^{1/2} : = h(\gamma,\tau)$, where $h$ is the half-integral factor of automorphy as defined for example in \cite[page 180]{Sh00}. Then for $l$ odd and $\gamma \in \Gamma^{\theta}$ we have
\[
j(\gamma,\tau)^{l/2} = h(\gamma,\tau) j(\gamma,\tau)^{[l/2]}.
\]
Therefore it makes sense to define 
\[
 \Theta_{S,\Lambda_{a_\zeta}} (z)|_{S,l/2} \gamma := h(\gamma,\tau)^{-1} J_{S, [l/2]} (\gamma,z)^{-1} \Theta_{S,\Lambda_{a_\zeta}}(\gamma z),\quad \gamma \in \Gamma^{\theta} .
\]
In fact, for a sufficiently deep subgroup $\Gamma_{a_\zeta}$ of finite index in $\Gamma_0(\mathfrak{b},\mathfrak{c})) \cap D^{\theta}$ we have that (see \cite{Sh00})
\[
 \Theta_{S,\Lambda_{a_\zeta}} (z)|_{S,l/2} g_1 = \psi_{S}(g_1) \Theta_{S,\Lambda_{a_\zeta} }(z) ,\quad \mbox{for all } g_1 \in \Gamma_{a_\zeta}, 
 \]
where $\psi_S$ is the Hecke character of $F$ corresponding to the extension $F(\det(2S)^{1/2})/F$ if $l$ is odd, and to the extension $F((-1)^{l/4} \det(2S)^{1/2})/F$ if $l$ is even. 

Moreover, for every $g \in Q_{\zeta}$ such that $g = \zeta g_1$, $g_1 \in\Gamma_{0}(\f{b},\f{c})$, we have
\begin{multline*}
\chi[g] j(g,\tau)^{-k} \delta(g \tau)^{s-k/2}\mathbf{e_a}(-\tr(S[w](c_{g_1} \tau + d_{g_1})^{-1} c_{g_1})) \Theta_{S,\Lambda_{a_\zeta}} (g_1 z)\\
= N_{F/\mathbb{Q}}(a_\zeta)^{l/2} \psi_S(a_{\zeta})\phi[g] j(g,\tau)^{-(k-l/2)} \delta(g \tau)^{s-k/2}  (\Theta_{S,\Lambda_{a_\zeta}} (z)|_{S,l/2} g_1),
\end{multline*}
where $\phi : = \chi \psi_S$, and we have used the fact that 
\[
j(g,\tau) =j(\zeta g_1, \tau) = j(\zeta, g_1\cdot\tau) j(g_1,\tau) =N_{F/\mathbb{Q}}(a_{\zeta})^{-1} j(g_1,\tau).  
\] 
In particular, if we set $Q'_\zeta := \zeta \Gamma_{a_\zeta}$, we obtain 
\[
\b{E}_{\zeta}(z,s) = N_{F/\mathbb{Q}}(a)^{l/2} \psi_S(a_{\zeta})\sum_{\gamma \in \Gamma_{a_\zeta} \setminus \Gamma_0(\mathfrak{b},\mathfrak{c})} \overline{\chi[\gamma]} (E_{\zeta}(\tau,s-l/4) \Theta_{S,\Lambda_{a_\zeta} }(z) )|_{S,k} \gamma,
\]
where $E_{\zeta}(\tau,s)=\sum_{g \in Q'_{\zeta}} \phi[g] j(g,\tau)^{-(k-l/2)} \delta(g \tau)^{s-k/2+l/4}$ is a symplectic Eisenstein series of Siegel type of weight $k-l/2$. Since the above sum is finite, it follows that the series $\b{E}_{\zeta}(z,s)$ has poles at most at the same places where $E_{\zeta}(\tau,s-l/4)$ may have.

Hence our focus now moves to detect the poles of the series $E_{\zeta}(\tau,s)$. Series of this form appear as summands of the classical (i.e. symplectic) Siegel Eisenstein series of some (perhaps half-integral) weight $k$ and character $\chi$, namely
\[
E(\tau,\chi,s) :=E(\tau,s) =\sum_{\zeta \in Z} N(\mathfrak{a}(\zeta))^{2s}\sum_{\gamma \in R_{\zeta}}\chi[\gamma] \delta(\tau)^{s-k/2}|_{k} \gamma ,
\] 
where 
\[
E_{\zeta}(\tau,\chi, s) := E_{\zeta}(\tau,s) := \sum_{\gamma \in R_{\zeta}}\chi[\gamma] \delta(\tau)^{s-k/2}|_{k} \gamma.
\]

The analytic properties of $E(\tau,s)$ are well known, and thus we may use them to derive similar properties for $E_{\zeta}(\tau,s)$. 

We will use discrete Fourier analysis on the class group $Cl(F)$ of $F$. Recall that $Cl(F) \cong \mathbb{A}^{\times}_{F} /F^{\times}U$, where $U = F^{\times}_{\infty} \prod_v \mathfrak{o}_v^{\times}$.
Moreover, we may pick the representatives $\mathfrak{a}(\zeta)$ for $Cl(F)$ in such a way that the $\zeta$'s form the set of representatives for the set $Z$ (see \cite[Lemma 3.2]{Sh95}). 

Note that for any character $\chi$ and any character $\psi$ of $Cl(F)$, 
\[
E(\tau,\chi \psi, s) =\sum_{\zeta \in Z} \psi(\zeta) N(\mathfrak{a}(\zeta))^{2s}\sum_{\gamma \in R_{\zeta}}\chi[\gamma] \delta(\tau)^{s-k/2}|_{k} \gamma = \sum_{\zeta \in Z} \psi(\zeta) N(\mathfrak{a}(\zeta))^{2s}E_{\zeta}(\tau,s), 
\]
that is, for every character $\psi_i$ of $Cl(F)$,  
\[
E(\tau,\chi \psi_i, s) = \sum_{\zeta \in Z} \psi_i(\zeta) N(\mathfrak{a}(\zeta))^{2s}E_{\zeta}(\tau,s),\,\,\,i = 1,2, \ldots, cl(F),
\]
where $cl(F)$ denotes the cardinality of $Cl(F)$. Since the characters $\psi_i$ are linearly independent over the group $Cl(F)$, we can solve the linear system of equations with respect to the unknowns $N(\mathfrak{a}(\zeta))^{2s}E_{\zeta}(\tau,s)$. In particular, the analytic properties of $E_{\zeta}(\tau,s)$ can be read off from the ones of $E(\tau,\chi \psi_i, s), i = 1,2, \ldots, cl(F)$. 
Hence, since
\[
\b{E}_{\zeta}(z,s) = N_{F/\mathbb{Q}}(a)^{l/2} \sum_{\gamma \in \Gamma_{a_\zeta} \setminus \Gamma_0(\mathfrak{b},\mathfrak{c}))} (E_{\zeta}(\tau,s-l/4) \Theta_{S,\Lambda_{a_\zeta} }(z) )|_{S,k} \gamma,
\]
we see that the analytic properties of $\b{E}$ can be obtained from those of $E(\tau,\chi \psi_i, s)$ for the various $\psi_i$'s. To do that we will employ the following theorem of Shimura \cite{Sh00} on the analytic properties of symplectic Siegel type Eisenstein series, where 
\[
\Gamma_n(s) := \pi^{n(n-1)/4} \prod_{j=0}^{n-1} \Gamma(s- j/2).
\]

\begin{thm}[Shimura, Theorem 16.11 in \cite{Sh00}] \label{analytic properties of Siegel Eisenstein series} For a weight $ k \in \frac{1}{2}\mathbb{Z}^{\mathbf{a}}$ we define 
\[
\mathcal{G}_{k,n}(s) := \prod_{v \in \mathbf{a}} \gamma(s,|k_v|),
\]
where
\[
\gamma(s,h) := \begin{cases} \Gamma\left(s + \frac{h}{2} - \left[\frac{2h+n}{4}\right] \right) \Gamma_n(s + \frac{h}{2}) & \hbox{if } n/2 \leq h \in \mathbb{Z},\,\, n \in 2 \mathbb{Z}, \\
\Gamma_n(s+\frac{h}{2}) & \hbox{if } n/2 < h \in \mathbb{Z},\,\,n \in 2\mathbb{Z} + 1, \\
\Gamma_{2h+1}(s + \frac{h}{2}) \prod_{i = h+1}^{[n/2]} \Gamma(2s-i) & \hbox{if } 0 \leq h < n/2,\,\, h \in \mathbb{Z}, \\
\Gamma \left(s+ \frac{h-1}{2} - \left[\frac{2h+n-2}{4}\right]\right) \Gamma_n(s+h/2) & \hbox{if } n/2 < h \not \in \mathbb{Z},\,\,n \in 2\mathbb{Z} + 1, \\  
\Gamma_n(s + h/2) & \hbox{if } n/2 < h \in \mathbb{Z},\,\,n \in 2\mathbb{Z}, \\
\Gamma_{2h+1}(s + \frac{h}{2}) \prod_{i = [h]+1}^{[(n-1)/2]} \Gamma(2s-i-\frac{1}{2}) & \hbox{if } 0 < h \leq n/2,\,\, h \not \in \mathbb{Z}. 
\end{cases}
\]
We also set $\mathcal{E}(s) := \mathcal{G}(s) \Lambda_{k,\mathfrak{c}}^n(s,\chi) E(z,\chi,s)$, where 
\[
\Lambda^{n}_{k,\mathfrak{c}}(s,\chi) := \begin{cases} L_{\mathfrak{c}}(2s,\chi) \prod_{i=1}^{[n/2]} L_{\mathfrak{c}}(4s-2i,\chi^{2}) & \hbox{if } k \in \mathbb{Z}^{\mathbf{a}},\\
 \prod_{i=1}^{[(n+1)/2]} L_{\mathfrak{c}}(4s-2i+1,\chi^{2})  & \hbox{if } k \not \in \mathbb{Z}^{\mathbf{a}}.
 \end{cases}
\] 
The function $\mathcal{E}(s)$ has a meromorphic continuation to the whole of $\mathbb{C}$ and is entire if $\chi^{2} \neq 1$. If $\chi^{2}=1$, we distinguish two cases:
\begin{enumerate}
\item if $\chi^2=1$ and $\mathfrak{c} \neq \mathfrak{o}$. Set $m := \min_{v \in \mathbf{a}}\{ k_v\}$. Then if $m > n/2$, the function $\mathcal{E}(s)$ has no poles except for a possible simple pole at $s= (n+2)/4$, which occurs only if $2|k_v| - n \in 4 \mathbb{Z}$ for every $v$ such that $2|k_v| > n$. If $m \leq n/2$, then $\mathcal{E}$ has possible poles, which are all simple, in the set
\[
S^{(1)}_k := \begin{cases} \{j/2 : j \in \mathbb{Z}, [(n+3)/2]\leq j \leq n+1-m \} & \hbox{if } k \in \mathbb{Z}^{\mathbf{a}}, \\ 
\{ (2j+1)/4 : j \in \mathbb{Z}, 1 + [n/2] \leq j \leq n + 1/2 - m \} & \hbox{if } k \not \in \mathbb{Z}^{\mathbf{a}}.
 \end{cases}
\]
\item if $\chi^2 = 1$, $\mathfrak{c} = \mathfrak{o}$, and $ k \in \mathbb{Z}^{\mathbf{a}}$.  In this case each pole, which is simple, belongs to the set of poles described in (1) or to 
\[
S_k^{(2)} := \{j/2 : j \in \mathbb{Z}, 0 \leq j \leq [n/2] \},
\]
where $j=0$ is unnecessary if $\chi \neq 1$.
\end{enumerate}

\end{thm}

We can now state a theorem regarding the analytic properties of the Eisenstein series $\b{E}(z,\chi,s)$, which extends a previous theorem due to Heim \cite[Theorem 4.1]{Heim}. Recall that $\psi_S$ is the Hecke character of $F$ corresponding to the extension $F(\det(2S)^{1/2})/F$ if $l$ is odd, and to the extension $F((-1)^{l/4} \det(2S)^{1/2})/F$ if $l$ is even.
\newline

\begin{thm}\label{analytic properties of Eisenstein series} With notation as above, let
\[
\b{\mathcal{E}}(s) := \mathcal{G}_{k-l/2,n}(s-l/4) \Lambda_{k-l/2,\mathfrak{c}}^n(s-l/4,\chi \psi_S) \b{E}(z,\chi,s).
\]
The function $\b{\mathcal{E}}$  has a meromorphic continuation to the whole of $\mathbb{C}$, and its poles are caused by the functions
\[
\frac{\Lambda_{k-l/2,\mathfrak{c}}^n(s-l/4,\chi \psi_S)}{\Lambda_{k-l/2,\mathfrak{c}}^n(s-l/4,\chi \psi_S \psi_i)},\,\,\,i = 1,\ldots,cl(F).
\] 
These poles may appear only when $F$ has class number larger than one and $supp(\mathfrak{c}) \neq supp(cond(\chi\psi_S))$. More precisely:
\begin{enumerate}
\item Assume that $\chi^2 \psi^2_i \neq 1$ for all $i=1,\ldots,cl(F)$. Then $\b{\mathcal{E}}(s)$ has no extra poles.
\item Assume that there exist $\psi_i$ such that $\chi^2 \psi_i^2 =1$. Then we consider the following cases.
\begin{enumerate}
\item $\mathfrak{c} \neq \mathfrak{o}$. Set $m := \min_{v \in \mathbf{a}}\{ k_v - l/2\}$. If $m > n/2$, then the function $\b{\mathcal{E}}(s)$ has no extra poles except for a possible simple pole at $s= (n+2)/4$, which occurs only if $2|k_v -l/2| - n \in 4 \mathbb{Z}$ for every $v$ such that $2|k_v - l/2| > n$. If $m \leq n/2$, then all possible poles of $\mathcal{E}$ are simple and belong to the set $S^{(1)}_{k-l/2}$.

\item $\mathfrak{c} = \mathfrak{o}$, and $ k-l/2 \in \mathbb{Z}^{\mathbf{a}}$.  In this case each extra pole is simple and belongs to the set described in (a) or to 
\[
S_{k-l/2}^{(2)} := \{j/2 : j \in \mathbb{Z}, [0 \leq j \leq [n/2] \},
\]
where $j=0$ is unnecessary if $\chi \psi\neq 1$. 
\end{enumerate}
\end{enumerate}
\end{thm}

Before we proceed to the proof of the theorem we recall the following fact regarding zeros of Dirichlet series. For a Hecke character $\psi$ of $F$ and an integral ideal $\mathfrak{c}$ we considered the series 
\[
L_{\mathfrak{c}}(s,\psi) := \prod_{\mathfrak{q} | \mathfrak{c}} (1- \psi(\mathfrak{q}) N(\mathfrak{q})^{-s}) L(s,\psi)  
\]
with functional equation
\[
\prod_{v \in \mathbf{a}} \Gamma((s+t_v)/2) L(s,\psi) = W(\psi,s) \prod_{v \in \mathbf{a}} \Gamma((1-s+t_v)/2) L(1-s,\psi),
\]
where $W(\psi,s)$ is a non-vanishing holomorphic function, and $t_v \in \{ 0,1\}$ is the infinite type of the character. It is well known that if $\psi \neq 1$, then $L(s,\psi) \neq 0$ for $\Re(s) \geq 1$, and  $\prod_{v \in \mathbf{a}} \Gamma((s+k_v)/2) L(s,\psi)$ is entire. If $\psi = 1$, then this function is meromorphic with simple poles at $s=0$ and $s=1$, and $L(s,\psi) \neq 0$ for $\Re(s) > 1$. 

The absolute convergence and the functional equation imply that if two non-trivial characters $\psi_1$ and $\psi_2$ have the same infinite type, then the zeros of $L(s,\psi_1)$ and $L(s,\psi_2)$ as well as their orders are the same at the integers of the real axis. Namely, for any $ 0 \leq m \in \mathbb{Z}$, $L(-m,\psi_1) = L(-m,\psi_2) = 0$ if and only if there exists $v \in \mathbf{a}$ such that $\psi_1(x_v) = \psi_2(x_v) = \sgn(x_v)^m$. Moreover, the order of the zero equals precisely the number of places where this is happening. In particular, the function
\[
\frac{L_{\mathfrak{c}}(s,\psi_1)}{L_{\mathfrak{c}}(s,\psi_2) } = \left(\prod_{\mathfrak{q} | \mathfrak{c}} \frac{(1- \psi_1(\mathfrak{q})N(\mathfrak{q})^{-s})}{(1- \psi_2(\mathfrak{q})N(\mathfrak{q})^{-s})}\right) \frac{L(s,\psi_1)}{L(s,\psi_2) }
\]  
may have poles only at the integers where $\prod_{\mathfrak{q} | \mathfrak{c}} \frac{(1- \psi_1(\mathfrak{q})N(\mathfrak{q})^{-s})}{(1- \psi_2(\mathfrak{q})N(\mathfrak{q})^{-s})}$ has poles. \newline

If the characters $\psi_1 =1$ and $\psi_2$ have trivial type at infinity, then the same argument as above shows that the function
\[
\frac{L_{\mathfrak{c}}(s,\psi_1)}{L_{\mathfrak{c}}(s,\psi_2) }
\]
may have poles at the integers where the function $\prod_{\mathfrak{q} | \mathfrak{c}} \frac{(1- \psi_1(\mathfrak{q})N(\mathfrak{q})^{-s})}{(1- \psi_2(\mathfrak{q})N(\mathfrak{q})^{-s})}$ has poles. However, this time there may be an additional zero also at $s=0$. This is because at this point the order of vanishing of $L(s,\psi_1)$ is smaller by one from the order of vanishing of $L(s,\psi_2)$. 

\begin{proof}[Proof of Theorem \ref{analytic properties of Eisenstein series}]
First note that since $\psi_i$'s are the characters of $Cl(F) \equiv \mathbb{A}_F^{\times}/F^{\times} U$, where $U=F^{\times}_{\mathbf{a}} \prod_v \mathfrak{o}_v^{\times}$, their signature is trivial, that is, ${\psi_i}_{\infty}(x) = 1$ for all $x \in F_{\mathbf{a}}^{\times}$. In particular, the characters $\chi \psi_S$ and $\chi \psi_S \psi_i$, $i = 1,\ldots,cl(F)$, have the same signature at infinity. The discussion above implies that the functions $\Lambda_{k-l/2,\mathfrak{c}}^n(s-l/2,\chi \psi_S)$ and $\Lambda_{k-l/2,\mathfrak{c}}^n(s-l/2,\chi \psi_S \psi_i)$ have the same zeros on the integers at the real line, and the ratio
\[
\frac{\Lambda_{k-l/2,\mathfrak{c}}^n(s-l/4,\chi \psi_S)}{\Lambda_{k-l/2,\mathfrak{c}}^n(s-l/4,\chi \psi_S \psi_i)}
\]
may have poles in cases indicated in the theorem. However, then (Theorem \ref{analytic properties of Siegel Eisenstein series}) the series 
\[
 \frac{\Lambda_{k-l/2,\mathfrak{c}}^n(s-l/4,\chi \psi_S)}{\Lambda_{k-l/2,\mathfrak{c}}^n(s-l/2,\chi \psi_S \psi_i)}\mathcal{G}_{k-l/2,n}(s-l/4)
\Lambda_{k-l/2,\mathfrak{c}}^n(s-l/4,\chi \psi_S \psi_i) E(\tau,\chi \psi_i \psi_S,s-l/4)
\]
does not have any more poles unless $\chi^2 \psi_i^2 =1$ for some $i$, in which case the poles are as described in the theorem.
\end{proof}

\begin{rem} 
The analytic properties of Jacobi Eisenstein series presented in Theorem \ref{analytic properties of Eisenstein series} were obtained from the well-studied symplectic Eisenstein series via establishing the link between these two kinds of Eisenstein series. However, perhaps one could also try to use the results of Arakawa in \cite{Ar93} on the Fourier coefficients of Jacobi Eisenstein series.
\end{rem}

\section{Analytic continuation of the standard $L$-function} \label{Analytic Continuation}
\newcommand{\ord}{\mathrm{ord}}

We are now ready to establish two main theorems regarding the analytic properties of the standard $L$-function and the Klingen-type Jacobi Eisenstein series. The approach taken here can be regarded as an extension from the symplectic to the Jacobi setting of the method utilized in \cite{Sh95}.

We keep the notation introduced at the beginning of section \ref{Hecke algebra} and additionally we define groups
$$\b{D}':=\{ (\l,\mu,\kappa)x\in C[\f{o}, \f{b}^{-1}, \f{b}^{-1}]D[\f{b}^{-1}\f{c},\f{be}]:(a_x-1_n)_v\in M_n(\f{e}_v)\mbox{ for every } v|\f{e}\} ,$$
$$\b{\Gamma}' :=\b{G}^n(F)\cap \b{D}'$$
and
$$R(\f{e},\f{c}):=\{ \diag [\tilde{q},q]:q\in Q(\f{e}), q_v\in M_n(\f{c}_v) \mbox{ for every } v|\f{e}^{-1}\f{c}\} .$$
For $\diag [\tilde{q},q]\in R(\f{e},\f{c})$ and $f\in M^n_{k,S}(\b{\Gamma},\psi)$, in a manner similar to $f|T_{r,\psi}$, we define
%
%
%
%
\begin{equation}
f|U_{q,\psi}:=\sum_{\b{\beta}\in\b{B}} \psi_{\f{c}}(\det (a_{\beta})_{\f{c}})^{-1} f|_{k,S}\b{\beta},
\end{equation}
where $\b{B}\subset\b{G}^n(F)$ is such that $\b{G}^n(F)\cap\b{D}\diag [\tilde{q},q]\b{D}'=\coprod_{\b{\beta}\in\b{B}}\b{\Gamma\beta}$. As in section \ref{Hecke algebra}, if we write $\mathbf{f}|U_{q,\psi}$ for the  adelic Jacobi form associated to $f|U_{q,\psi}$ (with $\b{g} =1$) and $\b{D}\diag [\tilde{q},q]\b{D}'=\coprod_{\b{\beta}\in\b{B}}\b{D\beta}$ with $\b{B}\subset\b{G}_{\h}$, then 
$$(\mathbf{f}|U_{q,\psi})(x)=\sum_{\b{\beta}\in\b{B}} \psi_{\f{c}}(\det (a_{\beta})_{\f{c}})^{-1} \mathbf{f}(x\b{\beta}^{-1}),\qquad x\in\b{G}^n(\A).$$

For the rest of this section we assume that $f\in S^n_{k,S}(\b{\Gamma},\psi)$ is a non-zero eigenfunction of $T_{\psi}(\f{a})$ for every $\f{a}$ with eigenvalues $\l(\f{a})$. Note that $T_{\psi}(\f{a})\neq 0$ only if $\f{a}$ is coprime to $\f{e}$. 

We start with a version of \cite[Lemma 6.2]{Sh95} for Hecke operators in our Jacobi setting.

\begin{lem}\label{lem:mixed_cosets}
Let $h$ be an element of $\mathbb{A}^{\times}_{\h}$ such that its corresponding ideal is $\f{e}^{-1}\f{c}$ and $h_v=1$ for $v\nmid\f{e}^{-1}\f{c}$. Then $U_{hr,\psi}=T_{r,\psi}U_{h1_n,\psi}$ for every $r\in Q(\f{e})$. Moreover, for $f\in M^n_{k,S}(\b{\Gamma},\psi)$ we have $f|T_{h1_n,\psi}\neq 0$ only if $f|U_{h1_n,\psi}\neq 0$.
\end{lem}
\begin{proof}
To prove the first statement it suffices to show that 
$$\b{D}\(\begin{smallmatrix} h^{-1}\tilde{r} & \\ & hr\end{smallmatrix}\)\b{D}'=\b{D}\(\begin{smallmatrix}\tilde{r} & \\ & r\end{smallmatrix}\)\b{D}\cdot \b{D}\(\begin{smallmatrix} h^{-1}1_n & \\ & h1_n\end{smallmatrix}\)\b{D}'.$$
This may be done place by place. As we established in \eqref{double_coset_dec},
$$\b{D}_v\begin{pmatrix} \tilde{r}_v & \\ & r_v\end{pmatrix}\b{D}_v = \bigsqcup_{d,b,\mu} \b{D}_v (0,\mu,0)\begin{pmatrix}  \tilde{d} & \tilde{d} b \\  & d \end{pmatrix}$$
at each place $v|\f{c}$, where $d \in\GL_n(\f{o}_v) \back\GL_n(\f{o}_v) r_v\GL_n(\f{o}_v)$, $b \in Sym_n(\f{b}_v^{-1}) / \transpose{d} Sym_n(\f{b}_v^{-1}) d$ and $\mu \in M_{l,n}(\f{b}^{-1}_v) d^{-1} / M_{l,n}(\f{b}^{-1}_v)$. Using the same argument and a double coset decomposition for symplectic groups, we get
$$\b{D}_v\begin{pmatrix} h^{-1}_v\tilde{r}_v & \\ & h_vr_v\end{pmatrix}\b{D}'_v = \bigsqcup_{d_1,b_1,\upsilon_1} \b{D}_v (0,\upsilon_1,0)\begin{pmatrix}  \tilde{d}_1 & \tilde{d}_1 b_1 \\  & d_1 \end{pmatrix},$$
where $d_1 \in\GL_n(\f{o}_v) \back\GL_n(\f{o}_v) h_v r_v\GL_n(\f{o}_v)$, $b_1 \in Sym_n(\f{b}_v^{-1}\f{c}_v) / \transpose{d_1} Sym_n(\f{b}_v^{-1}) d_1$ and $\upsilon_1 \in M_{l,n}(\f{b}^{-1}_v) d_1^{-1} / M_{l,n}(\f{b}^{-1}_v)$. In particular, if we take $r=1_n$ and a coset decomposition over $d_2,b_2,\upsilon_2$, then we can take $d_2=h_v1_n$ and it is easy to see that the set
\begin{multline*}
\{(0,\mu,0)\(\begin{smallmatrix} \tilde{d} & \tilde{d} b \\  & d \end{smallmatrix}\) (0,\upsilon_2,0)\(\begin{smallmatrix} h_v^{-1}1_n & h_v^{-1} b_2 \\  & h_v1_n\end{smallmatrix}\)\colon\mu, \upsilon_2, b, b_2, d\} \\
=\{(0,\mu +\upsilon_2 d^{-1},0)\(\begin{smallmatrix} h_v^{-1}\tilde{d} & h_v^{-1}\tilde{d} (b_2+h_v^2b) \\  & h_vd \end{smallmatrix}\)\colon\mu, \upsilon_2, b, b_2, d\}
\end{multline*} 
represents 
$\b{D}_v\back (\b{D} \(\begin{smallmatrix} h^{-1}\tilde{r} & \\ & hr\end{smallmatrix}\)\b{D}')_v$ for each $v|\f{c}$.

To prove the second statement we use Proposition \ref{prop:adelic_Fourier_expansion}. We recall that the Siegel-Jacobi modular form $f$ and its adelic counterpart are related by $\mathbf{f}(\b{y})=J_{k,S}(\b{y}, \b{i}_0)^{-1}f(\b{y}\cdot\b{i}_0)$, for every $\b{y}\in\Jac_{\a}$. Moreover, recall that the symmetric space $\H_{n,l}$ is contained in $\{\b{y}\cdot\b{i}_0:\b{y}\in\Jac_{\a}\mbox{ of the form } (\l,\mu,0)\(\begin{smallmatrix} q & \sigma\tilde{q}\\ & \tilde{q}\end{smallmatrix}\)\}$.

For an $\b{\alpha}$ of the form $(0,\nu,0)\(\begin{smallmatrix} h^{-1}1_n & h^{-1}b \\  & h1_n \end{smallmatrix}\)$, with $\nu_\mathbf{a}=0$, $b_{\mathbf{a}} = 0$, and $\b{y} \in \b{G}(\mathbb{A})$ such that $\b{y}_{\h}=(0,0,0)1_{2n}$ and $\b{y}_{\mathbf{a}}$ as above, we have
\begin{align*}
\b{y}\b{\alpha}^{-1}&=(\l,\mu,0)(0,-h\nu\T{q},0)\(\begin{smallmatrix} hq & h^{-1}(-qb+\sigma\tilde{q}) \\  & h^{-1}\tilde{q} \end{smallmatrix}\)\\
&=(\l,\mu-h\nu\T{q},-h\l q\T{\nu}-h\nu\T{q}\T{\l})\(\begin{smallmatrix} hq & (-qb\T{q}+\sigma)h^{-1}\tilde{q} \\  & h^{-1}\tilde{q} \end{smallmatrix}\) ,
\end{align*}
and thus by the expansion \eqref{eq:adelic_Fourier_expansion}, 
\begin{multline*}
\mathbf{f}(\b{y}\b{\alpha}^{-1})=\sum_{t,r}c(t,r;hq,\l)\mathbf{e}_{\A}(\tr(t\sigma-tqb\T{q}))\mathbf{e}_{\A}(\tr(\T{r}(\l\sigma-\l qb\T{q}+\mu-h\nu\T{q})))=\\
 \sum_{t,r}c(t,r;hq,\l)\!\left(\prod_{v | \f{c}}\mathbf{e}_{v}(\tr(-tq_vb_v\T{q_v}))\right)\hspace{-0.2cm}\left(\prod_{v | \f{c}}\mathbf{e}_{v}(\tr(\T{r}(-h_v\nu_v\T{q_v})))\right)\! \mathbf{e}_{\mathbf{a}}(\tr(t\sigma+\T{r}(\l\sigma+\mu))).
\end{multline*}
Hence,
\begin{align*}
\mathbf{f}|T_{h1_n,\psi}(\b{y}) &= \sum_{b,\nu} \psi_{\f{c}}(h_{\f{c}})^n\sum_{t,r}c(t,r;hq,\l)\left(\prod_{v | \f{c}}\mathbf{e}_{v}(\tr(-tq_vb_v\T{q_v}+\T{r}(-h_v\nu_v\T{q_v})))\right) \\
&\hspace{0.4cm}\cdot\mathbf{e}_{\mathbf{a}}(\tr(t\sigma+\T{r}(\l\sigma+\mu))),
\end{align*}
where $b \in \prod_{v | \f{c}}Sym_n(\f{b}_v^{-1}) / h_v^2 Sym_n(\f{b}_v^{-1})$, and $\nu \in \prod_{v | \f{c}}M_{l,n}(\f{b}^{-1}_v) h_v^{-1} / M_{l,n}(\f{b}^{-1}_v)$. That is, if we write $c(\mathbf{f}|T_{h1_n,\psi};t,r;q,\lambda)$ for the $(t,r)$-coefficient of $\mathbf{f}|T_{h1_n,\psi}$, we have
\[
c(\mathbf{f}|T_{h1_n,\psi};t,r;q,\lambda) = \psi_{\f{c}}(h_{\f{c}})^n\sum_{b,\nu} \left(\prod_{v | \f{c}}\mathbf{e}_{v}(\tr(-tq_vb_v\T{q_v}+\T{r}(-h_v\nu_v\T{q_v})))\right).
\]
Therefore, if 
$$\mathbf{e}_{\h}(\tr (\T{q}tq h^{-2}Sym_n(\f{b}^{-1})))=1\quad\mbox{ and }\quad\mathbf{e}_{\h}(\tr (\T{q}\T{r}M_{l,n}(\f{b}^{-1})))=1,$$ 
then
$$c(\mathbf{f}|T_{h1_n,\psi};t,r;q,\lambda)=N(\f{e}^{-1}\f{c})^{n(l+n+1)} \psi_{\f{c}}(h_{\f{c}})^n c(t,r;hq,\lambda),$$ 
otherwise $c(\mathbf{f}|T_{h1_n,\psi};t,r;q,\lambda) = 0.$

Arguing exactly in the same way we can also conclude that if both
$$\mathbf{e}_{\h}(\tr (\T{q}tq h^{-2}Sym_n(\f{b}^{-1}\f{c})))=1\quad\mbox{ and }\quad\mathbf{e}_{\h}(\tr (\T{q}\T{r}M_{l,n}(\f{b}^{-1})))=1,$$ 
then 
$$c(\mathbf{f}|U_{h1_n,\psi};t,r;q,\lambda)=N(\f{e}^{-1}\f{c})^{nl+n(n+1)/2}\psi_{\f{c}}(h_{\f{c}})^n c(t,r;hq,\lambda),$$ 
otherwise $c(\mathbf{f}|U_{h1_n,\psi};t,r;q,\lambda)=0$, where we write $c(\mathbf{f}|U_{h1_n,\psi};t,r;q,\lambda)$ for the $(t,r)$-coefficient of $\mathbf{f}|U_{h1_n,\psi}$.

Hence, if $\mathbf{f}|U_{h1_n,\psi} = 0$, then $c(\mathbf{f}|U_{h1_n,\psi};t,r;q,\lambda) = 0$ for all $t,r$. In particular, if for a pair $t,r$ both $\mathbf{e}_{\h}(\tr (\T{q}tq h^{-2}Sym_n(\f{b}^{-1}\f{c})))=1$ and $\mathbf{e}_{\h}(\tr (\T{q}\T{r}M_{l,n}(\f{b}^{-1})))=1$, then $c(t,r;hq,\lambda)=0$ and hence also $c(\mathbf{f}|T_{h1_n,\psi};t,r;q,\lambda)=0$. If on the other hand for a pair $t,r$ either $\mathbf{e}_{\h}(\tr (\T{q}tq h^{-2}Sym_n(\f{b}^{-1}\f{c})))\neq 1$ or $\mathbf{e}_{\h}(\tr (\T{q}\T{r}M_{l,n}(\f{b}^{-1})))\neq 1$, then also either $\mathbf{e}_{\h}(\tr (\T{q}tq h^{-2}Sym_n(\f{b}^{-1})))\neq 1$ (since $Sym_n(\f{b}^{-1}\f{c}) \subset Sym_n(\f{b}^{-1})$ ) or $\mathbf{e}_{\h}(\tr (\T{q}\T{r}M_{l,n}(\f{b}^{-1})))\neq 1$, which also implies that $c(\mathbf{f}|T_{h1_n,\psi};t,r;q,\lambda) = 0$. Therefore $\mathbf{f}|T_{h1_n,\psi} = 0$.
\end{proof}

We now fix uniformizers $\pi_v \in \mathfrak{o}_v$ for every finite place $v$ in the support of $\f{e}$. Then for a fractional ideal $\f{t}$ we pick $t \in \mathbb{A}^{\times}_{\h}$, such that $\f{t}$ is the ideal corresponding to the idele $t$, and at every place $v | \mathfrak{e}$ we have $t_v = \pi_v^{\ord_v(\f{t})}$, where $\ord_v(\cdot)$ is the usual valuation at the place $v$. Further, we set $\b{\tau}:=1_H\diag [t^{-1}1_n,t1_n]$ and define an isomorphism
$$I_{\f{t}}\colon\mathcal{M}^n_{k,S}(\b{D},\psi)\rightarrow\mathcal{M}^n_{k,S}(\b{\tau}^{-1}\b{D}\b{\tau},\psi),\quad \mathbf{f}|I_{\f{t}}(x):=\psi(t^n)\mathbf{f}(x\b{\tau}^{-1})\quad (x\in\b{G}^n(\A)).$$ 

\begin{lem}
The map $I_{\f{t}}$ has the following properties:
\begin{enumerate}
\item it is independent of the choice of $t$,
\item it commutes with the operators $T_{r,\psi}$ and $U_{q,\psi}$,
\item $(f|I_{\f{t}})^c=f^c|I_{\f{t}}$, where $f$ is the Siegel-Jacobi form corresponding to $\mathbf{f}$.
\end{enumerate}
\end{lem}
\begin{proof}
\begin{enumerate}
\item  If $t'\in \mathbb{A}^{\times}_{\h}$ is another idele that corresponds to the ideal $\f{t}$, then  $t=t'l$ for some $l\in \prod_{v \in \mathbf{h}}\f{o}_v^{\times}$.
\begin{align*}
\psi(t^n)\mathbf{f}(x\b{\tau}^{-1})&=\psi((lt')^n)\mathbf{f}(x\diag[t' 1_n, t'^{-1}1_n]\diag[l 1_n, l^{-1}1_n]1_H)\\
&=\psi(t'^n) \mathbf{f}(x\diag[t'1_n, t'^{-1}1_n]1_H),
\end{align*}
where we have used the fact that $\diag[l 1_n, l^{-1}1_n] \in \b{D}$ since $l_v =1$ if $v | \f{e}$.
\item  This follows from direct computation, e.g. in case of $T_{r,\psi}$:
$$\b{\tau}^{-1}\b{D}\diag [\tilde{r},r]\b{D}=
\b{D}_{\f{t}}\diag [\tilde{r},r]\b{D}_{\f{t}}\b{\tau}^{-1},$$
where $$\b{D}_{\f{t}}:=\{ (\l,\mu,\kappa)x\in C[\f{t}^{-1}, \f{tb}^{-1}, \f{tb}^{-1}]D[\f{b}^{-1}\f{e}\f{t}^2,\f{bc}\f{t}^{-2}]\colon (a_x-1_n)_v\in M_n(\f{e}_v)\mbox{ for } v|\f{e}\} .$$
\item  By strong approximation we may write $\tau = \b{\gamma} \b{d}$ for some $\b{\gamma} \in \b{G}(F)$ and $\b{d} \in \b{D}$. We moreover notice that since $\b{\tau}$ has no Heisenberg part we may take $\b{\gamma} = \gamma \in G(F) \hookrightarrow \b{G}(F)$, and $\b{d} \in D \hookrightarrow \b{D}$. Furthermore, for $\epsilon := \diag[1_n , -1_n]$, $\epsilon \b{\tau} \epsilon^{-1} = \epsilon \gamma \epsilon^{-1} \epsilon \b{d} \epsilon^{-1}$ as elements of $G(F)$. Note that  $\epsilon \b{d} \epsilon^{-1} \in \b{D}$ and $\epsilon \b{\tau} \epsilon^{-1} = \b{\tau}$.

\noindent Clearly, without loss of generality we may assume that $\psi = 1$. Then $(f |I_{\f{t}})^c = (f|_{k,S} \b{\gamma})^c = f^c|_{k,S}\epsilon \gamma \epsilon^{-1}  = f^c|I_{\f{t}}$, where for the second equality we have used Proposition \ref{Behaviour under complex conjugation}.
\end{enumerate}
\end{proof}

Let $\chi$ be a Hecke character as in subsection \ref{sec:adelic_ES} and assume that $\chi =\psi$ on $\prod_{v\nmid\f{e}}\f{o}_v\ti$. Then $\mathcal{S}^n_{k,S}(\b{D},\psi)=\mathcal{S}^n_{k,S}(\b{D},\chi)$ since the nebentype depends only on the finite places that divide $\mathfrak{c}$ and is trivial on places that divide $\f{e}$ ($\det(a_g) \equiv 1\hspace{-0.1cm} \mod{\f{e}_v}$ for $hg \in \b{D}$). Moreover, the Hecke operators are related via:
$$(\chi/\psi)^*(\f{a})\psi^*(\f{a}')T_{\psi}(\f{a})=\chi^*(\f{a}')T_{\chi}(\f{a}),\quad (\chi/\psi)^*(\f{e}^{-1}\f{c})^nU_{h1_n,\psi}=U_{h1_n,\chi} ,$$
where $\f{a}':=\prod_{v\nmid\f{c}}\f{a}_v$. Put $\b{\tau}:=1_H\diag [\t^{-1}1_n,\t1_n]$ with $\t$ as in Lemma \ref{key_lemma}. Then the set $\b{Y}_v$ is equal to the set $(\b{\tau}^{-1}\b{D}R(\f{e},\f{c})\b{D}'\b{\tau})_v$ at every place $v$. Put
$$\Delta (q):=\(\b{G}^n(F)\cap\b{\tau}^{-1}\b{D}\b{\tau}\) \back\(\b{G}^n(F)\cap\b{G}^n_{\a}\prod_{v\in\h}(\b{\tau}^{-1}\b{D}\(\begin{smallmatrix} \tilde{q} & \\ & q\end{smallmatrix}\)\b{D}'\b{\tau})_v\) .$$
For $f\in S^n_{k,S}(\b{\Gamma},\psi)$ such that $f|T_{\psi}(\f{a})=\l(\f{a})f$ and for $\mathcal{D}$ defined as in \eqref{eq:D(z,s,g)} we have: 
\begin{align*}
\mathcal{D}(z ,s,f|I_{\f{b}}) &=
\sum_{\b{\xi} \in\b{Y}} \ell'(\xi)^{-s}\chi^{*} (\ell'_1(\xi)) \chi_{\f{c}}(\det(a_{\xi}))^{-1} (f|I_{\f{b}})|_{k,S}\b{\xi}(z)\\
&=\sum_{q\in R(\f{e},\f{c})}\sum_{\b{\beta}\in\Delta(q)} N(\det (q)\f{o})^{-s}\chi^{*} (\prod_{v\nmid\f{c}}(\det (q)\f{o})_v)\chi_{\f{c}}(\det(a_{\beta}))^{-1}(f|I_{\f{b}})|_{k,S}\b{\beta}(z)\\
&\hspace{-0.6cm}\stackrel{Lemma\,\ref{lem:mixed_cosets}}{=} N(\f{e}^{-1}\f{c})^{-ns}\sum_{\f{a}}N(\f{a})^{-s}\chi^{*}(\f{a}')(f|I_{\f{b}})|T_{\chi}(\f{a})U_{h1_n,\chi}(z)\\
&=N(\f{e}^{-1}\f{c})^{-ns}\sum_{\f{a}}N(\f{a})^{-s}(\chi/\psi)^*(\f{a})\psi^*(\f{a}')\l(\f{a})f|U_{h1_n,\chi}I_{\f{b}}(z).
\end{align*}

Joining the above formula for $\mathcal{D}(z ,s,f|I_{\f{b}})$ together with \eqref{main_inner_product}, after setting $f^c|I_{\f{b}}$ for $f$ there, we obtain
\begin{align*}
N&(\f{b}\f{e}^{-1}\f{c})^{2ns}\chi_{\h}(\theta)^{-n}(-1)^{n(s-k/2)}vol(A)<(E |_{k,S} \b{\rho}) (\diag[z_1,z_2],s),(f^c|I_{\f{b}})(z_2)>\\
&=\nu_{\mathfrak{e}}c_{S,k}(s-k/2)E(z_1,s;(f|U_{h1_n,\chi}I_{\f{b}})|_{k,S}\b{\eta}_n^{-1}) \sum_{\f{a}}N(\f{a})^{-2s}(\chi/\psi)^*(\f{a})\psi^*(\f{a}')\l(\f{a}),
\end{align*}
where we have used the fact that $(f^c|I_{\f{b}})^c = f|I_{\f{b}}$. \newline

After multiplying both sides of the above equation with $\mathcal{G}_{k-l/2,n+m}(s-l/4) \Lambda_{k-l/2,\mathfrak{c}}^{n+m}(s-l/4,\chi \psi_S)$ with notation as in Theorem \ref{analytic properties of Eisenstein series} and setting $\mathcal{E}(z,s) := \mathcal{G}_{k-l/2,n+m}(s-l/4) \Lambda_{k-l/2,\mathfrak{c}}^{n+m}(s-l/4,\chi \psi_S) E(z,s)$, we obtain
\begin{align*}
N(\f{b}\f{e}^{-1}\f{c})^{2ns}\chi_{\h}(\theta)^{-n}&(-1)^{n(s-k/2)}vol(A) <(\mathcal{E} |_{k,S} \b{\rho}) (\diag[z_1,z_2],s),(f^c|I_{\f{b}})(z_2)>\\
&=\nu_{\mathfrak{e}}c_{S,k}(s-k/2) \mathcal{G}_{k-l/2,n+m}(s-l/4) E(z_1,s;(f|U_{h1_n,\chi}I_{\f{b}})|_{k,S}\b{\eta}_n^{-1})\\
&\hspace{0.4cm}\cdot \Lambda_{k-l/2,\mathfrak{c}}^{n+m}(s-l/4,\chi \psi_S)\sum_{\f{a}}N(\f{a})^{-2s}(\chi/\psi)^*(\f{a})\psi^*(\f{a}')\l(\f{a}),
\end{align*}
where we recall that 
\[
\Lambda^{n+m}_{k-l/2,\mathfrak{c}}(s,\chi\psi_S) := \begin{cases} L_{\mathfrak{c}}(2s-l/2,\chi \psi_S) \prod_{i=1}^{[(n+m)/2]} L_{\mathfrak{c}}(4s-l-2i,\chi^{2}) & \hbox{if } l \in 2\mathbb{Z},\\
 \prod_{i=1}^{[(n+m+1)/2]} L_{\mathfrak{c}}(4s-l-2i+1,\chi^{2}) & \hbox{if } l \notin 2\mathbb{Z}.
 \end{cases}
\] 
By the discussion in subsection \ref{The twisted L-function}, we have that 
\[
\f{L}_{\psi}(\chi \psi^{-1},2s-n-l/2) \sum_{\f{a}}N(\f{a})^{-2s}(\chi/\psi)^*(\f{a})\psi^*(\f{a}')\l(\f{a}) =  L_{\psi}(2s-n-l/2,\mathbf{f},\chi \psi^{-1}) 
\]
with 
$\f{L}_{\psi}(\chi \psi^{-1},2s-n-l/2) = \prod_{(\f{p},\f{c})=1} \f{L}_{\f{p}}(\chi,2s-n-l/2)$, where
\[
\f{L}_{\f{p}}(\chi,2s) := G_{\f{p}}(\chi,2s-n-l/2) \begin{cases}  \prod_{i=1}^{n} L_{\f{p}}(4s-l-2i,\chi^{2}) &\mbox{if } l \in 2\mathbb{Z}, \\ 
\prod_{i=1}^{n} L_{\f{p}}(4s-l-2i+1,\chi^{2})  & \mbox{if }  l \not \in 2\mathbb{Z}.\end{cases}
\]
That is, we obtain
\begin{align} \label{Main identity 1}\nonumber
N(\f{b}\f{e}^{-1}\f{c})^{2ns}&\chi_{\h}(\theta)^{-n}(-1)^{n(s-k/2)} vol(A) <(\mathcal{E}|_{k,S} \b{\rho}) (\diag[z_1,z_2],s),(f^c|I_{\f{b}})(z_2)>\\\nonumber
&=\nu_{\f{e}}c_{S,k}(s-k/2)\mathcal{G}_{k-l/2,n+m}(s-l/4)  E(z_1,s;(f|U_{h1_n,\chi}I_{\f{b}})|_{k,S}\b{\eta}_n^{-1}) \\
&\hspace{0.4cm}\cdot G(\chi,2s-n-l/2)^{-1} L_{\psi}(2s-n-l/2,\mathbf{f},\chi \psi^{-1})\\\nonumber
&\cdot\begin{cases} L_{\mathfrak{c}}(2s-l/2,\chi \psi_S) \prod_{i=n+1}^{[(n+m)/2]} L_{\mathfrak{c}}(4s-l-2i,\chi^{2}) & \hbox{if } l \in 2\mathbb{Z},\\
\prod_{i=n+1}^{[(n+m+1)/2]} L_{\mathfrak{c}}(4s-l-2i+1,\chi^{2}) & \hbox{if } l \not \in 2\mathbb{Z},
 \end{cases}
\end{align}
where we have set 
\begin{equation}\label{definition of $G$}
G(\chi,2s-n-l/2) = \prod_{(\f{p},\f{c})=1} G_{\f{p}}(\chi,2s-n-l/2). 
\end{equation}
In particular, if $m=n$, we obtain
\begin{align}\label{Main identity 2}
N&(\f{b}\f{e}^{-1}\f{c})^{2ns}\chi_{\h}(\theta)^{-n}(-1)^{n(s-k/2)}  vol(A)<(\mathcal{E}|_{k,S} \b{\rho}) (\diag[z_1,z_2],s),(f^c|I_{\f{b}})(z_2)>\nonumber\\
&=\nu_{\f{e}}c_{S,k}(s-k/2)\mathcal{G}_{k-l/2,2n}(s-l/4)  (f|U_{h1_n,\chi}I_{\f{b}})|_{k,S}\b{\eta}_n^{-1} G(\chi,2s-n-l/2)^{-1} \\\nonumber
&\hspace{0.4cm} \cdot L_{\psi}(2s-n-l/2,\mathbf{f},\chi \psi^{-1})\begin{cases} L_{\mathfrak{c}}(2s-l/2,\chi \psi_S) ,\,\,\,\hbox{if}\,\,\, l \in 2\mathbb{Z},\\
 1, \,\,\,\hbox{if}\,\,\, l \not \in 2\mathbb{Z}.
 \end{cases}
\end{align}
We are now ready to prove our first main theorem regarding the analytic properties of the function $L_{\psi}(s,\mathbf{f},\chi)$, which should be seen as an extension of the Theorem 6.1 in \cite{Sh95} to the Siegel-Jacobi setting.

\begin{thm} \label{Main Theorem 1} Let $\mathbf{f} \in \mathcal{S}_{k,S}^n(\b{D},\psi)$ be a Hecke eigenform of index $S$ which satisfies the $M_{\f{p}}^+$ condition for every prime $\f{p} \nmid \f{c}$. Moreover, let $\phi$ be a Hecke character of $F$ of conductor $\f{f}_{\phi}$ such that $\phi_{\mathbf{a}}(x) = \sgn(x_{\mathbf{a}})^{k}$. Write $\f{x}$ for the product of all primes ideals $\f{p}$ in the support of $\mathfrak{e}^{-1}\f{c}$ such that $\mathbf{f}|T_{\pi_{\f{p}}1_n,\psi} = 0$. Then the function
\[
\Lambda_{\psi,\f{x}}(s,\mathbf{f},\phi) :=  L_{\mathbf{a}}(s,k) L_{\psi,\f{x}}(s,\mathbf{f},\phi) \cdot \begin{cases} L_{\mathfrak{c}}(s+n,\phi \psi \psi_S) ,\,\,\,\hbox{if}\,\,\, l \in 2\mathbb{Z},\\
 1, \,\,\,\hbox{if}\,\,\, l \not \in 2\mathbb{Z},
 \end{cases}
\] 
where
\[ 
L_{\mathbf{a}}(s,k):=c_{S,k}((s+n-k)/2+l/4)\mathcal{G}_{k-l/2,2n}((s+n)/2)
\]
has a meromorphic continuation to the whole complex plane. More precisely, the poles are exactly the poles of the Eisenstein series $\mathcal{E}((s+n+l/2)/2)$ as described in Theorem \ref{analytic properties of Eisenstein series} plus the poles of the function $G(\chi,s+n).$
\end{thm}

\begin{proof} The theorem follows now from equation \eqref{Main identity 2} and Theorem \ref{analytic properties of Eisenstein series} arguing similarly to the proof of \cite[Theorem 6.1]{Sh95}. Assume first that $\f{f}_{\phi} | \f{e}$, which is equivalent to $\phi_v (\f{o}_v^{\times}) = 1$ (i.e. $\phi_v$ is unramified) for all $v$ that do not divide $\f{e}$ and that $\f{f}_{\phi} | \f{c}$. Then we can use the equation \eqref{Main identity 2} with $\chi := \phi \psi$. We obtain the statement of the theorem by observing that the function $L_{\psi,\f{x}}(s,\mathbf{f},\phi)$ may be obtained by changing $\f{e}$ to $\f{e} \prod_{v | \f{x}} \f{c}_v$ and employing Lemma \ref{lem:mixed_cosets}. This guarantees that the equation \eqref{Main identity 2} is not trivial (0=0) and hence we can read off the analytic properties of $L_{\psi,\f{x}}(s,\mathbf{f},\phi)$ from those of $\mathcal{E}$. \newline
We also give the proof of the general case by repeating the idea which was used to show \cite[Theorem 6.1]{Sh95}. Set $\f{c}^0 := \f{c} \cap \f{f}_{\phi}$ and decompose $\f{c}^0 = \f{e}^0 \f{e}^1$ with $(\f{e}^0 ,\f{e}^1)=1$, such that $\f{e}_v^0 = \f{c}_v^0 $ for every $v | \f{e}\f{x}\f{f}_{\phi}$, and $\f{e}_v^0 =  \f{o}_v$ otherwise. Then if $\b{D}^0$ denotes the group $\b{D}$ with $\f{c}^0, \f{e}^0$ in place of $\f{c}$ and $\f{e}$, $\mathbf{f} \in \mathcal{S}_{k,S}^n(\b{D}^0, \psi) = \mathcal{S}_{k,S}^n(\b{D}^0, \chi)$. In particular, we can apply the argument of the previous paragraph with $\chi := \phi \psi$ and the group $\b{D}^0$ to conclude the proof. 
\end{proof}

\begin{rem} The proof above indicates the significance of considering in the whole paper the case of a non-trivial ideal $\f{e}$. Indeed, let us consider a cusp form $\mathbf{f} \in S_{k,S}^n (\b{D}[\f{b}^{-1}, \f{bc}],\psi)$, that is with $\f{e} = \f{o}$, and assume for simplicity that $\f{x}$ is trivial. Moreover, consider a Hecke character $\phi$ whose conductor $\f{f}_{\phi}$ - again, for simplicity - is prime to $\f{c}$. Then $\f{c}^0 = \f{c}\f{f}_{\phi}$ and $\f{e}^0 = \f{f}_{\phi}$, and thus we need to consider non-trivial $\f{e}$ even if we start with a form of trivial one.
\end{rem}

Now we can also prove a theorem regarding the analytic continuation of the Klingen-type Jacobi Eisenstein series attached to a form $f$ in the case of $\f{e}= \f{c}$. 

\begin{thm} \label{Main Theorem 2}Let $f \in S_{k,S}^n(\b{\Gamma})$ be a Hecke eigenform with $\b{\Gamma} = \b{D} \cap \b{G}$ where we take $\f{e}=\f{c}$ (i.e., in particular $\psi = 1$) and let $\chi$ be a Hecke character of $F$ such that $\chi_{\b{a}}(x) = \sgn_{\b{a}}(x)^k$. Then the Klingen-type Eisenstein series 
\[
\mathcal{E}(z,s;f, \chi)  := c_{S,k}(s-k/2)\mathcal{G}_{k-l/2,n+m}(s-l/4)  \b{\Lambda}(s,f,\chi) E(z,s;f,\chi), 
\]
where 
\[
\b{\Lambda}(s,f,\chi)\! :=\! L(2s-n-l/2,\mathbf{f},\chi)\!  
\begin{cases} L_{\mathfrak{c}}(2s-l/2,\chi \psi_S) \prod_{i=n+1}^{[(n+m)/2]}\! L_{\mathfrak{c}}(4s-l-2i,\chi^{2}), & l \in 2\Z,\\
 \prod_{i=n+1}^{[(n+m+1)/2]} L_{\mathfrak{c}}(4s-l-2i+1,\chi^{2}), & l \notin 2\Z,
 \end{cases}
\]
has a meromorphic continuation to the entire complex plane.
\end{thm}

\begin{proof}
We need to rewrite the equation \eqref{Main identity 1}. First note that since $\f{e}=\f{c}$, we have $U_{h1_n,\chi} = 1$. Now we extend an argument in \cite[page 569]{Sh96} to the Siegel-Jacobi case.\\
Observe that for every finite place $v$ we have 
$\mathbf{Y}_v = \b{\eta}_n \b{D}_vR_v(\f{c})\b{D}_v\b{\eta}_n^{-1}$.  Further, consider the isomorphism 
\[
S_{k,S}^n(\b{D}) \cong S_{k,S}^n(\widetilde{\b{D}}),\,\,\,\ \mathbf{f} \mapsto \mathbf{f} |_{k,S} \b{\eta}_n,
\]
where $\widetilde{\b{D}}:= C[\f{b}^{-1}, \f{o}, \f{b}^{-1}]D[\f{b}\f{c},\f{b}^{-1}\f{c}]$. Note that since $\f{e}=\f{c}$ we do not have any nebentype (i.e. $\psi=1$). Now note that for any $g \in R(\f{c})$ 
\[
\b{\eta}_n \widetilde{\b{D}} g \widetilde{\b{D}} \b{\eta}_n^{-1}= \b{D} g \b{D}, 
\]
and hence we can conclude that $(f | T_{g})|_{k,S}{\b{\eta}_n} = (f |_{k,S}\b{\eta}_n) | \widetilde{T_{g}}$, where $\widetilde{T_g}$ denotes the Hecke operator defined with respect to the group $\widetilde{\b{D}}$. Putting all these observations together we see that the equation \eqref{Main identity 1} can be also written as
\begin{align}\nonumber
G&(\chi,2s-n-l/2) N(\f{b}\f{e}^{-1}\f{c})^{2ns}\chi_{\h}(\theta)^{-n}(-1)^{n(s-k/2)} vol(A)\\
&\cdot <(\mathcal{E}|_{k,S} \b{\rho}) (\diag[z_1,z_2],s),(f_{k,S}|\b{\eta}_n)^c(z_2)>\nonumber\\
&\hspace{3cm}=\nu_{\f{e}}c_{S,k}(s-k/2)\mathcal{G}_{k-l/2,n+m}(s-l/4)  \b{\Lambda}(s,f,\chi) E(z_1,s;f),
\end{align}
where, recall, $G(\chi,2s-n-l/2)$ is meromorphic on $\mathbb{C}$. In particular, we can extend the Klingen-type Eisenstein series to the whole of $\mathbb{C}$ with respect to variable $s$ by using the analytic properties of the Siegel-type Eisenstein series. Moreover, we can read off the various poles from this expression.
\end{proof}

\section{Algebraicity of special $L$-values}

In the previous section we proved results on the analytic continuation of the standard $L$-function attached to a Siegel-Jacobi eigenfunction $\mathbf{f}$. Assuming that one can define a sensible algebraic structure on the space of Siegel-Jacobi modular forms, it is natural to ask whether a ``Deligne's Conjecture''-style result may hold for some values of the standard $L$-function, which are often called special $L$-values. 

As we indicated in the introduction, this is indeed the case for Siegel modular forms, as shown for example in \cite{Sturm,Sh00}. Indeed, by using the theory of canonical models for the Siegel modular varieties (as it is explained in \cite[Chapter 2]{Sh00}), one can define an algebraic structure on the space of Siegel modular forms, and for an algebraic eigenfunction establish 
algebraicity results for the special $L$-values of the attached standard $L$-function (see for example Theorem 28.8 in \cite{Sh00}). Furthermore, one can, conjecturally, attach a motive to such a Siegel modular form, such that the associated motivic $L$-function can be identified with the standard $L$-function (see for example \cite{Yoshida}). Then the special values of the standard $L$-function can be identified with the critical values of the motivic $L$-function and then the algebraicity results can be seen in the light of Deligne's Period conjectures \cite{Deligne} (up to the difficult issue of comparing motivic and automorphic period).     

The main aim of this section is to establish results indicating that the picture described above holds also for Siegel-Jacobi forms. That is, we will establish results towards the algebraicity of special $L$-values of Siegel-Jacobi modular forms. The starting point of our investigation is the paper of Shimura \cite{Sh78}, where the arithmetic nature of Siegel-Jacobi modular forms is studied. We should remark right away that the paper of Shimura is written for $F= \mathbb{Q}$, but it is not very hard to see that almost everything there can be generalized to the situation of any totally real field $F$. Indeed, in what follows, whenever we state a result from that paper, we will always comment on what is needed to extend it to the case of a totally real field. 

In this section we change our convention: we will write $\b{f}$ (instead of $f$) for Siegel-Jacobi modular forms, $\mathbf{f}$ will still denote the corresponding adelic form, and $f$ will be used for other types of forms.

\subsection{Arithmetic properties of Siegel-Jacobi modular forms}
For a congruence subgroup $\b{\Gamma}$ of $\b{G}(F)$ and a subfield $K$ of $\mathbb{C}$ we define the set
\[
M_{k,S}^{n}(\b{\Gamma},\psi, K) := \{ \b{f} \in M_{k,S}^{n}(\b{\Gamma},\psi) : \b{f}(\tau,w) = \sum_{t,r} c(t,r) \mathbf{e}_{\a}(\tr(t\tau +\T{r}w)), \,\,c(t,r) \in K \} ;
\]
the subspace $S_{k,S}^{n}(\b{\Gamma},\psi, K)$ consisting of cusp forms is defined in a similar way. Moreover, we write $M_{k,S}^{n}(K)$ for the union of all spaces $M_{k,S}^{n}(\b{\Gamma}_1(\f{b},\f{e}),K)$ for all integral ideals $\f{e}$ and fractional ideals $\mathfrak{b}$, where $\b{\Gamma}_1(\f{b},\f{e}):=\b{G}^n(F)\cap \b{D}_1(\mathfrak{b},\f{e})$, and
\[
 \b{D}_1(\f{b},\f{e}):=\{ (\l,\mu,\kappa)x\in C[\f{o}, \f{b}^{-1}, \f{b}^{-1}]D[\f{b}^{-1}\f{e},\f{be}]:(a_x-1_n)_v\in M_n(\f{e}_v)\mbox{ for every } v|\f{e}\}.
\]

For an element $\sigma \in Aut(\mathbb{C})$ and an element $k = (k_v) \in \mathbb{Z}^{\mathbf{a}}$ we define $k^{\sigma} := (k_{v\sigma}) \in \mathbb{Z}^{\mathbf{a}}$, where $v \sigma$ is the archimedean place corresponding to the embedding $ K \stackrel{\tau_v}{\hookrightarrow} \mathbb{C} \stackrel{\sigma}{\rightarrow} \mathbb{C}$, if $\tau_v$ is the embedding in $\mathbb{C}$ corresponding to the archimedean place $v$.  
\begin{prop} Let $k \in \mathbb{Z}^{\mathbf{a}}$, and let $\Phi$ be the Galois closure of $F$ in $\overline{\mathbb{Q}}$, and $\Phi_k$ the subfield of $\Phi$ such that
\[
Gal(\Phi/ \Phi_k) := \left\{ \sigma \in Gal(\Phi /F) : \,\,\,  k^{\sigma} = k \right\}.
\]
Then $M_{k,S}^{n}(\mathbb{C}) = M_{k,S}^{n}(\Phi_k) \otimes_{\Phi_k} \mathbb{C}.  $ 
\end{prop}
\begin{proof} If $F= \mathbb{Q}$, this is \cite[Proposition 3.8]{Sh78}. A careful examination of the proof \cite[page 60]{Sh78} shows that the proof is eventually reduced to the corresponding statement for Siegel modular forms of integral (if $l$ is even) or half-integral (if $l$ is odd) weight. However, in both cases the needed statement does generalize to the case of totally real fields, as it was established in \cite[Theorems 10.4 and 10.7]{Sh00}. 
\end{proof}

Given an $\b{f} \in M_{k,S}^{n}(\mathbb{C})$, we define
 \[
 \b{f}_*(\tau,w) := \mathbf{e}_{\a}(S w(\tau- \overline{\tau})^{-1}\transpose{w}) \b{f}(\tau,w)
 \]
and write $\mathbb{Q}^{ab}$ for the maximal abelian extension of $\mathbb{Q}$. Moreover, for $k \in \frac{1}{2} \mathbb{Z}^{\mathbf{a}}$ such that $k_v - \frac{1}{2} \in \mathbb{Z}$ for all $v \in \mathbf{a}$ we write $M_k^n$ for the space of Siegel modular forms of weight $k$, and of any congruence subgroup, and $M_k^n(K)$ for those with the property that all their Fourier coefficients at infinity  lie in $K$ (see for example \cite[Chapter 2]{Sh00} for a detailed study of these sets).

\begin{prop} \label{Hecke Operators preserve field of definition}Let $K$ be a field that contains $\Q^{ab}$ and $\Phi$ as above. Then  
\begin{enumerate}
\item $\b{f} \in M_{k,S}^{n}(K)$ if and only if $\b{f}_{*}(\tau, v \Omega_{\tau} ) \in M_{k}^{n}(K)$, where $\Omega_{\tau} := \transpose{(\tau \,\,\,1_n)}$, and $v \in M_{l,2n}(F)$.
\item For any element $\gamma \in \Sp_n(F) \hookrightarrow \b{G}^n(F)$ and $\b{f} \in M_{k,S}^{n}(K)$,  we have 
\[
\b{f}|_{k,S} \gamma \in M_{k,S}^{n}(K).
\]
Moreover, if $K$ contains the values of the character $\psi$, then if $\b{f} \in M_{k,S}^n(\b{\Gamma},\psi,K)$, it follows that $\b{f}|T_{r,\psi} \in M_{k,S}^n(\b{\Gamma},\psi,K)$ for any $r \in Q(\f{e})$.
\end{enumerate}
\end{prop}

\begin{proof} If $F=\mathbb{Q}$, this is \cite[Proposition 3.2]{Sh78}. It is easy to see that the proof generalizes to the case of any totally real field. Indeed, the first part of the proof is a direct generalization of the argument used by Shimura. The second part requires the fact that the space $M_{k}^n(K)$ is stable under the action of elements in $\Sp_n(F)$, which is true for any totally real field, as it is proved in \cite[Theorem 10.7 (6)]{Sh00}. The last statement follows from the definition of the Hecke operator $T_{r,\psi}$.
\end{proof}

For a symmetric matrix $S\in Sym_l(F)$, $h \in M_{l,n}(F)$ and a lattice $L \subset M_{l,n}(F)$ we define the Jacobi theta series of characteristic $h$ by
\[
\Theta_{S,L,h}(\tau,w) = \sum_{x \in L} \mathbf{e_a}(\tr(S(\frac{1}{2} \transpose{(x + h)}\tau (x+h) + (x+h) w  ) )).
\]

\begin{thm}  \label{Structure of holomorphic Siegel-Jacobi}Assume that $n > 1$ or $F \neq \mathbb{Q}$, and let $K$ be any subfield of $\mathbb{C}$. Let $A \in \GL_l(F)$ be such that $A S \,\transpose{A} = \diag[s_1, \ldots, s_l ]$, and define the lattices $\Lambda_1 := A M_{l,n}(\mathfrak{o}) \subset M_{l,n}(F)$ and $\Lambda_2 := 2 \diag[s_1^{-1}, \ldots, s_l^{-1}] M_{l,n}(\mathfrak{o}) \subset M_{l,n}(F)$. Then there is an isomorphism
\[
\Phi: M^n_{k,S}(K) \cong \bigoplus_{h \in \Lambda_1/\Lambda_2} M^n_{k-l/2}(K)
\] 
given by $\b{f} \mapsto \left( f_{h}\right)_h$, where the $f_h \in M^n_{k-l/2}(K)$ are defined by the expression
\[
\b{f}(\tau,w) = \sum_{h \in \Lambda_1/\Lambda_2} f_h(\tau) \Theta_{2S,\Lambda_2,h}(\tau,w).
\]

Moreover, under the above isomorphism,
\[
   \Phi^{-1} \left( \bigoplus_{h \in \Lambda_1/\Lambda_2} S^n_{k-l/2}(K) \right) \subset S^n_{k,S}(K).
\]
\end{thm}

\begin{rem} \label{Condition $n=1$}We remark here that the assumption of $n>1$ or $F \neq \mathbb{Q}$ is needed to guarantee that the $f_h$'s are holomorphic at the cusps, which follows from the K\"{ocher} principle. However, even in the case of $F=\mathbb{Q}$ and $n=1$, if we take $\b{f}$ to be of trivial level, then the $f_h$'s are holomorphic at infinity (see for example \cite[page 59]{EZ}).
\end{rem}

\begin{proof}[Proof of Theorem \ref{Structure of holomorphic Siegel-Jacobi}] The first statement is \cite[Proposition 3.5]{Sh78} for $F=\mathbb{Q}$ and it easily generalizes to the case of any totally real field. We explain the statement about  cusp forms. 

Consider first expansions around the cusp at infinity. Fix $h\in \Lambda_1/\Lambda_2$ and let $f_h(\tau) = \sum_{t_2 >0} c(t_2) \mathbf{e}_{\a}(\tr(t_2\tau))$. It is known that Fourier coefficients $c(t_1,r)$ of a Jacobi theta series  
\[
\Theta_{2S,\Lambda_2,h}(\tau,w) = \sum_{t_1,r} c(t_1,r) \mathbf{e}_{\a}(\tr(t_1\tau)) \mathbf{e}_{\a}(\tr(\T{r}w))
\]
are nonzero only if $4t_1 = r S^{-1} \transpose{r}$ (see \cite[p. 210]{Z89}). Hence, the coefficients of
\[
f_h(\tau) \Theta_{2S,\Lambda_2,h}(\tau,w) = \sum_{t,r}\left( \sum_{t_1 + t_2 = t} c(t_1,r) c(t_2)\right) \mathbf{e}_{\a}(\tr(t\tau)) \mathbf{e}_{\a}(\tr(\T{r}w))
\] 
are nonzero only if $4t = 4(t_1+t_2) = rS^{-1} \transpose{r} + 4t_2 > rS^{-1}\transpose{r}$. This means that the function $f_h(\tau) \Theta_{2S,\Lambda_2,h}(\tau,w)$ satisfies cuspidality condition at infinity. 

Now let $\gamma$ be any element in $\Sp_n(F)$. The first statement in the Theorem states that for every $h_1 \in \Lambda_1/\Lambda_2$ there exist $f_{h_1,h_2}\in M_{k-l/2}^n(K),h_2\in\Lambda_1/\Lambda_2$, such that
\[
\Theta_{2S,\Lambda_2,h_1}|_{k,S}\gamma (\tau,w)= \sum_{h_2} f_{h_1,h_2}(\tau) \Theta_{2S, \Lambda_2,h_2} (\tau,w).
\]
Hence, for some cusp forms $f_{h_1}\in S_{k-l/2}^n(K)$,
\begin{align*}
\b{f}|_{k,S}\gamma (\tau,w) &:=  \sum_{h_1 } f_{h_1}|_{k}\gamma(\tau) \left(\sum_{h_2} f_{h_1,h_2}(\tau)\Theta_{2S, \Lambda_2,h_2} (\tau,w)\right)\\
&=\sum_{h_2} \left(\sum_{h_1} f_{h_1}|_{k}\gamma(\tau) f_{h_1,h_2}(\tau) \right) \Theta_{2S, \Lambda_2,h_2} (\tau,w).
\end{align*}
The same argument as used for the cusp at infinity implies that the functions $\b{f}|_{k,S}\gamma (\tau,w)$ and $\sum_{h_1} f_{h_1}|_{k}\gamma (\tau)f_{h_1,h_2}(\tau)$ are cuspidal. This finishes the proof.
\end{proof}

Note that the above theorem does not state that  $\Phi^{-1}\left( \bigoplus_{h \in \Lambda_1/\Lambda_2} S^n_{k-l/2}(K) \right) = S^n_{k,S}(K)$.
For this reason we make the following definition.

\noindent \textbf{Property A.} We say that a cusp form $\b{f} \in S_{k,S}^n(K)$ has the Property A if 
\[
\Phi(\b{f}) \in \bigoplus_{h \in \Lambda_1/\Lambda_2} S^n_{k-l/2}(K).
\]
 \textbf{Examples of Siegel-Jacobi forms that satisfy the Property A:} 
\begin{enumerate}
\item  Siegel-Jacobi forms over a field $F$ of class number one, and with trivial level, i.e. with $\mathfrak{c} = \mathfrak{o}$. Note that in this situation there is only one cusp. Then, keeping the notation as in the proof of the theorem above we need to verify that if $\b{f}(\tau,w) = \sum_{t,r} c_{\mathbf{f}}(t,r) \mathbf{e}_{\a}(\tr(t\tau)) \mathbf{e}_{\a}(\tr(\T{r}w))$ with $4t > rS^{-1}\transpose{r}$ whenever $c(t,r)\neq 0$, then the $f_{h}$ have to be cuspidal. Observe first that if $h_1,h_2 \in \Lambda_1/\Lambda_2$ are different, $\Theta_{2S,\Lambda_2,h_1}(\tau,w) = \sum_{t,r} c_1(t,r) \mathbf{e}_{\a}(\tr(t\tau)) \mathbf{e}_{\a}(\tr(\T{r}w)),$ and $\Theta_{2S,\Lambda_2,h_2}(\tau,w) = \sum_{t,r} c_2(t,r) \mathbf{e}_{\a}(\tr(t\tau)) \mathbf{e}_{\a}(\tr(\T{r}w))$, then there is no $r$ such that at the same time $c_1(t,r) \neq 0$ and $c_2(t,r) \neq 0$. Indeed, if it was not the case then there would be $\lambda_1,\lambda_2 \in \Lambda_2$ such that $\transpose{r} = 2S (\lambda_1 + h_1)$ and $\transpose{r} = 2 S (\lambda_2 + h_2)$, that is, $\lambda_1 + h_1 = \lambda_2 + h_2$ or, equivalently, $h_1 - h_2 \in \Lambda_2$; contradiction. Hence, for any given $r$ there is a unique $h \in \Lambda_1/\Lambda_2$ such that $\Theta_{2S,\Lambda_2,h}$ has a nonzero coefficient $c(t,r)$. This means that there exists a unique $h$ such that $c_{\mathbf{f}}(t,r)$ is the Fourier coefficient of $f_h(\tau) \Theta_{2S,\Lambda_2,h}(\tau,w) = \sum_{t,r} \sum_{t_1 + t_2 = t} c(t_1,r) c(t_2) \mathbf{e}_{\a}(\tr(t\tau)) \mathbf{e}_{\a}(\tr(\T{r}w))$. But then $rS^{-1} \transpose{r} < 4t = 4(t_1+t_2) = rS^{-1} \transpose{r} + 4t_2$ and so $t_2 > 0$, which proves that $f_h$ is cuspidal. 

\item  Siegel-Jacobi forms of index $S$ such that $\det(2S) \in \mathfrak{o}^{\times}$, as in this case the lattices $\Lambda_1$ and $\Lambda_2$ from Theorem \ref{Structure of holomorphic Siegel-Jacobi} are equal.
\item Siegel-Jacobi forms of non-parallel weight, that is, if there exist distinct $v,v' \in \mathbf{a}$ such that $k_v \neq k_{v'}$. Indeed, in this case $M^n_{k-l/2}(K) = S^n_{k-l/2}(K)$ for all $h \in \Lambda_1/\Lambda_2$ (see \cite[Proposition 10.6]{Sh96}).
\end{enumerate}

Let us now explain the significance of the Property A. Recall first that we have defined a Petersson inner product $<\b{f},\b{g}>$ when $\b{f},\b{g} \in M^n_{k,S}(K)$ and one of them, say, $\b{f}$ is cuspidal. If $\b{f}$ satisfies the Property A, then we claim that
\[
<\b{f},\b{g}> = N(\det(4S))^{-n/2}\sum_{h \in \Lambda_1/\Lambda_2} <f_h, g_h>.
\]
Indeed, as in \cite[Lemma 3.4]{Z89},  
\[
<\b{f},\b{g}> = N(\det(4S))^{-n/2}  vol(A)^{-1}\int_{A} \sum_{h \in \Lambda_1/\Lambda_2} f_h(\tau) \overline{g_h(\tau)} \det(\Im(\tau))^{k-l/2 -(n+1)}d\tau ,  
\]
where $A = \Gamma \back \mathbb{H}^{\mathbf{a}}_n$ and a congruence subgroup $\Gamma$ is deep enough. We obtain the claimed equality after exchanging the order of integration and summation. This can be done exactly because each $f_h$ is cuspidal, which makes each individual integral well defined.

\begin{lem} \label{projection to cuspidal part} Assume that $n>1$ or $F \neq \mathbb{Q}$ and that $\b{f} \in S_{k,S}^{n}(\overline{\mathbb{Q}})$ satisfies the Property A and one of the following two conditions hold:
\begin{enumerate}
\item[(i)] there exist $v,v' \in \mathbf{a}$ such that $k_v \neq k_{v'}$;
\item[(ii)] $k = \mu \mathbf{a} = (\mu,\ldots,\mu) \in \mathbb{Z}^{\mathbf{a}}$, with $\mu \in \mathbb{Z}$ depending on $n$ and $F$ in the following way:\\
\begin{tabular}{cccccccc}
$n >2$ & & $n =2, F=\mathbb{Q}$ & & $n =2, F\neq\mathbb{Q}$ & & $n=1$ & .\\
$\mu > 3n/2 +l/2$ & & $\mu > 3$ & & $\mu > 2$ & & $\mu \geq 1/2$ &
\end{tabular}
\end{enumerate} 
Then for any $\b{g}\in M_{k,S}^{n}(\overline{\mathbb{Q}})$ there exists $\widetilde{\b{g}} := \mathfrak{q}(\b{g}) \in S_{k,S}^{n}(\overline{\mathbb{Q}})$ such that 
\[
<\b{f},\b{g}> = <\b{f},\widetilde{\b{g}}>.
\]
\end{lem}

\begin{proof}  There is nothing to show in the case of non-parallel weight, since as it was mentioned above there is no (holomorphic) Eisenstein part in this case. In the parallel weight case, since $\b{f}$ has the Property A, $<\b{f},\b{g}> = N(\det(4S))^{-n/2} \sum_{h \in \Lambda_1/\Lambda_2} <f_h, g_h>$. Let $\widetilde{\mathfrak{q}} : M_{k-l/2}^n(\overline{\mathbb{Q}}) \rightarrow S_{k-l/2}^n(\overline{\mathbb{Q}})$ be the projection operator defined in \cite[Theorem 27.14]{Sh00}. Then, if we put $\widetilde{g}_h := \widetilde{\mathfrak{q}}(g_h)$ for all $h \in \Lambda_1 /\Lambda_2$, it follows that 
\[
<\b{f},\b{g}> = N(\det(4S))^{-n/2} \sum_{h \in \Lambda_1/\Lambda_2} <f_h, g_h> = N(\det(4S))^{-n/2} \sum_{h \in \Lambda_1/\Lambda_2} <f_h, \widetilde{g}_h> .
\]
In particular, if we set $\widetilde{\b{g}} := \Phi^{-1}((\widetilde{g}_h)_h)$, we obtain the statement of the lemma.
\end{proof}

We consider now a non-zero $\b{f} \in S_{k,S}^{n}(\b{\Gamma},\overline{\mathbb{Q}})$ with $\b{\Gamma} := \b{G} \cap \b{D}$, where
$$\b{D}:=\{ (\l,\mu,\kappa)x\in C[\f{o}, \f{b}^{-1}, \f{b}^{-1}]D[\f{b}^{-1}\f{c}_{\b{f}},\f{bc}_{\b{f}}]:(a_x-1_n)_v\in M_n((\f{c}_{\b{f}})_v)\mbox{ for every } v|\f{c}_{\b{f}}\}. $$ 
We assume that $\b{f}$ is an eigenfunction of the operators $T(\f{a})$ for all integral ideals $\f{a}$, write $\b{f} | T(\f{a}) = \lambda(\f{a}) \b{f}$ and define the space
\[
V(\b{f}) := \{ \widetilde{\b{f}} \in S_{k,S}^{n}(\b{\Gamma},\overline{\mathbb{Q}}) :\widetilde{\b{f}} | T(\f{a}) = \lambda(\f{a}) \widetilde{\b{f}}\mbox{ for all }\f{a}\} .
\]
For simplicity, from now on we will only consider the case of $\f{c}_{\b{f}}=\f{e}_{\b{f}}$, but our arguments can be easily generalized to the more general case of $\f{e}_{\b{f}}\neq \f{c}_{\b{f}}$. We are now ready to state the main theorem of this section on algebraic properties of 
\[
\b{\Lambda}(s,\b{f},\chi) = L(2s-n-l/2,\b{f},\chi) \begin{cases} L_{\mathfrak{c}}(2s-l/2,\chi \psi_S)  & \hbox{if } l \in 2\mathbb{Z},\\
1 & \hbox{if } l \not \in 2\mathbb{Z}.
 \end{cases} 
\]

\begin{thm}\label{Main Theorem on algebraicity}
Assume $n >1$ or $F\neq \mathbb{Q}$. Let $\chi$ be a Hecke character of $F$ such that $\chi_{\b{a}}(x) =\sgn_{\b{a}}(x)^k$, and $0 \neq \b{f} \in S_{k,S}^{n}(\b{\Gamma},\overline{\mathbb{Q}})$ an eigenfunction of all $T(\f{a})$. Set $\mu := \min_v{k_v}$ and assume that 
\begin{enumerate}
\item $\mu > 2n +l +1$,
\item  Property A holds for all $\widetilde{\b{f}} \in V(\b{f})$,
\item $k_v \equiv k_{v'} \mod{2}$ for all $v, v' \in \mathbf{a}$.
\end{enumerate}
Let $\sigma \in  \mathbb{Z}$ be such that 
\begin{enumerate}
\item $2n+1 - (k_v - l/2)  \leq \sigma -l/2 \leq k_v -l/2$ for all $v \in \mathbf{a}$, 
\item $| \sigma - \frac{l}{2} - \frac{2n+1}{2} | + \frac{2n+1}{2} - (k_v - l/2) \in 2 \mathbb{Z}$ for all $v \in \mathbf{a}$,
\item $k_v > l/2 + n(1+k_v - l/2 -|\sigma - l/2 - (2n+1)/2| - (2n+1)/2)$ for all $v \in \mathbf{a}$,
\end{enumerate}
but exclude the cases
\begin{enumerate}
\item $\sigma = n+1+ l/2$, $F = \mathbb{Q}$ and $\chi^2 \psi_i^2 = 1$ for some $\psi_i$,
\item $\sigma = l/2$, $\mathfrak{c} = \mathfrak{o}$ and $\chi \psi_S \psi_i = 1$ for some $\psi_i$,
\item $0 < \sigma - l/2 \leq n$, $\mathfrak{c} = \mathfrak{o}$ and $\chi^2 \psi_i^2 = 1$ for some $\psi_i$.
\item $\sigma \leq l +n$ in case $F$ has class number larger than one.
\end{enumerate}

Under these conditions
\[
\frac{\b{\Lambda}(\sigma/2,\b{f},\chi)}{\pi^{e_{\sigma}} <\b{f},\b{f}>} \in \overline{\mathbb{Q}},
\]
where  
$$e_{\sigma} =  n \sum_{v \in \mathbf{a}} (k_v - l + \sigma) - de,\quad e := \begin{cases}  n^2 + n - \sigma + l/2, & \mbox{if } 2 \sigma -l  \in 2 \mathbb{Z} \mbox{ and } \sigma \geq 2n + l/2, \\ n^2, & \mbox{otherwise}. \end{cases} $$
\end{thm}

This theorem will be proved at the end of this section.
We first need to introduce the notion of nearly holomorphic Siegel-Jacobi modular forms $N_{k,S}^{n,r}(\Gamma)$ for $r \in \mathbb{Z}^{\mathbf{a}}$.
\subsection{Nearly holomorphic Siegel-Jacobi modular forms}

\begin{defn} A $C^{\infty}$ function $\b{f} (\tau,w): \mathcal{H}_{n,l} \rightarrow \mathbb{C}$ is said to be a nearly holomorphic Siegel-Jacobi modular form (of weight $k$ and index $S$) for the congruence subgroup $\b{\Gamma}$ if 
\begin{enumerate}
\item $\b{f}$ is holomorphic with respect to the variable $w$ and nearly holomorphic with respect to the variable $\tau$, that is, $\b{f}$ belongs to the space $N^r(\mathbb{H}_n^d)$ for some $r \in \mathbb{N}$ defined in \cite[page 99]{Sh00};
\item $\b{f} |_{k,S} \gamma = \b{f}$ for all $\gamma \in \b{\Gamma}$.
\end{enumerate}
We denote this space by $N_{k,S}^{n,r}(\b{\Gamma})$ and write $N_{k,S}^{n,r} := \bigcup_{\b{\Gamma}} N_{k,S}^{n,r}(\b{\Gamma})$ for the space of all nearly holomorphic Siegel-Jacobi modular forms of weight $k$ and index $S$.
\end{defn}

We note that if $\b{f} \in N_{k,S}^{n,r}$, then $\b{f}_*(\tau,v\,\,\Omega_{\tau}) \in N_{k}^{n,r}$, the space of nearly holomorphic Siegel modular forms, where recall $\Omega_{\tau} := \transpose{(\tau \,\,\,1_n)}$, and $v \in M_{l,2n}(F)$. Below we extend Theorem \ref{Structure of holomorphic Siegel-Jacobi} to the nearly-holomorphic situaton.

\begin{thm}  Assume that $n > 1$ or $F \neq \mathbb{Q}$. Let $A \in \GL_l(F)$ be such that $A S\, \transpose{A} = \diag[s_1, \ldots, s_l ]$, and define the lattices $\Lambda_1 := A M_{l,n}(\mathfrak{o}) \subset M_{l,n}(F)$ and $\Lambda_2 := 2 \diag[s_1^{-1}, \ldots, s_l^{-1}] M_{l,n}(\mathfrak{o}) \subset M_{l,n}(F)$. Then there is an isomorphism
\[
\Phi: N^{n,r}_{k,S}\cong \bigoplus_{h \in \Lambda_1/\Lambda_2} N^{n,r}_{k-l/2}
\] 
given by $\b{f} \mapsto \left( f_{h}\right)_h$, where the $f_h \in N^{n,r}_{k-l/2}$ are defined by the expression
\[
\b{f}(\tau,w) = \sum_{h \in \Lambda_1/\Lambda_2} f_h(\tau) \Theta_{2S,\Lambda_2,h}(\tau,w).
\]
\end{thm}
\begin{proof} Given an $\b{f} \in N^{n,r}_{k,S}$, the modularity properties with respect to the variable $w$ show that (see for example \cite[proof of Proposition 3.5]{Sh78}) we may write 
\[
\b{f}(\tau,w) = \sum_{h \in \Lambda_1/\Lambda_2} f_h(\tau) \Theta_{2S,\Lambda_2,h}(\tau,w)
\]
for some functions $f_h(\tau)$ with the needed modularity properties. In order to establish that they are actually nearly holomorphic one argues similarly to the holomorphic case. Indeed, a close look at the proof of \cite[Lemma 3.4]{Sh78} shows that the functions $f_h$ have the same properties (real analytic, holomorphic, nearly holomorphic, meromorphic, etc.) with respect to the variable $\tau$ as $\b{f} (\tau,w)$, since everything is reduced to a linear system of the form
\[
\b{f}(\tau,w_i) = \sum_{h \in \Lambda_1/\Lambda_2} f_h(\tau) \Theta_{2S,\Lambda_2,h}(\tau,w_i), \,\,\, i=1,\ldots ,\sharp \Lambda_1/\Lambda_2,
\] 
for some $\{w_i \}$ such that $\det(\Theta_{2S,\Lambda_2,h}(\tau,w_i)) \neq 0$. In particular, after solving the linear system of equations we see that the near holomorphicity of $f_h$ follows from that of $\b{f}$ since the $\Theta_{2S,\Lambda_2,h}(\tau,w_i)$ are holomorphic with respect to the variable $\tau$.  
\end{proof}

The above theorem immediately implies the following.

\begin{cor}\label{cor:N_finite_dim} For a congruence subgroup $\b{\Gamma}$, $N_{k,S}^{n,r}(\b{\Gamma})$ is a finite dimensional $\mathbb{C}$ vector space.
\end{cor}
\begin{proof} The theorem above states that $N_{k,S}^{n,r}(\b{\Gamma}) \cong \bigoplus_h N^{n,r}_{k-l/2}(\Gamma_h)$ for some congruence subgroups $\Gamma_h$, which are known to be finite dimensional (see \cite[Lemma 14.3]{Sh00}). 
\end{proof}

Given an automorphism $\sigma \in Aut(\mathbb{C})$ and $\b{f} \in N_{k,S}^{n,r}$, we define
\[
\b{f}^{\sigma}(\tau,w) := \sum_{h \in \Lambda_1/\Lambda_2} f^{\sigma}_h(\tau) \Theta_{2S,\Lambda_2,h}(\tau,w),
\]
where $f_h \in N^{n,r}_{k-l/2}$, and $f_h^\sigma$ is defined as in \cite[page 117]{Sh00}.
Also, for a subfield $K$ of $\mathbb{C}$, define the space $N^{n,r}_{k,S}(K)$ to be the subspace of $N^{n,r}_{k,S}$ such that $\Phi(N^{n,r}_{k,S}(K))= \bigoplus_{h \in \Lambda_1/\Lambda_2} N^n_{k-l/2}(K)$. In particular, $\b{f} \in N_{k,S}^{n,r}$ belongs to $N^{n,r}_{k,S}(K)$ if and only if $\b{f}^\sigma = \b{f}$ for all $\sigma \in Aut(\C /K)$. Moreover, if $K$ contains the Galois closure of $F$ in $\overline{\Q}$ and $\Q^{ab}$, then $N_{k,S}^{n,r} = N_{k,S}^{n,r}(K) \otimes_K \mathbb{C}$ as the same statement holds for $N_{k-l/2}^{n,r}$. Similarly it follows that if $\b{f} \in N_{k,S}^{n,r}(\overline{\mathbb{Q}})$, then $\b{f}|_{k,S} \b{\gamma} \in N_{k,S}^{n,r}(\overline{\mathbb{Q}})$ for all $\b{\gamma} \in \b{G}(F)$. At this point we also remark that for an $\b{f} \in M_{k,S}^n$ we have that $\b{f}^c$ defined before is nothing else than $\b{f}^{\rho}$ where $1 \neq \rho \in Gal(\mathbb{C}/\mathbb{R})$ i.e. complex conjugation. \newline

We now define a variant of the holomorphic projection in the Siegel-Jacobi case. We define a map $\f{p}\colon N_{k,S}^{n,r}(\overline{\Q}) \rightarrow M_{k,S}^n(\overline{\Q})$ whenever $k_v > n + r_v$ for all $v \in \mathbf{a}$ by 
\[
\f{p}(\b{f}) := \f{p}\left(\sum_{h \in \Lambda_1/\Lambda_2} f_h(\tau) \Theta_{2S,\Lambda_2,h}(\tau,w)\right) := \sum_{h \in \Lambda_1/\Lambda_2} \widetilde{\f{p}}(f_h(\tau)) \Theta_{2S,\Lambda_2,h}(\tau,w),
\]
where $\widetilde{\f{p}}\colon N_{k-l/2}^{n,r}(\overline{\mathbb{Q}}) \rightarrow M_{k-l/2}^n(\overline{\mathbb{Q}})$ is the holomorphic projection operator defined for example in \cite[Chapter III, section 15]{Sh00}. 

\begin{lem} Assume $n >1$ or $F\neq \mathbb{Q}$ and that $\b{f} \in S_{k,S}^n$ satisfies the Property A, and $k_v > n + r_v$ for all $v \in \mathbf{a}$. Then for any $\b{g} \in N_{k}^{n,r}(\overline{\mathbb{Q}})$,
\[
<\b{f},\b{g}> = <\b{f},\f{p}(\b{g})>.
\]
\end{lem}
\begin{proof} This follows from the fact that the above property holds for nearly holomorphic Siegel modular forms, and the fact that the Property A allows us to write the Petersson inner product of Siegel-Jacobi forms as a sum of Petersson inner products Siegel modular forms, in a similar way as we did in the proof of Lemma \ref{projection to cuspidal part}.  
\end{proof}

Further, we define the operator 
$$\f{p}^0 :=\begin{cases} \f{p}, & k\mbox{ not parallel},\\
\f{q} \circ \f{p}, & k\,\,\mbox{parallel}. \end{cases}$$
We now state a theorem regarding the nearly holomorphicity of Siegel-type Jacobi Eisenstein series. The notation below follows the one used in section \ref{section of analytic properties of Eisenstein series}, where the analytic properties were investigated. In particular, the characters $\psi_i$ below are characters of the Hilbert class field extension of $F$.

\begin{thm}\label{algebraic properties of Eisenstein series} Consider the normalized Siegel-type Jacobi-Eisenstein series
\[
D(s):=D(z,s;k,\chi) :=  \Lambda_{k-l/2,\mathfrak{c}}^n(s-l/4,\chi \psi_S) E(z,\chi,s).
\]
Let $\mu \in \mathbb{Z}$ be such that 
\begin{enumerate}
\item $n+1 - (k_v - l/2) \leq \mu-l/2 \leq k_v - l/2$ for all $v \in \mathbf{a}$, and
\item $| \mu -l/2 - \frac{n+1}{2}| + \frac{n+1}{2} - k_v + l/2 \in 2\mathbb{Z}$,
\end{enumerate}
but exclude the cases
\begin{enumerate}
\item $\mu = \frac{n+2}{2}+ l/2$, $F = \mathbb{Q}$ and $\chi^2 \psi_i^2 = 1$ for some $\psi_i$,
\item $\mu = l/2$, $\mathfrak{c} = \mathfrak{o}$ and $\chi \psi_S \psi_i = 1$ for some $\psi_i$,
\item $0 < \mu - l/2 \leq n/2$, $\mathfrak{c} = \mathfrak{o}$ and $\chi^2 \psi_i^2 = 1$ for some $\psi_i$.
\item $\mu \leq l +n$ if $F$ has class number larger than one.
\end{enumerate}
Then 
\[
D(\mu/2) \in \pi^{\beta} N_{k,S}^{n,r}(\overline{\mathbb{Q}}),
\]
where 
$$r = \begin{cases} {n(k-\mu+2)\over 2} & \mbox{if } \mu = {n+2\over 2} +{l\over 2}$, $F=\mathbb{Q}, \chi^2=1,\\
{k\over 2}-{l\over 4} & \mbox{if } n=1, \mu = 2 + {l\over 2}, F= \mathbb{Q}, \chi \psi_S = 1,\\
{n\over 2} (k-{l\over 2} - |\mu - {l\over 2} - {n+1\over 2}|\mathbf{a} - {n+1\over 2} \mathbf{a} ) & \mbox{otherwise} .\end{cases}$$
Moreover,
$\beta = {n\over 2} \sum_{v \in \mathbf{a}} (k_v - l + \mu) - de$,
where 
$$e :=\begin{cases} [{(n+1)^2\over 4}] - \mu + {l\over 2} & \mbox{if } 2\mu -l+n\in 2\mathbb{Z}, \mu \geq n + {l\over 2},\\
[{n^2\over 4}] & \mbox{otherwise} .\end{cases} $$
\end{thm}

\begin{proof} The proof is similar to the proof of Theorem \ref{analytic properties of Eisenstein series}. As in there, we can read off the nearly holomorphicity of the Jacobi Eisenstein series from the classical Siegel Eisenstein series; to be more precise, from the series $E(z,s-l/4;\chi \psi_S \psi_i, k-l/2)$, where $\psi_i$'s vary over all the Hilbert characters. Indeed, the series
\[
\frac{\Lambda_{k-l/2,\mathfrak{c}}^n(\mu/2-l/4,\chi \psi_S)}{\Lambda_{k-l/2,\mathfrak{c}}^n(\mu/2-l/4,\chi \psi_S \psi_i)} \Lambda_{k-l/2,\mathfrak{c}}^n(\mu/2-l/4,\chi \psi_S \psi_i)E(z,s-l/4;\chi \psi_S \psi_i, k-l/2)
\]
 has the same algebraic properties as the normalized series 
$$\Lambda_{k-l/2,\mathfrak{c}}^n(\mu/2-l/4,\chi \psi_S \psi_i)E(z,s-l/4;\chi \psi_S \psi_i, k-l/2),$$ 
if we exclude the cases where the factor $\frac{\Lambda_{k-l/2,\mathfrak{c}}^n(\mu/2-l/4,\chi \psi_S)}{\Lambda_{k-l/2,\mathfrak{c}}^n(\mu/2-l/4,\chi \psi_S \psi_i)}$ has a pole. Therefore all we need to check is that 
 \[
 \frac{\Lambda_{k-l/2}^n(\mu/2-l/4,\chi \psi_S)}{\Lambda_{k-l/2}^n(\mu/2-l/4,\chi \psi_S \psi_i)} \in \overline{\mathbb{Q}}.
 \] 
This should follow from the general Bellinson conjectures for motives associated to finite Hecke characters over totally real fields (see for example \cite{Scholl}). However this is not known in general, and hence we are forced to set the condition $\mu > n+l$ in case $F$ has class number larger than one, in which case we obtain values whose ratio is known to be algebraic, since we are then considering critical values.  
\end{proof}

\begin{lem} \label{diagonal restriction preserves algebraicity} Consider the embedding 
\[
\Delta : \mathcal{H}_{n,l} \times \mathcal{H}_{m,l} \hookrightarrow \mathcal{H}_{N,l},\,\,\,(\tau_1,w_1) \times (\tau_2,w_2) \mapsto (\diag[\tau_1, \tau_2], (w_1 \,\,w_2)),
\]
where $N:=m+n$. Then
\[
\Delta^*\left( N^{N,r}_{k,S}(\overline{\Q}) \right) \subset N^{n,r}_{k,S}(\overline{\Q}) \otimes_{\overline{\Q}}   N^{m,r}_{k,S}(\overline{\Q}). 
\]
\end{lem}
\begin{proof} The proof of this lemma is identical to the Siegel modular form case (see \cite[Lemma 24.11]{Sh00}). Let $\b{f} \in N^{N,r}_{k,S}(\b{\Gamma}^N, \overline{\mathbb{Q}})$ for a sufficiently deep congruence subgroup $\b{\Gamma}^N$. Note that  the function $\b{g}(z_1,z_2) := \Delta^*\b{f}(z)$ is in $N_{k,S}^{n,r}(\b{\Gamma}^n)$ as a function in $z_1$ and in  $N_{k,S}^{m,r}(\b{\Gamma}^m)$ as a function in $z_2$ for appropriate congruence subgroups $\b{\Gamma}^n$ and $\b{\Gamma}^m$. Hence, by Corollary \ref{cor:N_finite_dim} and the fact that $N_{k,S}^{n,r} = N_{k,S}^{n,r}(\overline{\mathbb{Q}}) \otimes_{\overline{\mathbb{Q}}} \mathbb{C}$, for each fixed $z_1$ we may write
\[
\b{g}(z_1,z_2) = \sum_i \b{g}_i(z_1) \b{h}_i(z_2),
\]
where $\b{g}_i(z_1) \in \mathbb{C}$, and $\b{h}_i \in \b{N}_{k,S}^{n,r}(\overline{\mathbb{Q}})$ form a basis of the space. The general argument used in \cite[Lemma 24.11]{Sh00}, which is based on the linear independence of the basis $\b{h}_i$,  shows that the functions $\b{g}_i(z_1)$ have the same properties as the function $\b{g}$ when viewed as a function of the variable $z_1$. Hence, $\b{g}_i \in N^{n,r}_{k,S}$. Now, for any $\sigma \in Aut(\mathbb{C}/\overline{\mathbb{Q}})$, 
\[
\b{g}(z_1,z_2) = \b{g}^{\sigma}(z_1,z_2) = \sum_i \b{g}^{\sigma}_i(z_1) \b{h}^{\sigma}_i(z_2) = \sum_i \b{g}^{\sigma}_i(z_1) \b{h}_i(z_2).
\] 
Hence, $\b{g}^{\sigma}_i(z_2) = \b{g}_i(z_2)$ for all $\sigma \in  Aut(\mathbb{C}/\overline{\mathbb{Q}})$, and thus $\b{g}_i \in N^{n,r}_{k,S}(\overline{\mathbb{Q}})$.
\end{proof}

We can now establish a theorem, which is the key result towards Theorem \ref{Main Theorem on algebraicity}.

\begin{thm} \label{ratio of inner products}
Assume $n >1$ or $F\neq \mathbb{Q}$. Let $0 \neq \b{f} \in S_{k,S}^{n}(\b{\Gamma},\overline{\mathbb{Q}})$ be an eigenfunction of $T(\mathfrak{a})$ for all integral ideas $\mathfrak{a}$ with $(\f{a},\f{c}_{\b{f}})=1$. Define $\mu := \min_{v \in \mathbf{a}}{\{k_v\}}$ and assume that
\begin{enumerate}
\item $\mu > 2n +l +1$,
\item Property A holds for all $\widetilde{\b{f}} \in V(\b{f})$,
\item $k_v \equiv k_{v'} \mod{2}$ for all $v, v' \in \mathbf{a}$.
\item $k_v  > l/2 + n(1+ k_v - \mu)$ for all $v \in \mathbf{a}$. 
\end{enumerate}
Then for any $\b{g} \in M_{k,S}^n(\overline{\mathbb{Q}})$,
\[
\frac{<\b{f},\b{g}>}{ <\b{f}, \b{f}>} \in \overline{\mathbb{Q}}
\]
\end{thm}
\begin{proof} By Lemma \ref{projection to cuspidal part} it suffices to prove this theorem for $\b{g} \in S_{k,S}^n(\overline{\mathbb{Q}})$. 

By the discussion in subsection \ref{Normal Operators} where it was shown that the Hecke operators are normal and Proposition \ref{Hecke Operators preserve field of definition} which states that the Hecke operators $T(\f{a})$ preserve $S_{k,S}^{n}(\b{\Gamma},\overline{\mathbb{Q}})$, we have a decomposition
\[
S_{k,S}^{n}(\b{\Gamma},\overline{\Q}) = V(\b{f}) \oplus \b{U},
\]
where $\b{U}$ is a $\overline{\mathbb{Q}}$-vector space orthogonal to $V(\b{f})$. Therefore, without loss of generality, we may assume that $\b{g} \in V(\b{f})$. 

Now consider a character $\chi$ of conductor $\f{f}_{\chi} \neq \mathfrak{o}$ such that $\chi_{\b{a}}(x) = \sgn_{\b{a}}(x)^{k}$, $\chi^2 \neq 1$ and $G(\chi,\mu-n-l/2) \in \overline{\mathbb{Q}}^{\times}$, where $G(\chi,\mu-n-l/2)$ is as in equation \eqref{definition of $G$}.  The existence of such a character follows from the fact that $G(\chi,2s-n-l/2)$ is the ratio of products of finitely many Euler polynomials.  
We now use a slightly different version of equation \eqref{Main identity 2}, i.e. before multiplying by the factor $\mathcal{G}_{k-l/2,2n}(s-l/4) $, where we take $\f{c} = \f{c}_{\b{f}} \cap \f{f}_{\chi}$, $\f{e}=\f{c}$ and $n=m$, and evaluate it at $s=\mu/2$. Moreover, thanks to Proposition \ref{Behaviour under complex conjugation} if $\tilde{\b{f}} \in V(\b{f})$, then so is $\tilde{\b{f}}^c \in V(\b{f})$ and their $L$-functions agree. In particular, we obtain the following equality up to some non-zero algebraic number:
\begin{multline*}
\Lambda_{k-l/2,\mathfrak{c}}^{2n}(\mu/2-l/4,\chi \psi_S)  vol(A)  <(E|_{k,S} \b{\rho}) (\diag[z_1,z_2],\mu/2;\chi),(\tilde{\b{f}}^c|_{k,S}\b{\eta}_n)^c(z_2)>\\
={^{\overline{\mathbb{Q}}}}^{\times} c_{S,k}(\mu/2-k/2) \b{\Lambda}(\mu/2,\b{f},\chi) \tilde{\b{f}}^c(z_1), 
\end{multline*}
where, recall,
\[
\Lambda^{2n}_{k-l/2,\mathfrak{c}}(\mu/2-l/4,\chi\psi_S) := \begin{cases} L_{\mathfrak{c}}(\mu-l/2,\chi \psi_S) \prod_{i=1}^{n} L_{\mathfrak{c}}(2\mu-l-2i,\chi^{2}) & \hbox{if } l \in 2\mathbb{Z},\\
 \prod_{i=1}^{[(2n+1)/2]} L_{\mathfrak{c}}(2\mu-l-2i+1,\chi^{2}) & \hbox{if } l \notin 2\mathbb{Z},
 \end{cases}
\] 
and
\[
\b{\Lambda}(s,\b{f},\chi) := L(2s-n-l/2,\b{f},\chi) \begin{cases} L_{\mathfrak{c}}(2s-l/2,\chi \psi_S)  & \hbox{if } l \in 2\mathbb{Z},\\
1 & \hbox{if } l \not \in 2\mathbb{Z}.
 \end{cases}
\]
By Theorem \ref{algebraic properties of Eisenstein series}, $\Lambda_{k-l/2,\mathfrak{c}}^{2n}(\mu/2-l/4,\chi \psi_S) E(z,\mu/2; \chi) \in \pi^{\beta} N_{k,S}^{2n}(\overline{\mathbb{Q}})$ for $\beta \in \mathbb{N}$, and hence the same holds for 
\[
\Lambda_{k-l/2,\mathfrak{c}}^{2n}(\mu/2-l/4,\chi \psi_S) E(z,\mu/2; \chi)|_{k,S} \b{\rho}.
\]
In particular,
\[
\pi^{-\beta}\Lambda_{k-l/2,\mathfrak{c}}^{2n}(\mu/2-l/4,\chi \psi_S) (E|_{k,S} \b{\rho}) (\diag[z_1,z_2],\mu/2; \chi) = \sum_i \b{f}_i(z_1) \b{g}_i(z_2),
\] 
where $\b{f}_i,\b{g}_i \in N^{n}_{k,S}(\overline{\mathbb{Q}})$ by Lemma \ref{diagonal restriction preserves algebraicity}. Moreover, $vol(A) = \pi^{d_0} \mathbb{Q}^{\times}$, where $d_0$ is the dimension of $\mathbb{H}_n^d$ since the volume of the Heisenberg part is normalized to one. Furthermore, 
\[
c_{S,k}(\mu/2-k/2) \in \pi^{\delta} \overline{\mathbb{Q}}^{\times},\,\,\,\delta \in \frac{1}{2}\mathbb{Z}.
\]
Altogether we obtain
\[
 \sum_i \b{f}_i(z_1) <\b{g}_i(z_2), \b{g}(z_2) > ={^{\overline{\mathbb{Q}}}}^{\times} \pi^{\delta- d_0 +\beta} \b{\Lambda}(\mu/2,\b{f},\chi) \tilde{\b{f}}^c(z_1),
\]
where $\b{g}:= (\tilde{\b{f}}^c|_{k,S}\b{\eta}_n)^c \in S_{k,S}^n(\overline{\mathbb{Q}})$. Considering the Fourier expansion of $\b{f}_i$'s and $\b{f}$, and comparing any $(r,t)$ coefficients for which $c(r,t;\tilde{\b{f}}^c) \neq 0$, we find that 
\[
<\sum_i \alpha_{i,r,t} \b{g}_i(z_2),\b{g}(z_2)> = {^{\overline{\mathbb{Q}}}}^{\times} \pi^{\delta- d_0 +\beta} \b{\Lambda}(\mu/2,\b{f},\chi) \neq 0
\]
for some $\alpha_{i,r,t} \in \overline{\mathbb{Q}}$, where the non-vanishing follows from Corollary \ref{non-vanishing of L-values}. Setting $\b{h}_{r,t}(z_2) := \sum_i \alpha_{i,r,t} \b{g}_i(z_2) \in N^n_{k,S}(\overline{\mathbb{Q}})$, we obtain 
\[
<\b{h}_{r,t}(z_2), \b{g}(z_2)>={^{\overline{\mathbb{Q}}}}^{\times} \pi^{\delta- d_0 +\beta} \b{\Lambda}(\mu/2,\b{f},\chi) \neq 0,
\]
or,
\[
<\mathfrak{p}^0(\b{h}_{r,t}|_{k,S}\b{\eta}_n)(z_2), \tilde{\b{f}}(z_2)>= {^{\overline{\mathbb{Q}}}}^{\times} \pi^{\delta- d_0 +\beta} \b{\Lambda}(\mu/2,\b{f},\chi) \neq 0,
\]
That is, the forms, or rather their projections to $V(\b{f})$, $\widetilde{\b{h}}_{r,t} := \mathfrak{p}^0(\b{h}_{r,t}|_{k,S}\b{\eta}_n) \in S_{k,S}^n(\overline{\mathbb{Q}})$ for the various $(r,t)$ span the space $V(\b{f})$ over $\overline{\mathbb{Q}}$ and
\[
<\widetilde{\b{h}}_{r,t},\tilde{\b{f}}> \in \pi^{\delta- d_0 +\beta} \b{\Lambda}(\mu/2,\b{f},\chi) \overline{\mathbb{Q}}^{\times}.
\]  
That is, for any $\b{g} \in V(\b{f})$ we have $<\b{g},\b{f}> \in \pi^{\delta- d_0 +\beta} \b{\Lambda}(\mu/2,\b{f},\chi) \overline{\mathbb{Q}}^{\times}$. In particular, the same holds for $\b{g} = \b{f}$, and that concludes the proof.
\end{proof}

\begin{proof}[Proof of Theorem \ref{Main Theorem on algebraicity}] We follow the same steps as in the proof of Theorem \ref{ratio of inner products} but this time we set $s = \sigma/2$. In exactly the same way as above we obtain
\[
<\b{h}_{r,t}(z_2), \b{f}(z_2)>={^{\overline{\mathbb{Q}}}}^{\times} \pi^{\delta- d_0 +\beta} \b{\Lambda}(\sigma/2,\b{f},\chi),
\]
for some $\b{h}_{r,t} \in \b{N}_{k,S}^n(\overline{\mathbb{Q}})$. Thanks to Theorem \ref{ratio of inner products} the proof will be finished after dividing the above equality by $<\b{f},\b{f}>$ if we make the powers of $\pi$ precise. 
Recall that
\begin{multline*}
c_{S,k}(\sigma/2 - k/2) ={^{\overline{\mathbb{Q}}}}^{\times} \pi^{dn(n+1)/2} \prod_{v \in \mathbf{a}} \frac{\Gamma_n(\sigma/2 + k_v -l/2 - (n+1)/2)}{\Gamma_n(\sigma/2 +k_v - l/2)}\\
=\pi^{dn(n+1)/2} \prod_{v \in \mathbf{a}} \frac{\prod_{i=0}^{n-1}\Gamma(\sigma/2 + k_v -l/2 - (n+1)/2 - i/2)}{\prod_{i=0}^{n-1}\Gamma(\sigma/2 +k_v - l/2 - i/2)} ={^{\overline{\mathbb{Q}}}}^{\times} \pi^{dn(n+1)/2}. 
\end{multline*}
Hence, $\delta = dn(n+1)/2$. However, this is also equal to the dimension of the space $\mathbb{H}_n^d$, which we denoted by $d_0$. We are then left only with $\beta$, which is provided by Theorem \ref{algebraic properties of Eisenstein series}; namely,
\[
\beta =  n \sum_{v \in \mathbf{a}} (k_v - l + \sigma) - de, 
\]
where $e := n^2 + n - \sigma + l/2$ if $2 \sigma -l  \in 2 \mathbb{Z}$ and $\sigma \geq 2n + l/2$, and $e:= n^2$ otherwise. This concludes the proof of the theorem.
\end{proof}


\end{document}